\documentclass[11pt,a4paper,dvipsnames]{article}

\author{
    Fethi Bencherki$^{\star}$ \and  
    Anders Rantzer$^{\star}$
    \\[2ex]
}

\date{%
    $^{\star}$Department of Automatic Control, Lund University, Lund, Sweden \\
    \texttt{\{\href{mailto:fethi.bencherki@control.lth.se}{fethi.bencherki}, \href{mailto:anders.rantzer@control.lth.se}{anders.rantzer}\}@control.lth.se} 
}
\usepackage{wrapfig}
\usepackage{float}
\usepackage{tikz}
\usepackage{graphicx}
\usetikzlibrary{chains,shapes.multipart}
\usetikzlibrary{shapes,calc}
\usetikzlibrary{automata,positioning}
\usetikzlibrary{positioning, shapes, arrows, calc}
\tikzstyle{block} = [draw, rectangle, line width=0.7mm]

\usepackage{amsthm}
\newtheorem{problem}{Problem}

\newcommand{\qedblack}{\hfill$\blacksquare$}

 \usepackage[top=27mm, bottom=27mm, right=27mm, left=27mm]{geometry}
\usepackage[T1]{fontenc}
\usepackage[utf8]{inputenc}

\usepackage[english]{babel}      

\usepackage{lmodern}           
\usepackage{microtype}         
\usepackage{mathrsfs}          
\usepackage{amsthm, amsmath, amssymb, mathtools} 
\usepackage{bm, bbm}           
\usepackage{stmaryrd}          
\usepackage{accents}
\numberwithin{equation}{section}

\usepackage{cases,xfrac,gensymb} 
\usepackage{tocloft}             
\usepackage{listings}            
\usepackage{lipsum}              
\usepackage{titling}             
\usepackage{proof}               
\let\oldunderbar\underbar
\usepackage{sectsty}             
\let\underbar\oldunderbar
\usepackage{parskip}             
\usepackage{booktabs}            
\usepackage{enumerate}           
\usepackage{autobreak}           
\usepackage{xcolor}              
\usepackage{csquotes}            
\usepackage{adjustbox}
\usepackage[inline]{enumitem}

\usepackage[bf,sf]{titlesec}       
\usepackage[labelfont={bf,sf}]{caption} 
\usepackage{subcaption}            

\usepackage[
  maxnames=20,      
  giveninits=true,  
  isbn=false,        
]{biblatex}
\AtEveryBibitem{%
  \clearfield{month}%
}

\usepackage{graphicx}            
\usepackage{float}               
\usepackage{tikz}                
\usetikzlibrary{external}

\usepackage{pgfplots}
\pgfplotsset{compat=1.18}
\usepgfplotslibrary{colormaps}
\usepackage{pgfplotstable}
\usepgfplotslibrary{fillbetween}

\let\mathbf\bm
\DeclareMathAlphabet\mathbfcal{OMS}{cmsy}{b}{n} 

\makeatletter
\def\th@plain{
  \thm@headfont{\normalfont\sffamily\bfseries}%
  \itshape
}
\def\th@definition{
  \thm@headfont{\normalfont\sffamily\bfseries}%
  \thm@notefont{\normalfont\sffamily\bfseries}%
}
\theoremstyle{myStyle1}
\newtheorem{theorem}{Theorem}
\newtheorem{lemma}{Lemma}
\newtheorem{proposition}{Proposition}

\newtheorem{assumption}{Assumption}

\theoremstyle{myStyle2}

\newtheorem{remark}{Remark}
\newtheorem{example}{Example}
\providecommand{\keywords}[1]{\textbf{\textsf{Keywords.}} #1}

\definecolor{Red}{rgb}{1, 0, 0}
\definecolor{BlueIMT}{RGB}{0,42,72}
\definecolor{ForestGreen}{rgb}{0.13, 0.55, 0.13}
\definecolor{Gray}{rgb}{0.66, 0.66, 0.66}
\definecolor{Green}{rgb}{0.0, 0.5, 0.0}
\definecolor{MidnightBlue}{rgb}{0.098, 0.098, 0.439}
\definecolor{Orange}{rgb}{0.93, 0.53, 0.18}

\definecolor{color1}{RGB}{68,119,170}   
\definecolor{color2}{RGB}{102,204,238}  
\definecolor{color3}{RGB}{34,136,51}    
\definecolor{color4}{RGB}{17,119,51}    
\definecolor{color5}{RGB}{204,187,68}   
\definecolor{color6}{RGB}{221,204,119}  
\definecolor{color7}{RGB}{204,102,119}  
\definecolor{color8}{RGB}{136,34,85}    
\definecolor{color9}{RGB}{170,51,119}   
\definecolor{color10}{RGB}{102,102,102} 
\definecolor{color11}{RGB}{50,50,50} 

\usepackage{hyperref} 
\hypersetup{%
    hidelinks,
    hypertexnames = true,
    plainpages    = false,
    colorlinks = true,
    urlcolor   = MidnightBlue,
    linkcolor  = MidnightBlue,
    citecolor  = ForestGreen,
}
\usepackage[capitalize,nameinlink]{cleveref}

\newcommand{\overbar}[1]{\mkern 1.5mu\overline{\mkern-1.5mu#1\mkern-1.5mu}\mkern 1.5mu}
\newboolean{showcomments}
\setboolean{showcomments}{true}
\newcommand{\fethi}[1]{\ifthenelse{\boolean{showcomments}}
	{\textcolor{blue}{(Fethi says: #1)}}{}}
\newcommand{\tomas}[1]{\ifthenelse{\boolean{showcomments}}
	{\textcolor{orange}{(Tomas says: #1)}}{}}

\usepackage{comment}
\title{\sffamily\bfseries Minimax~Optimal~Dual~Control~of~Positive~Systems: An~Exact~Solution~for~Scalar~Input-Sign~Uncertainty}

\author{
    Fethi Bencherki$^{\star}$ \and  
    Tomas J. Meijer$^{\star}$  \and
    Anders Rantzer$^{\star}$
    \\[2ex]
}

\date{%
    $^{\star}$Department of Automatic Control, Lund University, Lund, Sweden \\
    \texttt{\{\href{mailto:fethi.bencherki@control.lth.se}{fethi.bencherki},\href{tomas.meijer@control.lth.se}{tomas.meijer},\href{mailto:anders.rantzer@control.lth.se}{anders.rantzer}\}@control.lth.se} 
}

\begin{document}

\maketitle

\begin{abstract}
    While recent advances in minimax dual control have led to exact solutions for uncertain general linear time-invariant systems as well as (sub)optimal dual controllers, corresponding results for linear positive systems are still lacking. This paper aims to fill this gap and thereby pave the way toward scalable dual control algorithms. We study the general minimax optimal dual control problem for positive linear systems with unknown dynamics and reformulate it as a standard zero-sum dynamic game. By allowing randomized control inputs, we solve the corresponding Bellman equation exactly for the scalar case with sign uncertainty in the input. This yields an implicit dual control policy that is optimal both in terms of cost and \(\ell_1\)-gain.
The optimal dual policy uses exploration in a specific region of the hyperstate space to conduct optimal probing. Outside this exploration regime, the controller reduces to a deterministic certainty equivalence policy, indicating that sufficient information has been obtained to identify the correct input direction. In addition, these results allow us to analyze fundamental limitations of minimax dual control for positive systems and provide a foundation for more general dual control problems for positive systems for future work.

\end{abstract}


\section{Introduction}
Modern control applications are increasingly complex, large-scale, and interconnected, see, for instance, the roadmap in~\cite{Annaswamy2023}, which presents many new challenges also from a control perspective. In particular, it is often difficult to derive accurate models for these systems. This has led to an increase in interest in fields such as system identification~\cite{Ljung1999,Overschee1996} and, more recently, data-driven control, see, e.g.,~\cite{Markovsky2008,Coulson2019,vanWaarde2023,Berberich2021,Meijer2025-fd-wfl-arxiv}. Both of these fields, however, are built on the premise that data have been collected a priori, which is not always possible and over time changes in the physical system may render previously collected data less representative of the current system dynamics. Adaptive control presents a promising avenue to deal with these challenges by developing controllers that adapt to a continuous stream of incoming data~\cite{AAstrom2013}. 

The design of such learning-based controllers, however, presents its own unique challenge: The controller needs to simultaneously learn and regulate the plant. The term \emph{dual control} was introduced by Feldbaum~\cite{Feldbaum1960} to capture the need for controllers to manage this exploration--exploitation tradeoff, which has since been studied extensively; see, for example, the survey papers~\cite{Wittenmark1995,Mesbah2018}. Interestingly, the tradeoff between exploration and exploitation plays a crucial role in a wide range of scientific fields, including behavioral ecology~\cite{Macarthur1966}, neuroscience~\cite{Cohen2007}, and computer science~\cite{Recht2019}. Dual control problems are often addressed through dynamic programming~\cite{Bellman57,Bertsekas2007} using a hyperstate combining the system's physical state with an information state, which describes what has been learned so far. Unfortunately, exact solutions to the resulting Bellman equation are difficult to obtain, pose computational challenges, and often need to be approximated, even in cases where information state allows a finite dimensional representation. Nevertheless, several attempts have been made to derive exact solutions for special cases of dual control, see, e.g.,~\cite{Sternby1976,AAstrom1986,Bernhardsson1989,Kulcsar1996}.

More recently, significant advances in dual control have been made by adopting a minimax optimal control perspective. The minimax approach, which can be seen as a natural extension of the $\mathcal{H}_{\infty}$ control problem~\cite{Zames1981,Zhou1996} to an adaptive setting, accounts both worst-case disturbances and uncertain parameters~\cite{Cusumano1988,Ioannou1988,Vinnicombe2004,Megretski2004}. The minimax dual control paradigm was first introduced in~\cite{Didinsky1994} for linear systems, and in~\cite{Pan1998} for nonlinear systems. Using a hyperstate consisting of the current state and a running covariance matrix that efficiently compressing past measurements, the minimax adaptive control problem can be reformulated as a two player zero-sum dynamic game, as was recently shown in~\cite{Rantzer2021a}. The resulting dynamic game can be addressed using standard dynamic programming techniques. It turns out that, when allowing for input randomization to facilitate exploration by the controller, the resulting Bellman equation can be solved exactly, leading to an explicit dual control policy for the case with input uncertainty in the input~\cite{Rantzer2025a}. These results have subsequently helped to develop suboptimal dual controllers also for more general uncertainty, see, e.g.,~\cite{Rantzer2025b,Rantzer2026}. Related work with suboptimal dual controllers has been published in, e.g.,~\cite{Kjellqvist2022a,Kjellqvist2022b,Orlov2018,Rantzer2020}. In parallel, a sizable body of research has focused on performance guarantees in terms of regret-based growth rates for dual controllers, see, e.g.,~\cite{Jedra2022,Dean2018,Tsiamis2023,matni2019self,Lee2024,Simchowitz2020,Ziemann2020,Cassel2020}. Fundamental limitations of minimax dual control were investigated in~\cite{Vinnicombe2004,Megretski2003,Rantzer2025a}. Specifically, Vinnicombe conjectured a lower bound on the best achievable $\ell_2$-gain from disturbances to states for a scalar system with sign uncertainty in the input. This conjecture was later shown to be true in~\cite{Rantzer2025a}.

In this paper, we consider minimax optimal dual control for \emph{positive} systems, i.e., systems in which the state only takes non-negative values. This class of systems is particularly important due to its wide range of applications in, for instance, social sciences, epidemiology, biology, econometrics, and stochastic processes. By exploiting their inherent structure, results for linear systems can often be formulated in terms of \emph{computationally efficient} linear programs as opposed to semi-definite programs for general linear systems. This leads to scalable solutions that are tractable also for modern, large-scale control applications. For further reading on positive systems, we refer the interested reader to the survey~\cite{Rantzer2021b}. Interestingly, it was recently shown that exact solutions can be obtained for various model-based optimal control problems involving positive systems, see, for instance,~\cite{Rantzer2022,Ohlin2025,blanchini2023optimal} and, for minimax problems,~\cite{Gurpegui2023,Gurpegui2026}. Control of positive systems without knowledge of the model has also been addressed from the perspective of both offline data-driven control, see, e.g.,~\cite{miller2023data,shafai2022data,Miller2026,Iwata2025}, and adaptive control~\cite{bencherki2025adaptive,bencherki2024data,Bencherki2026}. We build on these results and extend them to the dual control setting to develop and analyze dual controllers whose goal is to balance exploration and exploitation.

\subsection{Contributions}
Our paper presents the following main contributions:
\begin{itemize}
    \item[\textit{(C1).}]  We present the general optimal dual control problem for positive systems, in which we allow for input randomization to facilitate exploration. The problem is subsequently reformulated as a standard two player zero-sum dynamic game that can be addressed using minimax dynamic programming techniques.
\item[\textit{(C2).}] For the scalar case with sign uncertainty in the input, we derive an \emph{exact} solution to the Bellman equation corresponding to the reformulated problem. This solution yields an explicit dual controller with provable, non-asymptotic optimality guarantees in the form of optimal game value and optimal \(\ell_1\)-gain from disturbance to state.
    \item[\textit{(C3).}] We investigate fundamental limitations of minimax dual control for positive systems by characterizing the best achievable induced \(\ell_1\)-gain from positive disturbances to the state under both the optimal dual controller and the certainty equivalence controller. This allows us to derive results similar to those in~\cite{vinnicombe2004examples}.
\item[\textit{(C4).}]  The paper solves the exact Bellman equation for scalar systems with sign
uncertainty, providing an \emph{exemplary} setting for understanding the
structure of optimal dual control for positive systems more broadly. The
techniques developed here naturally suggest extensions to broader model
classes, thereby laying the groundwork for future research on exact and
approximate dual control of positive systems..
\end{itemize}

\subsection{Organization}
Section~\ref{prelim} provides the notation and summarizes previous work, the challenges it raises, and the motivation for this paper. Section~\ref{gen-framework} formulates the general minimax dual control framework considered in this paper, recasts it as a zero-sum dynamic game, and develops its solution via minimax dynamic programming. Section~\ref{bellman-eq} presents our main results, namely an exact explicit solution to the Bellman equation for the case of sign input uncertainty in scalar positive systems. Section~\ref{fund-limit} examines fundamental limitations of minimax dual control for positive systems. 
Section~\ref{concl} concludes the paper and proposes possible directions for future work. Preliminary lemmas, their proofs, and the proofs of the main results are deferred to Appendices~\ref{appenA} and~\ref{main-res}. Basic facts regarding minimax dynamic programming are given in Appendix~\ref{app:minimax-dp}.

\section{Preliminaries and previous work}\label{prelim}
\subsection{Notation}
We let \(\mathbb{R}^n_+\) denote the nonnegative orthant in dimension \(n\). The vector \(\mathbf{1}\) denotes the all-ones vector, with dimension understood from context. For any vector \(x\), \(|x|\) denotes the elementwise absolute value. For a vector \(x\), \(\max x\) and \(\min x\) denote its largest and smallest entries, respectively. For a set \(\mathcal M\), \(|\mathcal M|\) denotes its cardinality. If \(z_t^{(A,B)} \in \mathbb{R}_+\) is defined for every pair \((A,B)\in\mathcal M\), we write
$
Z_t \coloneqq \{z_t^{(A,B)}\}_{(A,B)\in\mathcal M}
$
for the indexed collection of nonnegative real numbers \(z_t^{(A,B)}\), one for each \((A,B)\in\mathcal M\). Inequalities between vectors and matrices are understood elementwise.
\subsection{Minimax adaptive control for positive systems}\label{minimax-adapt-contr}
The work in~\cite{Bencherki2026} considered discrete-time linear positive systems of the form
\begin{align}\label{positive-system}
    x_{t+1} = A x_t + B u_t + w_t,\qquad x_0 \in \mathbb{R}_+^n,
\end{align}
where the dynamics are parameterized by \(A \in \mathbb{R}_+^{n \times n}\) and
\(B \in \mathbb{R}^{n \times m}\). 
The inputs satisfy
\[
\mathcal{U}(x) \subseteq \{u \in \mathbb{R}^m : Ax + Bu \geq 0,\ \text{for all } x \in \mathbb{R}_+^n\}.
\]
Thus, at each time \(t \in \mathbb{N}\), the control action satisfies \(u_t \in \mathcal{U}(x_t)\). The disturbance \(w\) is restricted to a set \(\mathcal{W}(x,u)\) that preserves state positivity pointwise in time, i.e.,
\[
\mathcal{W}(x,u) := \{w \in \mathbb{R}^n : w \ge -(Ax+Bu),\ \text{for all } x \in \mathbb{R}_+^n,\ \text{for all } u \in \mathcal{U}(x)\}.
\]
Note that the set of admissible inputs depends on the current state, while the set of admissible disturbances depends on both the current state and the input. This is consistent with the minimax framework, in which the disturbance \(w_t\) is chosen after the controller has selected the input \(u_t\). For further details, we refer the interested reader to~\cite{Bencherki2026}, where concrete choices of the admissible input set \(\mathcal{U}(x)\) are also discussed.

It is assumed that the true unknown positive system $(A_\star,B_\star)$ belongs to a compact model class \(\mathcal M\) of admissible system matrices \((A,B)\), which captures the parametric uncertainty in the dynamics.
The deterministic minimax adaptive control problem studied in~\cite{Bencherki2026} is
\[
\begin{aligned}
J_\ast(x_0)= {}&
\inf_{\mu}\;
\sup_{\substack{(A,B)\in\mathcal M\\
w\ge -(Ax+Bu)\\
T\in\mathbb N}}
\sum_{t=0}^{T}
\bigl(s^\top x_t+r^\top u_t-\gamma^\top|w_t|\bigr)
\\
&\text{s.t.}\qquad
x_{t+1}=Ax_t+Bu_t+w_t,\quad x_0\in\mathbb R_+^n,\\
&\phantom{\text{s.t.}\qquad}
u_t=\mu_t\bigl(x_0,\ldots,x_t;\,u_0,\ldots,u_{t-1}\bigr),\quad\forall t\in\mathbb N,\\
&\phantom{\text{s.t.}\qquad}
u_t\in\mathcal U(x_t),\quad \text{for all } t\in\mathbb N.
\end{aligned}
\]
This problem was reformulated by introducing a historical data variable for each \((A,B)\in\mathcal M\) and a new disturbance $v_t$, leading to a new dynamics described by
\[
\left\{
\begin{aligned}
x_{t+1} &= v_t,\\
z_{t+1}^{(A,B)}
&=
z_t^{(A,B)}
+
\gamma^\top |v_t-Ax_t-Bu_t|,
\qquad
z_0^{(A,B)}=0.
\end{aligned}
\right.
\]
This yields the equivalent standard minimax dynamic game between two
opposing players, i.e., the controller \(\eta\) and the adversary \(v\), given by
\begin{equation}\label{eq:minimax-dyn-game}
\begin{aligned}
{}&\inf_{\eta}\;
\sup_{\substack{v \ge 0\\ T \in \mathbb{N}}}
\quad
\sum_{t=0}^{T} \bigl(s^\top x_t + r^\top u_t\bigr)
-
\min_{(A,B)\in\mathcal M} z_{T+1}^{(A,B)}
\\[1mm]
&\text{s.t.}\quad
x_{t+1} = v_t,
\qquad
x_0 \in \mathbb{R}_+^n,
\\
&\phantom{\text{s.t.}\quad}
z_{t+1}^{(A,B)}
=
z_t^{(A,B)}
+
\gamma^\top |v_t - A x_t - B u_t|,
\qquad
z_0^{(A,B)} = 0,
\\
&\phantom{\text{s.t.}\quad}
Z_t \coloneqq \{z_t^{(A,B)}\}_{(A,B)\in\mathcal M},
\qquad
Z_0 \coloneqq \{0\}_{(A,B)\in\mathcal M},
\\
&\phantom{\text{s.t.}\quad}
u_t = \eta(x_t,Z_t) \in \mathcal{U}(x_t),
\qquad
\text{for all }t\in\mathbb N.
\end{aligned}
\end{equation}
This reformulated problem was then addressed using standard minimax dynamic
programming, yielding solutions to the corresponding Bellman inequality
and, consequently, stabilizing adaptive controllers.
\subsection{Motivation and challenges}\label{motiv}
The work in~\cite{Bencherki2026} focused on solving
Problem~\ref{prob:l1-robust-control} under deterministic inputs. This makes
value iteration quickly intractable when seeking an exact solution via minimax
dynamic programming. This intractability stems from the inherent difficulty
of balancing exploration and exploitation, which present two competing
objectives and typically result in a nonconvex optimization problem. To illustrate this, we consider the following example.
\begin{example}[The exploration--exploitation tradeoff for positive systems]
Consider
\[
x_{t+1}=ax_t+\sigma u_t+w_t,\qquad \sigma\in\{+,-\},
\]
where \(a\ge0\) is known, while the input direction \(\sigma\) is unknown. The input set is taken to be \( \mathcal{U}(x_t)=\{u_t\in\mathbb R:\ |u_t|\leq ax_t\}\),
which preserves positivity of the next state for both possible input
directions. We take \(s=1\) and \(r=0\) as cost parameters. The two model-specific histories evolve
according to
\[
z_+^{(+)}=z^{(+)}+|v-ax-u|,
\qquad
z_+^{(-)}=z^{(-)}+|v-ax+u|,
\qquad x_+=v\geq0.
\]
The value iteration is initialized by
$
V_0(x,z^{(+)},z^{(-)})
\coloneqq
-\gamma\min\{z^{(+)},z^{(-)}\},
$
and proceeds recursively according to
\[
V_{k+1}(x,z^{(+)},z^{(-)})
=
\min_{|u|\leq ax}\max_{v\geq0}
\left[
x+
V_k\left(
v,\,
z^{(+)}+|v-ax-u|,\,
z^{(-)}+|v-ax+u|
\right)
\right].
\]
The first iterate $V_1$ is finite if and only if \(\gamma\geq1\), and in that case
\begin{align*}
V_1(x,z^{(+)},z^{(-)})
&=
\min_{|u|\leq ax}\max_{v\geq0}
\left[
x-\gamma
\min\Bigl\{
z^{(+)}+|v-ax-u|,
z^{(-)}+|v-ax+u|
\Bigr\}
\right]  \\
&=
x-\gamma\min\{z^{(+)},z^{(-)}\}.
\end{align*}
The second iterate $V_2$ is finite if and only if \(\gamma\geq1\), and in that case
\begin{multline*}\label{eq1}
	V_2(x,z^{(+)},z^{(-)})
	=
	\min_{\lvert u\rvert\leq ax}\max_{v\geq0}
	\left[
	x+v-\gamma
	\min\Bigl\{
	z^{(+)}+\lvert v-ax-u\rvert,\;
	z^{(-)}+\lvert v-ax+u\rvert
	\Bigr\}
	\right] \\
	=
	x+ax+
	\min_{\lvert u\rvert\leq ax}
	\max\Bigl\{
	\underbrace{u-\gamma z^{(+)}}_{V_{+}(u)},\;
	\underbrace{-u-\gamma z^{(-)}}_{V_{-}(u)}
	\Bigr\} \\
	=
	\max\left\{
	(1+a)x-\tfrac{\gamma}{2}\bigl(z^{(+)}+z^{(-)}\bigr),\;
	x-\gamma z^{(+)},\;
	x-\gamma z^{(-)}
	\right\},
\end{multline*}
Finally, using the expression for \(V_2\), the third iterate $V_3$ is
\begin{multline*}
V_3(x,z^{(+)},z^{(-)})
=
\min_{|u|\leq ax}\max_{v\geq0}
\max\Biggl\{
x+(1+a)v
-\frac{\gamma}{2}
\Bigl(
z^{(+)}+z^{(-)}
+|v-ax-u|+|v-ax+u|
\Bigr),\\
x+v-\gamma\bigl(z^{(+)}+|v-ax-u|\bigr),\;
x+v-\gamma\bigl(z^{(-)}+|v-ax+u|\bigr)
\Biggr\}.
\end{multline*}
The third iterate $V_3$ is finite if and only if \(\gamma\geq1+a\), and in that case
the maximization over \(v\geq0\) gives
\begin{multline*}
	V_3(x,z^{(+)},z^{(-)})
	=(1+a)x+
	\min_{|u|\leq ax}
	\max\Biggl\{ 
	\underbrace{a^2x 
		-\frac{\gamma}{2}\bigl(z^{(+)}+z^{(-)}\bigr)
		-\bigl(\gamma-1-a\bigr)|u|}_{V_{\mathrm{avg}}(u)},\\
	\underbrace{u-\gamma z^{(+)}}_{V_{+}(u)},\;
	\underbrace{-u-\gamma z^{(-)}}_{V_{-}(u)}
	\Biggr\}.
\end{multline*}
For \(\gamma>1+a\), the minimization over \(u\) in \(V_3\) is nonconvex due to the concave term \(-(\gamma-1-a)\lvert u\rvert\). In contrast, at the second step the controller minimizes the maximum of two affine functions of \(u\), as shown in the left panel of Figure~\ref{fig:positive-exploration-three-steps}. At the third step, the objective becomes the maximum of two affine functions and one concave function, as shown in the right panel of Figure~\ref{fig:positive-exploration-three-steps}. Since the concave term is largest at \(u=0\), the minimizer of the pointwise maximum moves away from zero, causing the controller to probe the system. To deal with the non-convexity, we will employ input randomization which helps to convexify the problem.

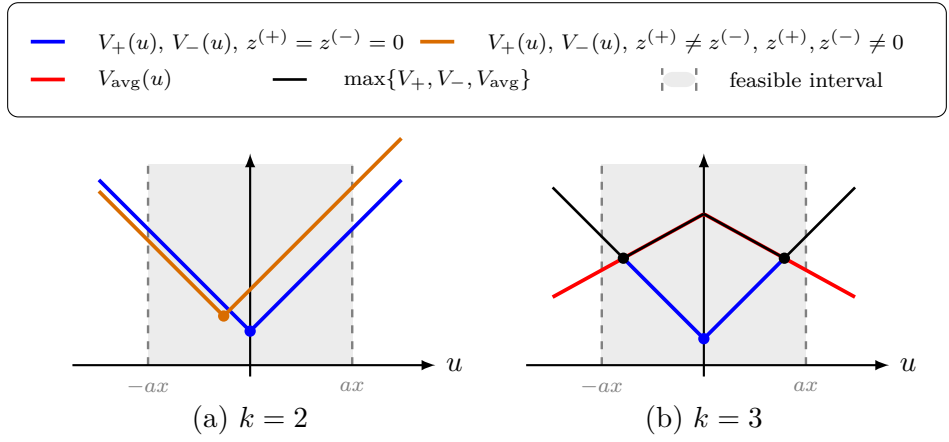
\begin{figure}[tbh]
\centering
\begin{tikzpicture}[
    >=latex,
    scale=1.0,
    axis/.style={->, line width=0.8pt},
    bluecurve/.style={blue, line width=1.4pt},
    orangecurve/.style={orange!85!black, line width=1.4pt},
    redcurve/.style={red, line width=1.4pt},
    blackcurve/.style={black, line width=1.0pt},
    constraint/.style={gray, dashed, line width=0.9pt},
    feasible/.style={gray!15},
    dot/.style={circle, fill, inner sep=1.5pt},
    legendbox/.style={
        draw,
        rounded corners,
        fill=white,
        inner xsep=8pt,
        inner ysep=5pt,
        font=\scriptsize
    }
]

\node[legendbox, anchor=north] at (3,4.8) {
	\begin{tikzpicture}[
		baseline,
		every node/.style={font=\scriptsize},
		bluecurve/.style={blue, line width=1.4pt},
		orangecurve/.style={orange!85!black, line width=1.4pt},
		redcurve/.style={red, line width=1.4pt},
		blackcurve/.style={black, line width=1.0pt},
		constraint/.style={gray, dashed, line width=0.9pt},
		feasible/.style={gray!15}
		]
		\draw[bluecurve] (0,0) -- (0.45,0);
		\node[anchor=west] at (0.60,0) {$V_{+}(u),\,V_{-}(u)$, $z^{(+)}=z^{(-)}=0$};
		
		\draw[orangecurve] (5.15,0) -- (5.60,0);
		\node[anchor=west] at (5.75,0) {$V_{+}(u),\,V_{-}(u)$, $z^{(+)}\neq z^{(-)}$, $z^{(+)},z^{(-)}\neq 0$};
		
		\draw[redcurve] (0,-0.50) -- (0.45,-0.50);
		\node[anchor=west] at (0.60,-0.50) {$V_{\mathrm{avg}}(u)$};
		
		\draw[blackcurve] (3.20,-0.50) -- (3.65,-0.50);
		\node[anchor=west] at (3.85,-0.50) {$\max\{V_{+},V_{-},V_{\mathrm{avg}}\}$};
		
		\fill[feasible] (8.35,-0.60) rectangle (8.80,-0.40);
		\draw[constraint] (8.35,-0.65) -- (8.35,-0.35);
		\draw[constraint] (8.80,-0.65) -- (8.80,-0.35);
		\node[anchor=west] at (8.95,-0.50) {feasible interval};
	\end{tikzpicture}
};

\begin{scope}[shift={(0,0)}]

    \pgfmathsetmacro{\Cblue}{0.45}
    \pgfmathsetmacro{\Corange}{1.25}
    \pgfmathsetmacro{\alphap}{0.25}
    \pgfmathsetmacro{\alpham}{0.95}
    \pgfmathsetmacro{\ustar}{0.5*(\alphap-\alpham)}
    \pgfmathsetmacro{\ystar}{\Corange-0.5*(\alphap+\alpham)}
    \pgfmathsetmacro{\umax}{1.35}

    \fill[feasible] (-\umax,0) rectangle (\umax,2.65);
    \draw[constraint] (-\umax,0) -- (-\umax,2.65);
    \draw[constraint] (\umax,0) -- (\umax,2.65);
    \node[gray,font=\scriptsize] at (-\umax,-0.28) {\(-ax\)};
    \node[gray,font=\scriptsize] at (\umax,-0.28) {\(ax\)};

    \draw[axis] (-2.35,0) -- (2.45,0) node[right] {$u$};
    \draw[axis] (0,-0.15) -- (0,2.8);

    \draw[bluecurve, domain=-2:0, samples=2]
        plot (\x,{\Cblue-\x});
    \draw[bluecurve, domain=0:2, samples=2]
        plot (\x,{\Cblue+\x});

    \draw[orangecurve, domain=-2:\ustar, samples=2]
        plot (\x,{\Corange-\x-\alpham});
    \draw[orangecurve, domain=\ustar:2, samples=2]
        plot (\x,{\Corange+\x-\alphap});

    \node[dot, blue] at (0,\Cblue) {};
    \node[dot, orange!85!black] at (\ustar,\ystar) {};

    \node at (0,-0.72) {(a) \(k=2\)};

\end{scope}

\begin{scope}[shift={(6.0,0)}]

    \pgfmathsetmacro{\L}{0.35}
    \pgfmathsetmacro{\C}{2.00}
    \pgfmathsetmacro{\beta}{0.55}
    \pgfmathsetmacro{\uint}{(\C-\L)/(1+\beta)}
    \pgfmathsetmacro{\nuint}{-\uint}
    \pgfmathsetmacro{\yint}{\L+\uint}
    \pgfmathsetmacro{\umax}{1.35}

    \fill[feasible] (-\umax,0) rectangle (\umax,2.65);
    \draw[constraint] (-\umax,0) -- (-\umax,2.65);
    \draw[constraint] (\umax,0) -- (\umax,2.65);
    \node[gray,font=\scriptsize] at (-\umax,-0.28) {\(-ax\)};
    \node[gray,font=\scriptsize] at (\umax,-0.28) {\(ax\)};

    \draw[axis] (-2.35,0) -- (2.45,0) node[right] {$u$};
    \draw[axis] (0,-0.15) -- (0,2.8);

    \draw[bluecurve, domain=\nuint:0, samples=2]
        plot (\x,{\L-\x});
    \draw[bluecurve, domain=0:\uint, samples=2]
        plot (\x,{\L+\x});

    \draw[redcurve, domain=-2:2, samples=100]
        plot (\x,{\C-\beta*abs(\x)});

    \draw[blackcurve, domain=-2:\nuint, samples=2]
        plot (\x,{\L-\x});
    \draw[blackcurve, domain=\nuint:0, samples=50]
        plot (\x,{\C-\beta*abs(\x)});
    \draw[blackcurve, domain=0:\uint, samples=50]
        plot (\x,{\C-\beta*abs(\x)});
    \draw[blackcurve, domain=\uint:2, samples=2]
        plot (\x,{\L+\x});

    \node[dot, black] at (\nuint,\yint) {};
    \node[dot, black] at (\uint,\yint) {};

    \node[dot, blue] at (0,\L) {};

    \node at (0,-0.72) {(b) \(k=3\)};

\end{scope}

\end{tikzpicture}

\caption{Optimal exploration in three steps for the scalar positive-system example. The optimal input is obtained by minimizing, over the admissible interval \(-ax\leq u\leq ax\) shown by the shaded region, the maximum of the functions appearing in the value iteration. \textbf{(a):} At \(k=2\), the objective is the maximum of two affine functions of \(u\). Without prior information, \(z^{(+)}=z^{(-)}\), the corresponding functions \(V_{+}(u)\) and \(V_{-}(u)\) are symmetric and are minimized at \(u=0\). Prior information shifts these functions, shown in orange, corresponding to \(u-\gamma z^{(+)}\) and \(-u-\gamma z^{(-)}\); the unconstrained minimizer is \(u^\ast=\tfrac{\gamma}{2}\bigl(z^{(+)}-z^{(-)}\bigr)\), and the resulting input is its projection onto \([-ax,ax]\). \textbf{(b):} At \(k=3\), for \(\gamma>1+a\), the objective is the pointwise maximum of \(V_{\mathrm{avg}}(u)\), shown in red, and \(V_{+}(u)\), \(V_{-}(u)\), shown in black. Since \(V_{\mathrm{avg}}(u)\) contains the concave term \(-(\gamma-1-a)\lvert u\rvert\), the minimization over \(u\) is generally nonconvex and may move the minimizer away from zero, thereby inducing exploration.}
\label{fig:positive-exploration-three-steps}
\end{figure}
\end{example}
To solve the exact Bellman equation, one must synthesize a policy that balances exploration and exploitation. Compared with general linear systems, dual control for positive systems poses additional challenges, yet it also benefits from additional structure:
\begin{itemize}
	\item[\textit{(1)}] The control input \(u_t\) must balance exploration and exploitation \emph{while preserving positivity}.
\item[\textit{(2)}] Once a dual policy has been synthesized, the intrinsic structure of positive systems often allows for efficient computation and implementation.
\end{itemize}
In the next section, we introduce a new minimax dual control formulation that
allows for input randomization, thereby convexifying the value iteration
steps and enabling an exact solution, as will be shown in later sections.
\section{General framework for minimax dual control for positive systems}\label{gen-framework}
In this section, we develop the general framework that extends the optimal control problem in Section~\ref{minimax-adapt-contr} by allowing randomized inputs. In the same spirit as~\cite{Bencherki2026}, we then present a reformulation, analogous to~\cite{Rantzer2021a}, that moves the dependence on the unknown system parameters \((A,B)\) to the terminal cost. We start by formulating the general problem.
\subsection{Minimax dual control under randomized inputs}\label{pbf}
In this paper, we extend the minimax adaptive control problem considered in~\cite{Bencherki2026} by allowing input randomization. To this end, we modify the optimal control problem from Section~\ref{prelim} so that it minimizes the expected value, leading to the following optimal dual control problem.
\begin{problem}\label{prob:l1-robust-control}
Find a causal potentially randomized policy \(\mu\) that solves
\[
\begin{aligned}
J_\ast(x_0)= {}&
\inf_{\mu}\;
\sup_{\substack{(A,B)\in\mathcal M\\
w\ge -(Ax+Bu)\\
T\in\mathbb N}} \mathbb{E}
\sum_{t=0}^{T}
\bigl(s^\top x_t+r^\top u_t-\gamma^\top|w_t|\bigr)
\\
&\mathrm{s.t.}\qquad
x_{t+1}=Ax_t+Bu_t+w_t,\quad x_0\in\mathbb R_+^n,\\
&\phantom{\mathrm{s.t.}\qquad}
u_t=\mu_t\bigl(x_0,\ldots,x_t;\,u_0,\ldots,u_{t-1}\bigr),\quad\text{ for all } t\in\mathbb N,\\
&\phantom{\mathrm{s.t.}\qquad}
u_t\in\mathcal U(x_t),\quad \mathrm{almost~surely}\text{, for all } t\in\mathbb N.
\end{aligned}
\]
Here, the policy $\mu$ generates, at time $t\in\mathbb{N}$, a (possibly) random input $u_t$ based on $u_0,\hdots,u_{t-1}$ and $x_0,\hdots,x_t$.
\end{problem}
In contrast to the minimax problem considered in~\cite{Bencherki2026}, Problem~\ref{prob:l1-robust-control} allows policies that generate randomized inputs, while imposing the constraint that \(u_t\in\mathcal U(x_t)\) almost surely. This requirement is enforced with probability one, as stated next.
\begin{assumption}\label{asm:UU}
	Inputs $u_t$ are positivity preserving almost surely, i.e., $u_t \in \mathcal{U}(x_t)$ a.s., where
	$
	\mathcal{U}(x) \subseteq \{u \in \mathbb{R}^m : Ax + Bu \geq 0,\ \forall x \in \mathbb{R}_+^n\}.
	$
\end{assumption}
The same requirement is imposed on the disturbance.
\begin{assumption}\label{asm:WW}
	The disturbance $w_t$ is positivity preserving almost surely, i.e., $w_t \in \mathcal{W}(x_t,u_t)$ a.s., where
	$
	\mathcal{W}(x,u)\coloneqq \{w\in\mathbb{R}^{n}\mid Ax+Bu+w\geq 0,\ \forall x\in\mathbb{R}_+^n,\ \forall u\in\mathcal{U}(x)\}.
	$
\end{assumption}
We require the worst case expected stage cost to be nonnegative, as stated in the following.
\begin{assumption}[Positivity of the expected stage cost]\label{ass:stage cost-pos}
The stage cost weights satisfy \(s\in\mathbb{R}^n_+\), \(r\in\mathbb{R}^m\), and \(\gamma\in\mathbb{R}^n_+\). Moreover, for every \((A,B)\in\mathcal M\), \(t\in\mathbb N\), \(x_t\in\mathbb{R}^n_+\), and \(u_t\in\mathcal U(x_t)\), we assume that
\(\max_{w_t\in\mathcal W(x_t,u_t)}\mathbb E\!\left[s^\top x_t+r^\top u_t-\gamma^\top |w_t|\right]\geq0\).
\end{assumption}

\begin{remark}
Since \(u_t\in\mathcal U(x_t)\) almost surely implies \(Ax_t+Bu_t\geq0\) almost surely, we have \(0\in\mathcal W(x_t,u_t)\) almost surely. Therefore, Assumption~\ref{ass:stage cost-pos} reduces to \(\mathbb E\!\left[s^\top x_t+r^\top u_t\right]\geq0\).
\end{remark}

\subsection{Reformulated problem}
To reformulate Problem~\ref{prob:l1-robust-control} as a standard minimax zero-sum dynamic game, we make two important changes: \begin{enumerate*}[label=\textit{(\alph*)}]
    \item we introduce a new auxiliary disturbance $v_t$ through which we remove $w_t$ and, more importantly, $A$ and $B$ from the system dynamics, and
    \item we introduce a model-specific history variable $z^{(A,B)}_{t}$ that ``compresses'' the past measurements and allows us to replace the original disturbance $w_t$ in the cost function by a new term that depends on $A$, $B$, and the new disturbance $v_t$.
\end{enumerate*}

First, we introduce the new auxiliary disturbance $v_t\geq 0$. Note that, since the disturbance $w_t$ enters additively and affects all states directly, Assumption~\ref{asm:WW} means that $w_t$ can drive the next state $x_{t+1}$ to any non-negative value. Hence, taking the supremum over $\{w_t\}_{t\in\mathbb{N}}$, with $w_t\in\mathcal{W}(x_t,u_t)$ for all $t\in\mathbb{N}$, for the model $x_{t+1}=Ax_t+Bu_t+w_t$ is equivalent to taking the supremum over $\{v_t\}_{t\in\mathbb{N}}$, with $v_t\in\mathbb{R}_{+}^{n}$ for all $t\in\mathbb{N}$, for the model 
\begin{equation}\label{eq:reformulated-model}
    x_{t+1} = v_t.
\end{equation}
Note that the new disturbance satisfies $w_t= v_{t} - Ax_t + Bu_t$. As mentioned before, this change of variables in the disturbance has eliminated the unknown system parameters $(A,B)$ and the original disturbance $w_t$ from the model dynamics.

Secondly, we observe that the original disturbance $w_t$ still appears in the cost function in Problem~\ref{prob:l1-robust-control}. To deal with this, we now introduce a model-specific history variable. Specifically, for each $(A,B)\in\mathcal{M}$, we define $z^{(A,B)}_{t}$ as 
\[
    z^{(A,B)}_{t} \coloneqq \sum_{\tau=0}^{t-1} |v_\tau - Ax_{\tau}-Bu_{\tau}|.
\]
Using the fact that $w_t = v_t - Ax_t+Bu_t$, we find that the model-specific history satisfies the following important relation:
\[
    \sum_{t=0}^{T} |w_t| = \sum_{t=0}^{T} |v_t-Ax_t-Bu_t| = z^{(A,B)}_{T+1}.
\]
This relation is useful because it allows the part of the cost function in Problem~\ref{prob:reformulated} that depends on $w_t$ to be turned into a terminal cost term that depends only on $z^{(A,B)}_{T}$. Finally, we note that the model-specific histories evolve according to
\begin{equation}\label{eq:model-specific-history-dynamics}
    z_{t+1}^{(A,B)} = z_{t}^{(A,B)} + |v_t-Ax_t-Bu_t|,\qquad z_0^{(A,B)} = 0.
\end{equation}
By augmenting the dynamics in~\eqref{eq:reformulated-model} with the dynamics of the model-specific histories in~\eqref{eq:model-specific-history-dynamics}, we obtain a dynamic model that describes a new, augmented state $(x_t,Z_t)$ with $Z_t=\{z_t^{(A,B)}\}_{(A,B)\in\mathcal{M}}$. Taking $(x_t,Z_t)$ as our so-called ``hyperstate'', we obtain the following reformulated problem.
\begin{problem}\label{prob:reformulated}
    Find a causal potentially randomized policy $\eta$ that solves
    \[
    \begin{aligned}
    {}&\inf_{\eta}\;
    \sup_{\substack{v \ge 0\\ T \in \mathbb{N}}}
    \quad
    \mathbb{E} \!\left[\sum_{t=0}^{T} \bigl(s^\top x_t + r^\top u_t\bigr)\right]
    -
    \min_{(A,B)\in\mathcal M} \mathbb{E} \left[\gamma^\top z_{T+1}^{(A,B)}\right]
    \\[1mm]
    &\mathrm{s.t.}\quad
    x_{t+1} = v_t,
    \qquad
    x_0 \in \mathbb{R}_+^n,
    \\
    &\phantom{\mathrm{s.t.}\quad}
    z_{t+1}^{(A,B)}
    =
    z_t^{(A,B)}
    +
     |v_t - A x_t - B u_t|,
    \
    z_0^{(A,B)} = 0,
    \\
    &\phantom{\mathrm{s.t.}\quad}
    Z_t \coloneqq \{z_t^{(A,B)}\}_{(A,B)\in\mathcal M},
    \qquad
    Z_0 \coloneqq \{0\}_{(A,B)\in\mathcal M},
    \\
    &\phantom{\mathrm{s.t.}\quad}
    u_t = \eta(x_t,Z_t),
    \qquad
    \text{ for all }t\in\mathbb N,\\
    &\phantom{\mathrm{s.t.}\quad}u_t\in\mathcal{U}(x_t),\quad \mathrm{almost~surely}, \text{ for all }t\in\mathbb{N}.
    \end{aligned}
    \]    
    Here, the policy $\eta$ generates, at time $t\in\mathbb{N}$, a (possibly) random input $u_t$ based on $x_t$ and $Z_t$.
\end{problem}
Notice that the dynamics in Problem~\ref{prob:reformulated} do not depend on the unknown system parameters \((A,B)\) but instead we evaluate the model-specific histories for all $(A,B)\in\mathcal{M}$. As a result, Problem~\ref{prob:reformulated} is a standard zero-sum dynamic game for positive systems, and it can be addressed using minimax dynamic programming.
One important remark is that the vector \(Z_t=\{z_t^{(A,B)}\}_{(A,B)\in\mathcal{M}}\in\mathbb{R}_+^{|\mathcal{M}|n}\) is, in general, infinite dimensional when the model class \(\mathcal{M}\) contains infinitely many parameter pairs $(A,B)$. In Section~\ref{bellman-eq}, we restrict attention to a finite set \(\mathcal{M}\), where this reformulation is instrumental in deriving an exact solution of the corresponding Bellman equation and an explicit optimal dual control policy. Extending this approach to more general, possibly infinite, model classes \(\mathcal{M}\) is an interesting direction for future work, for example via finite dimensional representations or approximations of \(Z_t\).

\subsection{Solution via minimax dynamic programming}\label{reform}
As noted above, a key advantage of Problem~\ref{prob:reformulated} is that it
takes the form of a standard zero-sum dynamic game and can therefore be adresssed
using minimax dynamic programming. We now use this formulation to establish the
equivalence between the original problem, Problem~\ref{prob:l1-robust-control},
and its reformulation, Problem~\ref{prob:reformulated}.

For a value function \(V(x,Z)\), define the operators \(\mathcal F\) and
\(\mathcal F_u\) by
\begin{align}\label{Foper}
	\mathcal F V(x,Z)
	\coloneqq
	\min_{\substack{u\in\mathcal U(x)\\ \mathrm{a.s.}}}
	\underbrace{
		\max_{v\geq 0}
		\left(
		\mathbb E\!\left[s^\top x+r^\top u\right]
		+V(v,Z_+)
		\right)
	}_{\eqqcolon\,\mathcal F_uV(x,Z)},
\end{align}
where the expectation is taken with respect to the randomized control input. Here,
$
Z_+
\coloneqq
\bigl\{
z_+^{(A,B)}
\bigr\}_{(A,B)\in\mathcal M}
$
denotes the updated nonnegative history collection, whose components evolve
according to
\begin{align}\label{hist}
    z_+^{(A,B)}
=
z^{(A,B)}
+
\left|v-Ax-Bu\right|,
\qquad
(A,B)\in\mathcal M.
\end{align}
The following theorem characterizes crucial equivalence between Problem~\ref{prob:l1-robust-control}
and Problem~\ref{prob:reformulated}.
\begin{theorem}\label{thm:l1-dp}
Given \(\gamma\in\mathbb R_+^n\) and a parameter set \(\mathcal M\) satisfying
\((A_\star,B_\star)\in\mathcal M\), define \(\mathcal F\) by~\eqref{Foper} and
\(\{V_k\}_{k\in\mathbb N}\) by the following recursion:
\[
V_0(x,Z)\coloneqq-\min_{(A,B)\in\mathcal M}
\mathbb E\!\left[\gamma^\top z^{(A,B)}\right],
\qquad
V_k\coloneqq\mathcal F V_{k-1},
\quad k\geq1,
\]
where \(z^{(A,B)}\) is constructed for each \((A,B)\in\mathcal M\) according
to~\eqref{hist}. Suppose that Assumptions~\ref{asm:UU}-\ref{ass:stage cost-pos} hold. Then,
the following statements hold.

\begin{enumerate}
\item[\textit{(i)}]
Problems~\ref{prob:l1-robust-control} and~\ref{prob:reformulated} have the same
value. This value is finite if and only if
\(\{V_k(x,0)\}_{k=0}^{\infty}\) is bounded above for every
\(x\in\mathbb R_+^n\). In that case,
$
V_\star(x,Z)\coloneqq\lim_{k\to\infty}V_k(x,Z)$ and $
J_\star(x_0)=V_\star(x_0,0).
$

\item[\textit{(ii)}]
Let \(\eta^\star(x,Z)\) be the minimzing $u$ in 
\(\mathcal F V_\star(x,Z)\). Then \(\eta^\star\) is optimal for
Problem~\ref{prob:reformulated}, and the causal randomized policy
\[
u_t=\mu_t^\star(x_0,\ldots,x_t;u_0,\ldots,u_{t-1})
\coloneqq
\eta^\star\!\left(
x_t,
\left\{
\sum_{\tau=0}^{t-1}
\lvert x_{\tau+1}-Ax_\tau-Bu_\tau\rvert
\right\}_{(A,B)\in\mathcal M}
\right)
\]
is optimal for Problem~\ref{prob:l1-robust-control}.

\item[\textit{(iii)}]
Suppose that there exist a function \(\overbar V\) and a possibly randomized policy
\(\overbar\eta\) such that
$
V_0(x,Z)\leq\overbar V(x,Z),
$ and $
\mathcal F_{\overbar\eta}\overbar V(x,Z)\leq\overbar V(x,Z) $
for every admissible \((x,Z)\), where
$
\mathcal F_{\overbar\eta}V(x,Z)
\coloneqq
\mathcal F_uV(x,Z)\big|_{u=\overbar\eta(x,Z)}.
$
Then, the causal possibly randomized policy
\[
u_t=\overbar\mu_t(x_0,\ldots,x_t;u_0,\ldots,u_{t-1})
\coloneqq
\overbar\eta\!\left(
x_t,
\left\{
\sum_{\tau=0}^{t-1}
\lvert x_{\tau+1}-Ax_\tau-Bu_\tau\rvert
\right\}_{(A,B)\in\mathcal M}
\right)
\]
satisfies
$
V_\star(x_0,0)\leq J_{\overbar\mu}(x_0)\leq\overbar V(x_0,0).
$
In particular, \(\overbar\mu\) is optimal whenever \(\overbar V=V_\star\).
\end{enumerate}
\end{theorem}
  {\em Proof.}  See appendix~\ref{dp-general-case-proof}. \qedblack
  
Theorem~\ref{thm:l1-dp} shows that the reformulated problem can be used to compute both the optimal value of Problem~\ref{prob:l1-robust-control} and a control law that achieves this value. Moreover, it also provides a way to approximately solve Problem~\ref{prob:l1-robust-control}, leading to suboptimal controllers. 

\section{Scalar case}\label{bellman-eq}
In this section, we use Theorem~\ref{thm:l1-dp} to derive an exact dual control solution for model sets of the form
$
\mathcal M=\{(a,b),(a,-b)\}.
$ As we will see in Section~\ref{fund-limit}, this also allows us to study some fundamental limitations of minimax dual control for positive systems. For \(a>0\), \(b>0\), and \(r=0\), we define the admissible input set as
$
\mathcal U(x)\coloneqq\left\{u\in\mathbb R:\ |u|\leq\tfrac{a}{b}x\right\},
$
and Problem~\ref{prob:reformulated} reduces to the following problem.
\begin{problem}\label{pb2}
Find a causal potentially randomized policy \(\eta\) solving
\begin{equation}\label{eq:scalar-sign-minimax-min-history}
\hspace{-0.5em}
\begin{aligned}
J_\star(x_0)
&=
\inf_{\eta}\;
\sup_{\substack{v\ge0\\ T\in\mathbb N}}
\left(
\mathbb E\!\left[
\sum_{t=0}^{T} x_t
\right]
-
\gamma
\min\left\{
\mathbb E\!\left[z^{(+)}_{T+1}\right],
\mathbb E\!\left[z^{(-)}_{T+1}\right]
\right\}
\right)
\\
\mathrm{s.t.}\quad
&
x_{t+1}=v_t,
\qquad
x_0\in\mathbb R_+,
\\
&
z^{(+)}_{t+1}
=
z^{(+)}_{t}
+
|v_t-a x_t-bu_t|,
\qquad
z^{(+)}_{0}=0,
\\
&
z^{(-)}_{t+1}
=
z^{(-)}_{t}
+
|v_t-a x_t+bu_t|,
\qquad
z^{(-)}_{0}=0,
\\
&
Z_t
\coloneqq
\{z_t^{(+)},z_t^{(-)}\},
\qquad
Z_0
=
\{0,0\},
\\
&
u_t
=
\eta_t(x_t,Z_t),
\qquad
|u_t|\le \tfrac{a}{b}x_t \ \text{almost surely}.
\end{aligned}
\end{equation}
\end{problem}
The following remark characterizes the moment constraints induced by the almost sure input constraint $|u_t|\le \tfrac{a}{b}x_t$ and will be useful when minimizing over randomized control inputs.
\begin{remark}
 Assume that \(\mathbb P(\lvert u_t\rvert\leq\tfrac{a}{b}x_t)=1\). Since expectation is monotone, \(\mathbb E[\lvert u_t\rvert]\leq \mathbb E[\tfrac{a}{b}x_t]\). Moreover, by the triangle inequality, \(\lvert\mathbb E[u_t]\rvert\leq \mathbb E[\lvert u_t\rvert]\). Consequently, \(\lvert\mathbb E[u_t]\rvert\leq \mathbb E[\lvert u_t\rvert]\leq \mathbb E[\tfrac{a}{b}x_t]\). If \(x_t\) is observed by the controller at time \(t\), this reduces to the pointwise bound \(\lvert\mathbb E[u_t]\rvert\leq \mathbb E[\lvert u_t\rvert]\leq\tfrac{a}{b}x_t\).
\end{remark}
\subsection{Exact solution of the Bellman equation}
To characterize the optimal solution of Problem~\ref{pb2}, we distinguish two regimes depending on the value of \(a\):
\begin{itemize}
	\item[\emph{(i)}] If \(a\ge\tfrac12\), we derive an optimal adaptive dual controller that balances exploration and exploitation. In this regime, active learning of the unknown input direction is essential for achieving the optimal minimax value.

	\item[\emph{(ii)}] If \(0\leq a<\tfrac12\), the uncontrolled dynamics are stable enough that the zero controller can be optimal when \(\gamma\) is sufficiently small. In this case, exploration is too costly relative to its benefit, and applying no control is preferable. As \(\gamma\) increases, however, exploration becomes worthwhile: the controller activates the input, learns the input direction, and reduces the overall cost further.
\end{itemize}
In the next sections, we will treat each of the above cases separately.
\subsection{Case $a\ge\tfrac{1}{2}$}\label{scalar-regime1}
The following theorem characterizes the optimal solution to
Problem~\ref{pb2} for \(a\ge\tfrac12\) and \(b>0\). The case \(b<0\) follows
by symmetry of \(\mathcal M\).
\begin{theorem}\label{thm:bellman-termes}
Assume without loss of generality that \(b>0\) and \(a\ge\tfrac12\). Define the sequence \(\{V_k\}_{k\in\mathbb N}\) recursively by
\begin{align}\label{vk}
    V_0(x,z^{(+)},z^{(-)})
\coloneqq
\max\Big\{
\mathbb E[x-\gamma z^{(+)}],\;
\mathbb E[x-\gamma z^{(-)}]
\Big\},
\qquad
V_k \coloneqq \mathcal F V_{k-1},
\quad k\ge1.
\end{align}
Then, the following statements hold.
\begin{enumerate}
\item[\textit{(i)}] The \(k\)-th iterate of value iteration $V_k$ is finite if and only if
$
\gamma \geq \max_{0\leq j\leq k-1} h_j(a,\gamma),
$
where the sequence \(\{h_k\}_{k\in\mathbb N}\) is defined by
\begin{align}\label{hk}
    h_0(a,\gamma)=1,\quad
    h_1(a,\gamma)=1+a,\quad
    h_{k+1}(a,\gamma)
    =
    1+2a\,h_k(a,\gamma)-a\gamma,
    \quad
    k\geq 1.
\end{align}
In that case,
\begin{multline*}
V_k(x,z^{(+)},z^{(-)})
=
\max\Bigg\{
\mathbb E\!\left[x-\gamma z^{(+)}\right],\;
\mathbb E\!\left[x-\gamma z^{(-)}\right],\;
\mathbb E\!\Big[
h_1(a,\gamma)x
-\tfrac{\gamma}{2}\bigl(z^{(+)}+z^{(-)}\bigr)
\Big],\\
\mathbb E\!\Big[
h_k(a,\gamma)x
-\tfrac{\gamma}{2}\bigl(z^{(+)}+z^{(-)}\bigr)
\Big]
\Bigg\}.
\end{multline*}
\item[\textit{(ii)}] The value iteration sequence
\(\{V_k\}_{k\in \mathbb N}\) converges to a finite limit if and only if
\(\gamma\geq\gamma_\star\coloneqq1+2a\). For every
\(\gamma\geq\gamma_\star\), the sequence reaches its limit after one iteration,
that is,
\[
V_k(x,z^{(+)},z^{(-)})
=
V_\star(x,z^{(+)},z^{(-)}),
\qquad k\geq1,
\]
where
\[
V_\star(x,z^{(+)},z^{(-)})
=
\max\Bigg\{
\mathbb E\!\left[x-\gamma z^{(+)}\right],\;
\mathbb E\!\left[x-\gamma z^{(-)}\right],\;
\mathbb E\!\Big[
(1+a)x
-\tfrac{\gamma}{2}\bigl(z^{(+)}+z^{(-)}\bigr)
\Big]
\Bigg\}.
\]
\item[\textit{(iii)}] For inputs \(u_t\) satisfying
\(
|u_t|\leq \tfrac{a}{b}x_t
\)
almost surely,
the optimal minimizing policy of 
Problem~\ref{pb2} is achieved by the policy
\(u_t^\star=\eta^\star(x_t,z_t^{(+)},z_t^{(-)})\) where
\begin{equation}\label{eq:ustar-cases1-thm}
\left\{
\begin{aligned}
u_t^\star
&=
-\tfrac{a}{b}x_t,
&&
\text{if }
z_t^{(+)}-z_t^{(-)}
<
-\tfrac{2a}{\gamma}x_t,
\\[2pt]
u_t^\star
&=
\tfrac{a}{b}x_t,
&&
\text{if }
z_t^{(+)}-z_t^{(-)}
>
\tfrac{2a}{\gamma}x_t,
\\[2pt]
\mathbb E[u_t^\star]
&=
\tfrac{\gamma\bigl(z_t^{(+)}-z_t^{(-)}\bigr)}{2b},
\qquad
\mathbb E\!\bigl[|u_t^\star|\bigr]
=
\tfrac{a}{b}x_t,
&&
\text{if }
\bigl|z_t^{(+)}-z_t^{(-)}\bigr|
\le
\tfrac{2a}{\gamma}x_t.
\end{aligned}
\right.
\end{equation}
and by
\begin{align}\label{dual-cont}
    u_t
=
\mu_t^\star(x_0,\ldots,x_t;u_0,\ldots,u_{t-1})
\coloneqq
\eta^\star\!\left(
x_t,\,
\sum_{\tau=0}^{t-1}
\left|x_{\tau+1}-ax_\tau-bu_\tau\right|,
\sum_{\tau=0}^{t-1}
\left|x_{\tau+1}-ax_\tau+bu_\tau\right|
\right)
\end{align}
for Problem~\ref{prob:l1-robust-control}.
\item[\textit{(iv)}] The values of Problem~\ref{pb2} and
Problem~\ref{prob:l1-robust-control} are finite if and only if
\(\gamma\geq\gamma_\star\). Moreover, in that case, the optimal value of both
problems is
\[
J_\star(x_0)
=
V_k(x_0,0,0)
=
(1+a)x_0,
\qquad k\geq1,
\]
and is achieved by the policy
\(u_t=\eta^\star(x_t,z_t^{(+)},z_t^{(-)})\) in~\eqref{eq:ustar-cases1-thm} for
Problem~\ref{pb2}, and by~\eqref{dual-cont}
for Problem~\ref{prob:l1-robust-control}.
\end{enumerate}
\end{theorem}
{\em Proof.}
See Appendix~\ref{thm:bellman-branches-append}.  \qedblack
\begin{remark}
In Theorem~\ref{thm:bellman-termes}, the Bellman recursion is initialized at
$
V_0(x,z^{(+)},z^{(-)})
\coloneqq
\max\Big\{
\mathbb{E}[x-\gamma z^{(+)}],\;
\mathbb{E}[x-\gamma z^{(-)}]
\Big\},
$
which is the model-based value function and a lower bound on the cost achieved by any adaptive policy. This serves as a warm start in place of the true terminal cost
$-
\gamma
\min\left\{
\mathbb E\!\left[z^{(+)}\right],
\mathbb E\!\left[z^{(-)}\right]
\right\}$ that appears in~\eqref{eq:scalar-sign-minimax-min-history}.
Moreover, starting from the true terminal cost, a single Bellman step yields the same model-based cost. This is shown in Appendix~\ref{thm:bellman-branches-append}.
\end{remark}
\subsection{Case $0\le a < \tfrac{1}{2}$}\label{scalar-regime2}
We proceed by solving Problem~\ref{pb2} with \(0\le a< \tfrac12\), which yields the following theorem.
\begin{theorem}
\label{thm:small-a-zero-policy}
Assume without loss of generality that \(b>0\) and \(0\le a<\tfrac12\). Define the sequence \(\{V_k\}_{k\in\mathbb N}\) recursively by
\begin{align}\label{vk}
    V_0(x,z^{(+)},z^{(-)})
    \coloneqq
    \max\Big\{
    \mathbb E[x-\gamma z^{(+)}],\;
    \mathbb E[x-\gamma z^{(-)}]
    \Big\},
    \qquad
    V_k \coloneqq \mathcal F V_{k-1},
    \quad k\ge1.
\end{align}
Then the following statements hold:
\begin{enumerate}
\item[\textit{(i)}] The \(k\)-th iterate of value iteration $V_k$ is finite if and only if
$
\gamma \geq \max_{0\leq j\leq k-1} h_j(a,\gamma),
$
where the sequence \(\{h_k\}_{k\in\mathbb N}\) is defined by
\begin{align}\label{hk1}
    h_0(a,\gamma)=1,\quad
    h_1(a,\gamma)=1+a,\quad
    h_{k+1}(a,\gamma)
    =
    1+2a\,h_k(a,\gamma)-a\gamma,
    \quad
    k\in\mathbb N\setminus\{0\}.
\end{align}
In that case,
\begin{multline*}
V_k(x,z^{(+)},z^{(-)})
=
\max\Bigg\{
\mathbb E\!\left[x-\gamma z^{(+)}\right],\;
\mathbb E\!\left[x-\gamma z^{(-)}\right],\;
\mathbb E\!\Big[
h_1(a,\gamma)x
-\tfrac{\gamma}{2}\bigl(z^{(+)}+z^{(-)}\bigr)
\Big],\\
\mathbb E\!\Big[
h_k(a,\gamma)x
-\tfrac{\gamma}{2}\bigl(z^{(+)}+z^{(-)}\bigr)
\Big]
\Bigg\}.
\end{multline*}
\item[\textit{(ii)}] The value iteration sequence \(\{V_k\}_{k\in\mathbb N}\) converges to a finite limit if and only if \(\gamma \geq \gamma_\star \coloneqq \tfrac{1}{1-a}\). In this case, \(V_k(x,z^{(+)},z^{(-)})\longrightarrow V_\star(x,z^{(+)},z^{(-)})\) as \(k\to\infty\), where
\begin{multline}\label{V*_aless_half}
V_\star(x,z^{(+)},z^{(-)})
=
\max\Bigg\{
\mathbb E\!\left[x-\gamma z^{(+)}\right],\;
\mathbb E\!\left[x-\gamma z^{(-)}\right],\\
\mathbb E\!\Bigg[
\tfrac{1-a\gamma}{1-2a}\,x
-\tfrac{\gamma}{2}\bigl(z^{(+)}+z^{(-)}\bigr)
\Bigg],\;
\mathbb E\!\Bigg[
(1+a)x
-\tfrac{\gamma}{2}\bigl(z^{(+)}+z^{(-)}\bigr)
\Bigg]
\Bigg\}.
\end{multline}
\item[\textit{(iii)}]
If \(\gamma=\gamma_\star\coloneqq \tfrac{1}{1-a}\), then value iteration has the fixed point \(V_\star\big\rvert_{\gamma=\gamma_\star}\) given by
\[
V_\star\big\rvert_{\gamma=\gamma_\star}(x,z^{(+)},z^{(-)})
=
\max\Bigg\{
\mathbb E\!\left[x-\gamma_\star z^{(+)}\right],\;
\mathbb E\!\left[x-\gamma_\star z^{(-)}\right],\;
\mathbb E\!\Big[
\gamma_\star x-\tfrac{\gamma_\star}{2}\bigl(z^{(+)}+z^{(-)}\bigr)
\Big]
\Bigg\}.
\]
The optimal cost for both Problem~\ref{pb2} and Problem~\ref{prob:l1-robust-control} is
\[
J_\star(x_0)=V_\star(x_0,0,0)=\gamma_\star x_0=\tfrac{1}{1-a}x_0,
\]
and for Problem~\ref{pb2} it is achieved by the policy \(u_t^\star=\eta^\star(x_t,z_t^{(+)},z_t^{(-)})\) of the form
\begin{align*}
\begin{cases}
u_t^\star=-\tfrac{a}{b}x_t,
&
z_t^{(-)}-z_t^{(+)}>2a(1-a)x_t,
\\[4pt]
u_t^\star=\tfrac{z_t^{(+)}-z_t^{(-)}}{2b(1-a)},
&
\lvert z_t^{(+)}-z_t^{(-)}\rvert\leq2a(1-a)x_t,
\\[6pt]
u_t^\star=\tfrac{a}{b}x_t,
&
z_t^{(+)}-z_t^{(-)}>2a(1-a)x_t.
\end{cases}
\end{align*}
The corresponding optimal policy for Problem~\ref{prob:l1-robust-control} is
\begin{align}\label{eq:ustar-cases-orig}
u_t
=
\mu_t^\star(x_0,\ldots,x_t;u_0,\ldots,u_{t-1})
\coloneqq
\eta^\star\!\left(
x_t,\,
\sum_{\tau=0}^{t-1}\left|x_{\tau+1}-ax_\tau-bu_\tau\right|,
\sum_{\tau=0}^{t-1}\left|x_{\tau+1}-ax_\tau+bu_\tau\right|
\right).
\end{align}
When \(z_0^{(+)}=z_0^{(-)}=0\), this policy reduces to the deterministic zero input policy \(u_t^\star=0\) for all \(t\geq0\).

\item[\textit{(iv)}]
If, on the other hand, \(\gamma>\gamma_\star=\tfrac{1}{1-a}\), then the optimal cost value for both Problem~\ref{pb2} and Problem~\ref{prob:l1-robust-control} is
\[
J_\star(x_0)
=
V_\star(x_0,0,0)
=
\max\Bigg\{
(1+a)x_0,\;
\tfrac{1-a\gamma}{1-2a}x_0
\Bigg\},
\]
achieved for Problem~\ref{pb2} by the policy \(u_t^\star=\eta^\star(x_t,z_t^{(+)},z_t^{(-)})\) of the form
\begin{equation}\label{eq:ustar-cases}
\left\{
\begin{aligned}
u_t^\star
&=
-\tfrac{a}{b}x_t,
&&
\text{if }
z_t^{(-)}-z_t^{(+)}
>
\tfrac{2a}{\gamma}x_t,
\\[2pt]
u_t^\star
&=
\tfrac{a}{b}x_t,
&&
\text{if }
z_t^{(+)}-z_t^{(-)}
>
\tfrac{2a}{\gamma}x_t,
\\[2pt]
\mathbb E[u_t^\star]
&=
\tfrac{\gamma\bigl(z_t^{(+)}-z_t^{(-)}\bigr)}{2b},
\qquad
\mathbb E\!\bigl[|u_t^\star|\bigr]
=
\tfrac{a}{b}x_t,
&&
\text{if }
\bigl|z_t^{(+)}-z_t^{(-)}\bigr|
\le
\tfrac{2a}{\gamma}x_t.
\end{aligned}
\right.
\end{equation}
The corresponding policy for Problem~\ref{prob:l1-robust-control} is
\begin{align}\label{eq:ustar-cases-orig}
  u_t
=
\mu_t^\star(x_0,\ldots,x_t;u_0,\ldots,u_{t-1})
\coloneqq
\eta^\star\!\left(
x_t,\,
\sum_{\tau=0}^{t-1}
\left|x_{\tau+1}-ax_\tau-bu_\tau\right|,
\sum_{\tau=0}^{t-1}
\left|x_{\tau+1}-ax_\tau+bu_\tau\right|
\right).  
\end{align}

\end{enumerate}
\end{theorem}
{\em Proof.}
See Appendix~\ref{thm:small-a-zero-policy-append} \qedblack
\begin{remark}
In the case \(0\le a<\tfrac{1}{2}\), the system is sufficiently stable that, at \(\gamma=\gamma_\star\), the controller gains no benefit from exploration and therefore deactivates the input, yielding the cost \(J_\star(x_0)=V_\star(x_{0},0,0)=\tfrac{1}{1-a}\,x_0\). As \(\gamma>\gamma_\star\) increases, the adversary is weakened; the controller exploits this by applying nonzero exploratory inputs to learn the dynamics and reduce the cost below the \(\tfrac{1}{1-a}\) threshold, leading to the better cost \(J_\star(x_0)=V_\star(x_{0},0,0)=\max\Big\{(1+a)x_{0},\ \tfrac{1-a\gamma}{1-2a}\,x_{0}\Big\}<\tfrac{1}{1-a}\,x_0\).
\end{remark}
See Figure~\ref{fig:gain-cost-policy-summary} for a schematic summary of the results of Theorems~\ref{thm:bellman-termes} and~\ref{thm:small-a-zero-policy}.
\begin{figure}[h]
\centering
\begin{tikzpicture}[
    >=latex,
    scale=0.82,
    axis/.style={->, line width=0.9pt},
    threshold/.style={gray!70, dashed, line width=0.9pt},
    regionlabel/.style={font=\scriptsize, align=center},
    title/.style={font=\scriptsize\bfseries, align=center}
]

\begin{scope}[shift={(-6.4,0)}]

\pgfmathsetmacro{\L}{0}
\pgfmathsetmacro{\A}{6.4}
\pgfmathsetmacro{\R}{12.8}
\pgfmathsetmacro{\Ytop}{3.0}

\fill[gray!10] (\L,0) rectangle (\A,\Ytop);
\fill[gray!18] (\A,0) rectangle (\R,\Ytop);

\draw[threshold] (\A,-0.25) -- (\A,\Ytop+0.1);

\draw[axis] (\L-0.2,0) -- (\R+0.35,0)
node[right] {\(a\)};

\node[regionlabel] at (\L,-0.4) {\(0\)};
\node[regionlabel] at (\A,-0.4) {\(\tfrac12\)};

\node[title] at (3.2,2.4)
{Case 1: \(0\le a<\tfrac12\)};

\node[title] at (9.6,2.4)
{Case 2: \(a\ge\tfrac12\)};

\node[regionlabel] at (3.2,1.6)
{\(\gamma_\star(a)=\dfrac{1}{1-a}\)};

\node[regionlabel] at (9.6,1.6)
{\(\gamma_\star(a)=1+2a\)};

\node[regionlabel] at (3.2,0.65)
{\(\displaystyle
J_\star(x_0)=
\max\!\left\{
1+a,\tfrac{1-a\gamma}{1-2a}
\right\}x_0,\, \gamma \ge \gamma_\star\) };

\node[regionlabel] at (9.6,0.65)
{\(\displaystyle J_\star(x_0)=(1+a)x_0\)};

\draw[black,line width=1.1pt]
(\L,\Ytop+0.12) -- (\R,\Ytop+0.12);

\node at (6.4,-1.0)
{\textbf{(a)} Optimal cost and $\ell_1$-gain for the different regimes of $a$.};

\end{scope}

\begin{scope}[shift={(0,-5.2)}]

\pgfmathsetmacro{\L}{-6.4}
\pgfmathsetmacro{\R}{6.4}
\pgfmathsetmacro{\A}{-2.35}
\pgfmathsetmacro{\B}{2.35}
\pgfmathsetmacro{\Ytop}{3.05}

\fill[blue!9] (\L,0) rectangle (\A,\Ytop);
\fill[red!9]  (\A,0) rectangle (\B,\Ytop);
\fill[blue!9] (\B,0) rectangle (\R,\Ytop);

\draw[threshold] (\A,-0.25) -- (\A,\Ytop+0.1);
\draw[threshold] (\B,-0.25) -- (\B,\Ytop+0.1);

\draw[axis] (\L-0.3,0) -- (\R+0.45,0)
    node[right] {{\small \(z_t^{(+)}-z_t^{(-)}\)}};

\node[regionlabel] at (\A,-0.55) {\(-\dfrac{2a}{\gamma}x_t\)};
\node[regionlabel] at (\B,-0.55) {\(\dfrac{2a}{\gamma}x_t\)};
\node[regionlabel] at (0,-0.35) {\(0\)};

\node[title, blue!70!black] at (-4.35,2.45)
{Certainty\\equivalence};
\node[title, red!70!black] at (0,2.45)
{Optimal exploration};
\node[title, blue!70!black] at (4.35,2.45)
{Certainty\\equivalence};

\node[regionlabel, blue!70!black] at (-4.35,1.30)
{\(u_t^\star=-\dfrac{a}{b}x_t\)};

\node[regionlabel, red!70!black] at (0,1.30)
{\(\mathbb E[u_t^\star]=\dfrac{\gamma(z_t^{(+)}-z_t^{(-)})}{2b}\)\\[1mm]
\(\mathbb E[|u_t^\star|]=\dfrac{a}{b}x_t\)};

\node[regionlabel, blue!70!black] at (4.35,1.30)
{\(u_t^\star=\dfrac{a}{b}x_t\)};

\draw[blue!70!black, line width=1.1pt] (\L,\Ytop+0.12) -- (\A,\Ytop+0.12);
\draw[red!70!black, line width=1.1pt] (\A,\Ytop+0.12) -- (\B,\Ytop+0.12);
\draw[blue!70!black, line width=1.1pt] (\B,\Ytop+0.12) -- (\R,\Ytop+0.12);

\node at (0,-1.5) {\textbf{(b)} Optimal dual control policy for $\gamma\ge\gamma_\star$, excluding
$\gamma=\tfrac{1}{1-a}$ when $0\leq a<\tfrac12$.};

\end{scope}

\end{tikzpicture}
\caption{Summary of the optimal gain, cost, and dual policy across the different regimes of \(a\). \textbf{(a)} The gain and cost regimes are shown as functions of \(a\). \textbf{(b)} For \(\gamma\ge\gamma_\star(a)\), excluding \(\gamma=\tfrac{1}{1-a}\) when \(0\le a<\tfrac12\), the optimal dual policy is determined by the history difference \(z_t^{(+)}-z_t^{(-)}\). Outside the central region, the policy acts by certainty equivalence; inside it explores by randomizing over the two extreme admissible inputs. At \(\gamma=\tfrac{1}{1-a}\), if the initial history variables satisfy \(z_0^{(+)}=z_0^{(-)}=0\), then the zero controller is optimal.}
\label{fig:gain-cost-policy-summary}
\end{figure}
\subsection{Realization of the optimal dual policy}\label{Rem1}
The following lemma shows that the dual policy~\eqref{eq:ustar-cases} in the exploration regime can be realized by randomizing between the two extreme
admissible controls.
\begin{lemma}[Realization of the randomized policy]
Assume
$
\bigl|z^{(+)}-z^{(-)}\bigr|
\le
\tfrac{2a}{\gamma}x.
$
Define
$
u^+ := \tfrac{a}{b}x$ and 
$u^- := -\tfrac{a}{b}x$
and let \(u\) take the values \(u^+\) and \(u^-\) with probabilities
\(\theta\) and \(1-\theta\), respectively, where
$
\theta
=
\tfrac12\left(
1+\tfrac{\gamma\bigl(z^{(+)}-z^{(-)}\bigr)}{2ax}
\right).
$
Then,
\[
\mathbb E[u]
=
\tfrac{\gamma\bigl(z^{(+)}-z^{(-)}\bigr)}{2b},
\qquad
\mathbb E[|u|]
=
\tfrac{a}{b}x
\]
\end{lemma}
{\em Proof.}
Since
$
\bigl|z^{(+)}-z^{(-)}\bigr|
\le
\tfrac{2a}{\gamma}x,
$
we have
$
\bigl|
\tfrac{\gamma\bigl(z^{(+)}-z^{(-)}\bigr)}{2ax}
\bigr|
\le 1,
$
and therefore \(0\le \theta\le 1\) meaning the probabilities are well defined.
Because \(u^\pm=\pm \tfrac{a}{b}x\), it follows that
\[
\mathbb E[u]
=
\theta u^+ + (1-\theta)u^-
=
(2\theta-1)\tfrac{a}{b}x
=
\tfrac{\gamma\bigl(z^{(+)}-z^{(-)}\bigr)}{2b}.
\]
Moreover,
\[
\mathbb E[|u|]
=
\theta |u^+| + (1-\theta)|u^-|
=
\tfrac{a}{b}x.
\]
Thus the randomized control satisfies the required moment conditions. Since
\(|u^\pm|=\tfrac{a}{b}x\), the control is admissible and attains the largest
possible absolute value under the constraint
$
|u|\le \tfrac{a}{b}x.
$
\qedblack

Outside the exploration regime, the optimal policy reduces to the deterministic
certainty-equivalence controls
$
u^\star=-\tfrac{a}{b}x
$ or $
u^\star=\tfrac{a}{b}x,
$
depending on the sign of the difference \(z^{(-)}-z^{(+)}\).

\section{Fundamental limitations of dual control for positive systems}\label{fund-limit}
The work in~\cite{Bencherki2026} derived an upper bound on the gain achieved by the certainty-equivalence (CE) policy, which we denote by $\gamma^{\mathrm{CE}}(a)$. Let $\gamma^{\mathrm{CE}}_\star$ denote the best gain achieved by the CE policy. Then
\[
\gamma^{\mathrm{CE}}_\star(a)
\leq \gamma^{\mathrm{CE}}(a)
=
\max\{1+2a,\,4a\},
\qquad a\geq 0.
\]
On the other hand, the optimal controller specified in section~\eqref{eq:ustar-cases} achieves the optimal gain \(\gamma_\star\), given by
\[
\gamma_\star(a)\coloneqq
\begin{cases}
\tfrac{1}{1-a}, & 0<a< \tfrac12,\\[4pt]
1+2a,           & a\ge \tfrac12.
\end{cases}
\]
In particular, as a sanity check, it holds that
$
\gamma_\star(a)\ \le\ \gamma^{\mathrm{CE}}(a),
$
 which certifies the superiority of the optimal dual controller.
See Figure~\ref{fig:CE} for a comparison of \(\gamma_\star(a)\) and \(\gamma^{\mathrm{CE}}(a)\). 

It should be emphasized that the bound
\(\gamma^{\mathrm{CE}}(a)\) derived in~\cite{Bencherki2026} is based
on a sufficient condition that guarantee a solution to the Bellman inequality,
and therefore need not be tight. In contrast,
\cite{vinnicombe2004examples} shows that, in the quadratic setting, the CE
policy achieves the optimal \(\ell_2\)-gain bound for model sets of the form
\(\mathcal M=\{(a,1),(a,-1)\}\). This naturally raises the question whether the CE policy also
achieves the optimal \(\ell_1\)-gain bound in the present setting for positive systems.
To investigate this, we derive the positive system and \(\ell_1\)-gain analogue of the result
reported in~\cite{vinnicombe2004examples} for the case in which the system is
affected by positive disturbances.

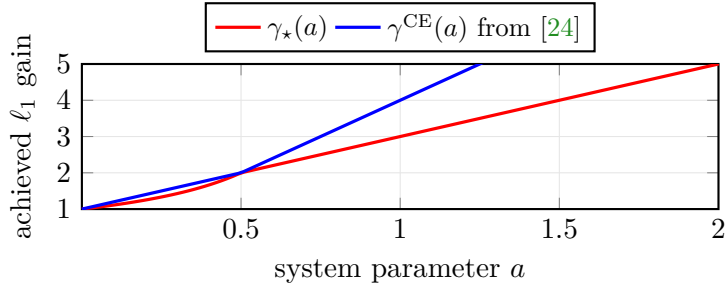
\begin{figure}[t!]
\centering
\begin{tikzpicture}
\begin{axis}[
  width=10cm,
  height=3.5cm,
  xlabel={system parameter $a$},
  ylabel={achieved $\ell_1$ gain},
  xmin=0, xmax=2,
  ymin=1, ymax=5,
  xtick={0.5,1,1.5, 2},
  grid=both,
  major grid style={gray!20},
  minor grid style={gray!10},
  legend style={
        at={(0.5,1.08)},
        anchor=south,
        legend columns=2,
        draw=black,
        fill=white,
        font=\small
    },
  legend cell align=left,
  thick,
  samples=400,
  unbounded coords=jump
]

\addplot[red, very thick, domain=0:7]
  ({x},{ (x < 0.5) ? (1/(1 - x)) : (1 + 2*x) });
\addlegendentry{$\gamma_\star(a)$}

\addplot[blue, very thick, domain=0:7]
  ({x},{ (x < 0.5) ? (1 + 2*x) : (4*x) });
\addlegendentry{$ \gamma^{\mathrm{CE}}(a)$ from~\cite{Bencherki2026}}

\end{axis}
\end{tikzpicture}
\caption{Comparison of the optimal gain \(\gamma_\star(a)\) (red) with the certainty--equivalence upper bound on achievable gain \(\gamma^{\mathrm{CE}}(a)\) (blue). For \(0\leq a\leq \tfrac12\), \(\gamma_\star(a)=\tfrac{1}{1-a}\) and \(\gamma^{\mathrm{CE}}(a)=1+2a\); for \(a>\tfrac12\), \(\gamma_\star(a)=1+2a\) while \(\gamma^{\mathrm{CE}}(a)=4a\).}
\label{fig:CE}
\end{figure}

\begin{theorem}\label{thm3}
Consider the scalar positive system
\[
x_{t+1}=a x_t+b u_t+w_t,\qquad b\in\{+1,-1\},
\]
where \(a\ge0\), \(x_0\geq 0\), and \(w_t\geq 0\) for all \(t\geq 0\). Let the
model-specific histories be updated according to
\[
z_{t+1}^{(\sigma)}
=
z_t^{(\sigma)}
+
\left|x_{t+1}-a x_t-\sigma u_t\right|,
\qquad
z_0^{(+)}=z_0^{(-)}=0,
\qquad
\sigma\in\{+1,-1\}.
\]
Let the controller be
\[
u_t =
\begin{cases}
0, & 0\leq a<\dfrac12,\\[6pt]
-k_t a x_t, & a\geq \dfrac12,
\end{cases}
\qquad
k_t\in\arg\min_{\sigma\in\{+1,-1\}} z_t^{(\sigma)} .
\]
Then, for every \(b\in\{+1,-1\}\) and every \(w\in\ell_1\), the closed-loop
trajectory satisfies
\[
\sum_{t=0}^{\infty} x_t
\le
\gamma_\star(a)\,x_0
+\gamma_\star(a)\sum_{t=0}^{\infty} w_t,
\qquad
\gamma_\star(a)\coloneqq
\begin{cases}
\tfrac{1}{1-a}, & 0<a< \tfrac12,\\[4pt]
1+2a,           & a\ge \tfrac12.
\end{cases}
\]
\end{theorem}
{\em Proof.}
    See Appendix~\ref{thm3-proof} \qedblack
\begin{remark}
The controller in Theorem~\ref{thm3} is not, strictly speaking, a
certainty-equivalence (CE) controller. The
controller is piecewise defined. Indeed, for \(0 \leq a<\tfrac12\), the
optimal policy is the zero input policy \(u_t=0\), whereas for
\(a\ge\tfrac12\) it reduces to a CE controller.
\end{remark}
The proof above applies to positive disturbances. It remains unclear
whether the same gain bound can be achieved by the CE policy when the admissible
disturbance set is enlarged to include negative yet positivity preserving disturbances. We defer this to future work. 
Note also that, even if the optimal gain bound is achieved by the CE policy for
the extended admissible disturbance set allowing positivity preserving negative
disturbances, CE controllers still suffer from performance limitations.
\begin{enumerate}
    \item[\textit{(i)}] The result derived in Theorem~\ref{thm3} only implies
    optimality in the induced \(\ell_1\)-gain sense. It does not imply that the
    CE policy is cost optimal. Indeed, Theorem~\ref{thm3} certifies that the CE
    policy achieves the cost bound $\gamma_\star(a)x_0$
    whereas the dual policy achieves the smaller bound
    $(1+a)x_0$
    for $\gamma\ge \gamma_\star(a)$. See Figure~\ref{fig:ce-vs-dual-cost} for a comparison of the achieved cost.
\item[\textit{(ii)}] As in~\cite{vinnicombe2004examples}, one may expect
performance degradation, in terms of the achieved $\ell_1$-gain or cost, for
other model sets $\mathcal{M}$. One such example was presented
in~\cite{vinnicombe2004examples} for the model set
$
\mathcal{M}=\{(-a,1),(a,1)\},
$
where the CE policy was shown to be suboptimal with respect to the
$\ell_2$-gain. Similar suboptimality may also arise in the present
setting.
\end{enumerate}
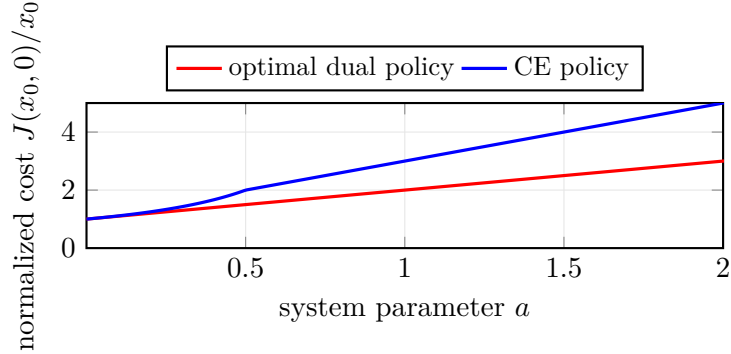
\begin{figure}[h!]
\centering
\begin{tikzpicture}
\begin{axis}[
  width=10cm,
  height=3.5cm,
  xlabel={system parameter \(a\)},
  ylabel={normalized cost \(J(x_0,0)/x_0\)},
  xmin=0, xmax=2,
  ymin=0, ymax=5,
  grid=both,
  major grid style={gray!20},
  minor grid style={gray!10},
  legend style={
        at={(0.5,1.08)},
        anchor=south,
        legend columns=2,
        draw=black,
        fill=white,
        font=\small
    },
    xtick={0.5,1,1.5, 2},
  legend cell align=left,
  thick,
  samples=400,
  unbounded coords=jump
]

\addplot[red, very thick, domain=0:3]
  ({x},{1 + x});
\addlegendentry{optimal dual policy}

\addplot[blue, very thick, domain=0:3]
  ({x},{ (x < 0.5) ? (1/(1 - x)) : (1 + 2*x) });
\addlegendentry{CE policy}


\end{axis}
\end{tikzpicture}

\caption{Comparison of the normalized cost achieved by the CE policy and by the
optimal dual policy. The optimal dual policy chooses \(\gamma\) so as to minimize
the cost, yielding \(J_\star(x_0,0)=(1+a)x_0\). In contrast, the CE policy
achieves the cost bound \(\tfrac{1}{1-a}x_0\) for \(0\le a<\tfrac12\), and
\((1+2a)x_0\) for \(a\ge\tfrac12\). Thus, even when the CE policy achieves the
optimal induced \(\ell_1\)-gain, it need not be cost optimal.}
\label{fig:ce-vs-dual-cost}
\end{figure}

\section{Conclusion and future work}\label{concl}

This paper studied minimax dual control for positive systems with uncertain dynamics. We reformulated the problem as a two player zero-sum dynamic game in which the controller may use randomized inputs. This randomization convexifies the control optimization in the Bellman recursion, making it possible to characterize the exact solution using minimax dynamic programming.

For a scalar positive system with an unknown input sign, we solved the Bellman equation exactly and derived an explicit minimax optimal control policy together with the corresponding optimal $\ell_1$-gain. The solution reveals a clear dual control structure. In certain regions of the hyperstate space, the controller randomizes its input to acquire information about the unknown input direction. Outside these regions, further exploration is not beneficial, and the optimal controller reduces to a deterministic certainty equivalence policy.

The results also show how positivity, the admissible control inputs, and model uncertainty influence the structure of the optimal policy and the achievable closed-loop performance. In particular, the exact solution identifies distinct exploration and certainty-equivalence regimes and characterizes the best achievable worst-case performance in the considered setting. It thereby provides insight into the fundamental limitations of minimax dual control for positive systems. Although this paper focuses on the finite model set $\mathcal M=\{(a,b),(a,-b)\}$, the current results offer insight into more general dual control problems for positive systems. Our treatment is closely related to the exact dual control formulations developed for the quadratic setting under sign uncertainty in~\cite{Rantzer2025a}, and later extended to continuum model classes in~\cite{Rantzer2026}. The scalar positive system problem studied here may therefore be viewed as a natural analogue of that line of work and as a representative case for richer positive system uncertainty classes. We leave such extensions for future work.

Several directions remain for future research:
\begin{itemize}
	\item[\textit{(1)}] Extend the analysis to vector-valued positive systems, where the scalar structure exploited in this paper is no longer available and the associated Bellman recursion is expected to be substantially more complex.
	
	\item[\textit{(2)}] Consider richer uncertainty models.
	
	\item[\textit{(3)}] Study output feedback and stochastic formulations in which the controller must account for partial observations, process disturbances, measurement noise, or combinations of these.
	
	\item[\textit{(4)}] Develop scalable approximate dual controllers for settings in which an exact solution to the Bellman equation is difficult to obtain, while retaining computable worst-case performance guarantees.

\end{itemize}

\section*{Acknowledgments}
The authors are affiliated with the ELLIIT Strategic Research Area at Lund University. This project received funding from the European Research Council (ERC) under grant agreements No.~834142 (ScalableControl) and No.~101199738 (DualControl), and was partially supported by the Wallenberg AI, Autonomous Systems and Software Program (WASP), funded by the Knut and Alice Wallenberg Foundation.  
\begin{appendix}

\section{Minimax Dynamic Programming}
\label{app:minimax-dp}

For completeness, we summarize the basic minimax dynamic programming framework used
throughout the paper. Further details can be found in~\cite{bacsar2008h}.
Consider the discrete-time system
\[
y_{t+1}=f(y_t,u_t,v_t),
\]
where \(u_t\) is chosen by the minimizing player and \(v_t\) by the maximizing
player. The minimizing player employs a causal policy
$
\mu=(\mu_0,\mu_1,\ldots)$,
$
u_t=\mu_t(y_0,\ldots,y_t),
$
so that the control action at time \(t\) depends only on the state measurements
available up to that time.
Given an initial state \(y_0\), the finite horizon value function is defined by
\begin{equation}
\label{eq:minimax-finite-horizon}
V^{(N)}(y_0)
\coloneqq
\inf_{\mu}\sup_{v}
\left\{
\sum_{t=0}^{N-1}g(y_t,u_t,v_t)+g_0(y_N)
\right\},
\end{equation}
where the state and control trajectories are generated by
\[
y_{t+1}=f(y_t,u_t,v_t),
\qquad
u_t=\mu_t(y_0,\ldots,y_t).
\]
For any function \(V\), define the Bellman operator
\begin{equation}
\label{eq:minimax-bellman-operator}
\mathcal F V(y)
\coloneqq
\min_u\max_v
\left\{
g(y,u,v)+V\bigl(f(y,u,v)\bigr)
\right\}.
\end{equation}
Whenever the extrema are attained, the infimum and supremum may be replaced by
a minimum and a maximum, respectively.

\begin{proposition}[Minimax dynamic programming]
\label{prop:minimax-dp}
The value function \(V^{(N)}\) in~\eqref{eq:minimax-finite-horizon} can be computed recursively according to
$
V^{(0)}(y)=g_0(y)
$
and
\begin{equation}
\label{eq:minimax-dp-recursion}
V^{(k+1)}(y)
=
\min_u\max_v
\left[
g(y,u,v)
+
V^{(k)}\bigl(f(y,u,v)\bigr)
\right],
\qquad k\geq 0.
\end{equation}
Moreover, if
$
\max_v g(y,u,v)\geq 0,
$ for all $(y,u)$
then
\begin{equation}
\label{eq:minimax-dp-monotonicity}
g_0
=
V^{(0)}
\leq
V^{(1)}
\leq
\cdots
\leq
V^{(N)}.
\end{equation}
Furthermore, \(V^{(N)}\) is upper bounded by every function
\(\overbar V\) satisfying
\begin{equation}
\label{eq:minimax-dp-upper-bound}
g_0(y)
\leq
\inf_u\sup_v
\left[
g(y,u,v)
+
\overbar V\bigl(f(y,u,v)\bigr)
\right]
\leq
\overbar V(y),\quad \text{for all} \, y
\end{equation}
\end{proposition}
{\em Proof.}
We first establish the recursive formula by induction on the horizon length.
For \(k=0\), there are no stage costs, and hence, by definition,
$
V^{(0)}(y)=g_0(y).
$
Assume that the recursion holds for some \(k\geq 0\). The value of the
\((k+1)\)-stage problem is
\begin{align*}
V^{(k+1)}(y_0)
&=
\inf_{\mu}\sup_{v}
\left[
\sum_{t=0}^{k}g(y_t,u_t,v_t)+g_0(y_{k+1})
\right]
\\
&=
\inf_{u_0}\sup_{v_0}
\Biggl[
g(y_0,u_0,v_0)
+
\inf_{\mu}\sup_{v}
\left(
\sum_{t=1}^{k}g(y_t,u_t,v_t)+g_0(y_{k+1})
\right)
\Biggr],
\end{align*}
Since
$
y_1=f(y_0,u_0,v_0),
$
the continuation term is precisely the value of the \(k\)-stage problem
initialized at \(y_1\). Therefore,
\begin{align*}
V^{(k+1)}(y_0)
=
\inf_{u_0}\sup_{v_0}
\left[
g(y_0,u_0,v_0)
+
V^{(k)}(y_1)
\right]
=
\inf_{u_0}\sup_{v_0}
\left[
g(y_0,u_0,v_0)
+
V^{(k)}\bigl(f(y_0,u_0,v_0)\bigr)
\right].
\end{align*}
When the extrema are attained, the infimum and supremum may be replaced by a
minimum and a maximum, respectively. This proves
\eqref{eq:minimax-dp-recursion}.
Next, define the Bellman operator
\[
\mathcal F V(y)
\coloneqq
\inf_u\sup_v
\left[
g(y,u,v)+V\bigl(f(y,u,v)\bigr)
\right].
\]
This operator is order preserving: if \(V\leq W\), then
$
\mathcal F V\leq\mathcal F W.
$
Indeed, for every \((y,u,v)\),
\[
g(y,u,v)+V\bigl(f(y,u,v)\bigr)
\leq
g(y,u,v)+W\bigl(f(y,u,v)\bigr),
\]
and the inequality is preserved by taking the supremum over \(v\) and the
infimum over \(u\).

The assumption
$
\max_v g(y,u,v)\geq0
$
gives
$
V^{(0)}\leq V^{(1)}.
$
Suppose that \(V^{(k-1)}\leq V^{(k)}\). By the order-preserving property of
\(\mathcal F\),
\[
V^{(k)}
=
\mathcal F V^{(k-1)}
\leq
\mathcal F V^{(k)}
=
V^{(k+1)}.
\]
Thus, induction yields
\[
V^{(0)}
\leq
V^{(1)}
\leq
\cdots
\leq
V^{(N)},
\]
which proves~\eqref{eq:minimax-dp-monotonicity}.
Finally, suppose that \(\overbar V\) satisfies
\eqref{eq:minimax-dp-upper-bound}. Then
$
V^{(0)}=g_0\leq\overbar V.
$
Assume that \(V^{(k)}\leq\overbar V\). Using the monotonicity of the Bellman
operator and~\eqref{eq:minimax-dp-upper-bound}, we obtain
\[
V^{(k+1)}
=
\mathcal F V^{(k)}
\leq
\mathcal F\overbar V
\leq
\overbar V.
\]
Consequently,
$
V^{(N)}\leq\overbar V $
for every $N \ge 0$.
\qedblack
\section{Auxiliary Lemmata}\label{appenA} 
\begin{lemma}\label{lem:sup-exp}
	Let $c,c_1,c_2$ be nonnegative random variables. Then the following statements hold:
	\begin{enumerate}
		\item[\textit{(L1)}]
		If $\lambda,\gamma\ge0$, then
		\[
		\max_{v\ge0}\,\mathbb{E}\bigl(\lambda v-\gamma|c-v|\bigr)=
		\begin{cases}
			\lambda\,\mathbb E[c], & \gamma\ge \lambda,\\[2pt]
			+\infty, & \gamma<\lambda,
		\end{cases}
		\tag{L1}
		\]
		where the maximum is taken over all nonnegative random variables $v$.
		
		\item[\textit{(L2)}]
		If $c_1,c_2\ge0$, $\alpha>0$, and $\lambda\ge0$, then
		\[
		\max_{v\ge0}\,\mathbb{E}\bigl(\lambda v-\alpha(|c_1-v|+|c_2-v|)\bigr)
		=
		\begin{cases}
			\displaystyle \lambda\,\mathbb{E}\big[\max\{c_1,c_2\}\big]
			-\alpha\,\mathbb{E}|c_1-c_2|, & 0\le \lambda \le 2\alpha,\\[6pt]
			\displaystyle +\infty, & \lambda>2\alpha,
		\end{cases}
		\tag{L2}
		\]
		where the maximum is taken over all nonnegative random variables $v$.
	\end{enumerate}
\end{lemma}
{\em Proof.}
Throughout, the maximum over $v\ge0$ is taken over all nonnegative random variables $v$. For each realization $\omega$ of the underlying randomness, the objective becomes an ordinary scalar expression; we optimize that expression and then take expectations.

\medskip
\noindent\emph{Proof of \textit{(L1)}.}
Fix a realization \(c=c(\omega)\ge 0\). For this fixed \(c\), the maximization
$
\max_{v\ge 0}\bigl(\lambda v-\gamma|c-v|\bigr)
$
is exactly as in~\cite{Bencherki2026}: it is finite if and only if \(\gamma\ge\lambda\), in which case it equals \(\lambda c\); if \(\gamma<\lambda\), it is \(+\infty\).

The same argument applies almost surely in the stochastic setting. Thus, if \(\gamma\ge\lambda\), choosing \(v=c\) gives
\[
\max_{v\ge0}\,\mathbb E\bigl[\lambda v-\gamma|c-v|\bigr]=\lambda\,\mathbb E[c].
\]
If \(\gamma<\lambda\), taking \(v=c+n\) yields
\[
\mathbb E\bigl[\lambda v-\gamma|c-v|\bigr]=\lambda\,\mathbb E[c]+(\lambda-\gamma)n,
\]
which diverges to \(+\infty\) as \(n\to\infty\). Hence,
\[
\max_{v\ge0}\,\mathbb E\bigl[\lambda v-\gamma|c-v|\bigr]
=
\begin{cases}
\lambda\,\mathbb E[c], & \gamma\ge \lambda,\\[2pt]
+\infty, & \gamma< \lambda,
\end{cases}
\]
which proves \textit{(L1)}.

\noindent\emph{Proof of \textit{(L2)}.}
Fix a realization \(c_1=c_1(\omega)\ge 0\) and \(c_2=c_2(\omega)\ge 0\). For these fixed values, the maximization
$
\max_{v\ge 0}\,\Bigl(\lambda v-\alpha\bigl(|c_1-v|+|c_2-v|\bigr)\Bigr)
$
is exactly as in~\cite{Bencherki2026}: it is finite if and only if \(\lambda\le 2\alpha\), in which case it equals \(\lambda\max\{c_1,c_2\}-\alpha|c_1-c_2|\); if \(\lambda>2\alpha\), it is \(+\infty\).

The same argument applies for almost every realization in the stochastic setting. Therefore, if \(\lambda\le 2\alpha\), choosing \(v=\max\{c_1,c_2\}\) gives
\[
\max_{v\ge0}\,\mathbb E\bigl[\lambda v-\alpha(|c_1-v|+|c_2-v|)\bigr]
=
\lambda\,\mathbb E\big[\max\{c_1,c_2\}\big]
-\alpha\,\mathbb E|c_1-c_2|.
\]
If \(\lambda>2\alpha\), then taking \(v=\max\{c_1,c_2\}+n\) yields
\[
\mathbb E\bigl[\lambda v-\alpha(|c_1-v|+|c_2-v|)\bigr]
=
\lambda\,\mathbb E\big[\max\{c_1,c_2\}\big]
+(\lambda-2\alpha)n,
\]
which diverges to \(+\infty\) as \(n\to\infty\). Hence,
\[
\max_{v\ge0}\,\mathbb E\bigl(\lambda v-\alpha(|c_1-v|+|c_2-v|)\bigr)
=
\begin{cases}
\displaystyle \lambda\,\mathbb E\big[\max\{c_1,c_2\}\big]
-\alpha\,\mathbb E|c_1-c_2|, & 0\le \lambda \le 2\alpha,\\[6pt]
\displaystyle +\infty, & \lambda>2\alpha,
\end{cases}
\]
which proves \textup{(L2)}.
\qedblack
\begin{lemma}\label{lemm1}
Let \(\gamma\geq1+a\), and consider the minimization problem
\[
\min_{\lvert u\rvert\leq\tfrac{a}{b}x}\max\{V_+(u),V_-(u),V_{\mathrm{avg}}(u)\},
\]
where
\[
\begin{aligned}
V_{\mathrm{avg}}(u)
&=
\mathbb E\!\Big[
x+ax+a^2x+(1+a-\gamma)\lvert bu\rvert
-\tfrac{\gamma}{2}\bigl(z^{(+)}+z^{(-)}\bigr)
\Big],\\
V_\pm(u)
&=
\mathbb E\!\Big[
x+ax\pm bu-\gamma z^{(\pm)}
\Big].
\end{aligned}
\]
At the current stage, \((x,z^{(+)},z^{(-)})\) is observed before \(u\) is chosen, so the expectations are taken only over the randomization of the control. If \(\lvert u\rvert\leq\tfrac{a}{b}x\) almost surely, then without loss of optimality the minimization may be restricted to controls satisfying \(\mathbb E[\lvert bu\rvert]=ax\).
\end{lemma}
{\em Proof.}
Fix a realization of the current state \((x,z^{(+)},z^{(-)})\). Since the controller observes this state before choosing \(u\), all expectations below are taken with respect to the policy randomization conditional on the observed state.
From \(|u|\le \tfrac{a}{b}x\) almost surely, it follows that
$
|bu|\le ax
$ with probability $1$
and therefore
$
\mathbb E[|bu|]\le ax.
$
Moreover, by Jensen's inequality,
\[
|b\,\mathbb E[u]|
=
|\mathbb E[bu]|
\le
\mathbb E[|bu|]
\le
ax.
\]
Hence every admissible control satisfies
$|b\,\mathbb E[u]|
\le
\mathbb E[|bu|]
\le
ax.
$
Conversely, fix any \(m\) with \(|m|\le ax\) and define
$
p\coloneqq\tfrac12\left(1+\tfrac{m}{ax}\right).
$
Since \(|m|\le ax\), we have \(p\in[0,1]\). Consider the randomized control
\[
u=
\begin{cases}
\dfrac{a}{b}x, & \text{with probability }p,\\[4pt]
-\dfrac{a}{b}x, & \text{with probability }1-p.
\end{cases}
\]
Both realizations satisfy \(|u|=\tfrac{a}{b}x\), so the control is
admissible almost surely. Its mean satisfies
\[
b\,\mathbb E[u]
=
b\left(
p\tfrac{a}{b}x-(1-p)\tfrac{a}{b}x
\right)
=(2p-1)ax=m,
\]
where the last equality follows from the definition of \(p\). Thus, the
mixing probability \(p\) can be chosen to realize any prescribed mean
\(m\in[-ax,ax]\). At the same time, both realizations have the maximal
admissible magnitude \(|bu|=ax\), and hence
$
\mathbb E[|bu|]=ax.
$
Therefore, the above
randomization realizes any admissible mean \(b\,\mathbb E[u]=m\) while
simultaneously attaining the largest possible value
\(\mathbb E[|bu|]=ax\).
Now observe that \(V_+(u)\) and \(V_-(u)\) depend on \(u\) only through the mean \(\mathbb E[u]\), whereas \(V_{\mathrm{avg}}(u)\) depends on \(u\) only through \(\mathbb E[|bu|]\):
\[
V_{\mathrm{avg}}(u)
=
\mathbb E\!\Big[
x+ax+a^2x
+(1+a-\gamma)|bu|
-\tfrac{\gamma}{2}\bigl(z^{(+)}+z^{(-)}\bigr)
\Big].
\]
Since \(\gamma\ge 1+a\), we have
$
1+a-\gamma\le0.
$
Therefore, for fixed \(\mathbb E[u]\), the quantity \(V_{\mathrm{avg}}(u)\) is nonincreasing in \(\mathbb E[|bu|]\), while \(V_+(u)\) and \(V_-(u)\) remain unchanged.
Consequently, given any admissible control \(u\), we may replace it by another admissible control \(\tilde u\) satisfying
\[
\mathbb E[\tilde u]=\mathbb E[u],
\qquad
\mathbb E[|b\tilde u|]=ax.
\]
Then
$
V_+(\tilde u)=V_+(u),
V_-(\tilde u)=V_-(u),
$
and
$
V_{\mathrm{avg}}(\tilde u)
\le
V_{\mathrm{avg}}(u).
$
Hence
\[
\max\{V_{\mathrm{avg}}(\tilde u),V_+(\tilde u),V_-(\tilde u)\}
\le
\max\{V_{\mathrm{avg}}(u),V_+(u),V_-(u)\}.
\]
It follows that there always exists an optimal control satisfying
$
\mathbb E[|bu|]=ax.
$
\qedblack
\section{Proof of main results}\label{main-res}
\subsection{Proof of Theorem~\ref{thm:l1-dp}}\label{dp-general-case-proof}
Recall that, for each \((A,B)\in\mathcal M\), the model-specific history
\(z^{(A,B)}\in\mathbb R_+^n\) evolves according to
\begin{align}\label{hist-var}
z^{(A,B)}_+
&=
z^{(A,B)}
+
\lvert v-Ax-Bu\rvert,
\qquad
x_+=v,\quad v\geq 0.
\end{align}
We collect all model-specific histories in the data collection \(Z\). All
expectations below are taken with respect to the randomization of the
control policy.
Define the dynamic programming operators
\[
\mathcal F_u V(x,Z)
\coloneqq
\max_{v\geq0}\left\{
\mathbb E\!\left[s^\top x+r^\top u\right]+V(v,Z_+)
\right\}
,
\qquad
\mathcal F V(x,Z)
\coloneqq
\min_{u \in \mathcal U(x)}
\mathcal F_u V(x,Z).
\]
The value iteration is initialized and updated according to
\begin{align}\label{v0vk}
V_0(x,Z)
&\coloneqq
-\min_{(A,B)\in\mathcal M}
\mathbb E\!\left[\gamma^\top z^{(A,B)}\right],
\qquad
V_{k+1}
=
\mathcal F V_k,
\quad k\geq 0.
\end{align}

\noindent\emph{Proof of \textit{(i)}--\textit{(ii)}.}
We first show that the sequence \(V_0,V_1,V_2,\ldots\) is monotonically
nondecreasing. By definition,
\begin{multline*}
V_1(x,Z)
=
\min_{u  \in \mathcal U(x)}\max_{v\geq0}\left\{
\mathbb E\!\left[s^\top x+r^\top u\right]+V_0(v,Z_+)
\right\}
=
\min_{u  \in \mathcal U(x)}
\max_{v\geq 0}
\Biggl\{
\mathbb E\!\left[s^\top x+r^\top u\right]
\\
-
\min_{(A,B)\in\mathcal M}
\mathbb E\!\left[
\gamma^\top
\left(
z^{(A,B)}
+
\lvert v-Ax-Bu\rvert
\right)
\right]
\Biggr\}.
\end{multline*}
Assumption~\ref{ass:stage cost-pos} reads
$
\mathbb E\!\left[s^\top x+r^\top u\right]\geq 0.
$
Therefore, 
\begin{multline*}
\mathbb E\!\left[s^\top x+r^\top u\right]
-
\min_{(A,B)\in\mathcal M}
\mathbb E\!\left[
\gamma^\top
\left(
z^{(A,B)}+\lvert v-Ax-Bu\rvert
\right)
\right]
\\
\geq
-
\min_{(A,B)\in\mathcal M}
\mathbb E\!\left[
\gamma^\top
\left(
z^{(A,B)}+\lvert v-Ax-Bu\rvert
\right)
\right].
\end{multline*}
Because the adversary may choose \(v\) as a causal response to the realized
control input, the choice
$
v=A x+B u
$
is admissible and satisfies \(v\geq 0\) almost surely since $u \in \mathcal{U}(x)$. For this choice,
\begin{multline*}
V_1(x,Z)
\ge
-\max_{u\in\mathcal U(x)}
\min_{(A,B)\in\mathcal M}
\min_{v\ge 0}
\mathbb E\!\left[
\gamma^\top
\left(
z^{(A,B)}+\lvert v-Ax-Bu\rvert
\right)
\right]
\\
=
-\min_{(A,B)\in\mathcal M}
\mathbb E\!\left[\gamma^\top z^{(A,B)}\right]
=
V_0(x,Z).
\end{multline*}
Consequently,
$
V_1(x,Z)\geq V_0(x,Z).
$
Since \(\mathcal F\) is order preserving, induction gives
\[
V_{k+1}
=
\mathcal F V_k
\geq
\mathcal F V_{k-1}
=
V_k,
\qquad k\geq 1.
\]
Thus, \(V_0,V_1,V_2,\ldots\) is monotonically nondecreasing.

Fix \(N\geq 0\). Consider the \(N\)-stage truncated version of the original
minimax problem:
\begin{equation}\label{eq:FH-lb}
	\begin{aligned}
		\inf_{\mu}\;
		\sup_{\substack{(A,B)\in\mathcal M\\
				w\geq-(Ax+Bu)}}
		&\;
		\mathbb E\!\left[
		\sum_{t=0}^{N}
		s^\top x_t+r^\top u_t-\gamma^\top|w_t|
		\right]
		\\
		\text{s.t.}\quad
		&
		x_{t+1}
		=
		Ax_t+Bu_t+w_t,
		\qquad
		x_0\in\mathbb R_+^n.
	\end{aligned}
\end{equation}
The value of Problem~\ref{prob:l1-robust-control} is bounded below by the
value of~\eqref{eq:FH-lb}. Moreover, the value of~\eqref{eq:FH-lb} is
nondecreasing in \(N\), and the limit equals the value of
Problem~\ref{prob:l1-robust-control}.
Introducing the change of variables
$
v_t\coloneqq x_{t+1}
$
and using the model-specific histories in~\eqref{hist-var} renders
\eqref{eq:FH-lb} equivalent to
\begin{equation}\label{eq:FH-change}
\begin{aligned}
\inf_{\eta}\;
\sup_{v\geq 0}
\Biggl\{
&
\mathbb E\!\left[
\sum_{t=0}^{N}
s^\top x_t+r^\top u_t
\right]
-
\min_{(A,B)\in\mathcal M}
\mathbb E\!\left[
\gamma^\top z_{N+1}^{(A,B)}
\right]
\Biggr\}
\\
\text{s.t.}\quad
&
x_{t+1}=v_t,
\qquad
x_0\in\mathbb R_+^n,
\\
&
z_{t+1}^{(A,B)}
=
z_t^{(A,B)}
+
\lvert v_t-Ax_t-Bu_t\rvert,\quad
z_0^{(A,B)}=0,
\qquad
(A,B)\in\mathcal M,
\\
&
Z_t
\coloneqq
\left\{
z_t^{(A,B)}
\right\}_{(A,B)\in\mathcal M},
\qquad
Z_0
\coloneqq
\left\{
0
\right\}_{(A,B)\in\mathcal M},
\\
&
u_t
=
\eta_t(x_t,Z_t),
\qquad
u_t\in\mathcal U(x_t)
\quad\text{almost surely},\,
t=0,\ldots,N.
\end{aligned}
\end{equation}
A standard minimax dynamic programming argument, with terminal cost and
recursion given in~\eqref{v0vk}, shows that the value of
\eqref{eq:FH-change} is
$
V_{N+1}(x_0,0).
$
Thus, Problem~\ref{prob:l1-robust-control} has a finite value \(J_\star(x_0)\)
if and only if the nondecreasing sequence
$
\left\{
V_k(x_0,0)
\right\}_{k\geq 0}
$
is bounded above. In that case,
$
V_\star(x_0,0)
=
\lim_{N\to\infty}V_N(x_0,0)
=
J_\star(x_0).
$
Moreover, \(V_k(x,Z)\) is nonincreasing with respect to each component of the
history collection \(Z\). This property holds for \(V_0\) and is preserved by
the operator \(\mathcal F\), since each history is updated by adding a
nonnegative vector. Therefore,
$
V_k(x,Z)\leq V_k(x,0)
$
for every nonnegative history collection \(Z\). Hence, if \(V_\star(x,0)\) is
finite, then \(V_\star(x,Z)\) exists and is finite for every nonnegative \(Z\);
indeed,
$
V_\star(x,Z)
=
\lim_{k\to\infty}V_k(x,Z)
\leq
V_\star(x,0).
$

If Problem~\ref{prob:reformulated} has a finite value, then
\(V_k(x_0,0)\) is bounded above by that value for every \(k\). Therefore,
$
V_\star(x_0,0)
=
\lim_{k\to\infty}V_k(x_0,0)
$
is finite. This proves items~\textit{(i)}--\textit{(ii)}.

\medskip
\noindent\emph{Proof of \textit{(iii)}.}
Suppose there exist a finite-valued function \(\overbar V\) and a randomized
policy \(\overbar\eta\) such that
$
V_0\leq\overbar V
$ and $
\mathcal F_{\overbar\eta}\overbar V
\leq
\overbar V.
$
Define the sequence \(W_0,W_1,W_2,\ldots\) recursively by
\[
W_0\coloneqq V_0,
\qquad
W_{k+1}
\coloneqq
\mathcal F_{\overbar\eta}W_k.
\]
By dynamic programming,
\begin{align*}
W_{N+1}(x,0)
=
\sup_{v\geq 0}
\Biggl\{
\mathbb E\!\left[
\sum_{t=0}^{N}
\left(
s^\top x_t+r^\top u_t
\right)
\right]
-
\min_{(A,B)\in\mathcal M}
\mathbb E\!\left[
\gamma^\top z_{N+1}^{(A,B)}
\right]
\Biggr\},
\end{align*}
where
$
x_{t+1}=v_t$,
$
u_t=\overbar\eta_t(x_t,Z_t),
$
and \(Z_t\) is generated by the model-specific history update
in~\eqref{hist-var}. Thus, the value achieved by the policy
\(\overbar\mu\) induced by \(\overbar\eta\) is
$
J_{\overbar\mu}(x_0)
=
\lim_{N\to\infty}W_N(x_0,0)
$ and Problem~\eqref{prob:reformulated} is bounded above by \(\lim_{k\to\infty} W_k(x_0,0)\).

The definitions of \(V_k\) and \(W_k\), together with the order preserving
property of the Bellman operators, yield
$
V_k\leq W_k\leq\overbar V,
$  $k\geq 0$.
Indeed, the first inequality follows from
$
\mathcal F V
\leq
\mathcal F_{\overbar\eta}V,
$
whereas the second follows from \(W_0=V_0\leq\overbar V\) and
$
W_{k+1}
=
\mathcal F_{\overbar\eta}W_k
\leq
\mathcal F_{\overbar\eta}\overbar V
\leq
\overbar V.$
Taking limits gives
\[
V_\star(x,Z)
\leq
\lim_{k\to\infty}W_k(x,Z)
\leq
\overbar V(x,Z).
\]
In particular,
$
V_\star(x_0,0)
\leq
J_{\overbar\mu}(x_0)
\leq
\overbar V(x_0,0).
$
Hence the policy induced by 
\begin{align*}
    \overbar\mu_t(x_0,\ldots,x_t;u_0,\ldots,u_{t-1})
\coloneqq
\overbar\eta\!\left(
x_t,\left\{
\sum_{\tau=0}^{t-1}
\lvert x_{\tau+1}-Ax_\tau-Bu_\tau\rvert
\right\}_{(A,B)\in\mathcal M}
\right)
\end{align*}
 satisfies the stated performance
bound. If \(\overbar V=V_\star\), then the induced policy is optimal.

\qedblack

\subsection{Proof of Theorem~\ref{thm:bellman-termes}}\label{thm:bellman-branches-append}
We proceed in four steps, corresponding to parts \textit{(i)}--\textit{(iv)} of
Theorem~\ref{thm:bellman-termes}. Let \(a\ge\tfrac{1}{2}\) and, without loss
of generality, \(b>0\). Recall that the Bellman operator is given by
\[
\begin{aligned}
\mathcal{F}V(x,z^{(+)},z^{(-)})
=
\min_{|u|\leq \tfrac{a}{b}x}
\ \max_{v\geq 0}\,
\Bigg(
\mathbb E[x]
+
V\!\left(
v,\,
z^{(+)}+\left|v-ax-bu\right|,
z^{(-)}+\left|v-ax+bu\right|
\right)
\Bigg)
\end{aligned}
\]
and that the value iteration sequence is initialized and updated according to~\eqref{vk}. For completeness, we show that starting from the terminal cost in~\eqref{eq:scalar-sign-minimax-min-history}, given by
$
 J_T\coloneqq -\gamma\min\{z^{(+)},z^{(-)}\},
$
one Bellman update gives the model-based cost
$
V_0(x,z^{(+)},z^{(-)})
\coloneqq
\max\Big\{
\mathbb E[x-\gamma z^{(+)}],\;
\mathbb E[x-\gamma z^{(-)}]
\Big\}.
$
First, rewrite the terminal cost as
$
J_T=\max\big\{-\gamma z^{(+)},-\gamma z^{(-)}\big\}.
$
Applying the Bellman operator gives
\begin{multline*}
\mathcal F J_T(x,z^{(+)},z^{(-)})
=
\min_{|u|\le \tfrac{a}{b}x}\max_{v\ge0}
\max\Big\{
\underbrace{\mathbb E\!\left[x-\gamma\bigl(z^{(+)}+|v-ax-bu|\bigr)\right]}_{\Phi_+(u,v)},\\
\underbrace{\mathbb E\!\left[x-\gamma\bigl(z^{(-)}+|v-ax+bu|\bigr)\right]}_{\Phi_-(u,v)}
\Big\}.
\end{multline*}
For a fixed \(u\), we now maximize each term over \(v\ge0\). By Lemma~\ref{lem:sup-exp}\textit{(L1)}, the maximizer is
$
v_\pm^\star=ax\pm bu,
$
which is admissible because \(|u|\le \tfrac{a}{b}x\) implies \(ax\pm bu\ge0\) almost surely. Therefore,
\[
\max_{v\ge0}\Phi_\pm(u,v)
=
\mathbb E\!\left[
x-\gamma\bigl(z^{(\pm)}+|v_\pm^\star-ax\mp bu|\bigr)
\right]
=
\mathbb E[x-\gamma z^{(\pm)}].
\]
Hence the \(v\)-maximization removes the absolute value term completely, and the result no longer depends on \(u\). Thus
\begin{multline*}
  \mathcal F J_T(x,z^{(+)},z^{(-)})
=V_0(x,z^{(+)},z^{(-)})
\Big\}=
\min_{|u|\le \tfrac{a}{b}x}
\max\Big\{
\mathbb E[x-\gamma z^{(+)}],\;
\mathbb E[x-\gamma z^{(-)}]\\= \max\Big\{
\mathbb E[x-\gamma z^{(+)}],\;
\mathbb E[x-\gamma z^{(-)}]
\Big\}
\end{multline*}
Therefore, starting from the terminal cost \(J_T=-\gamma\min\{z^{(+)},z^{(-)}\}\), one Bellman step indeed yields the model-based cost \(V_0\) in~\eqref{vk}.

\medskip
\noindent\emph{Proof of~Theorem~\ref{thm:bellman-termes}(i).} We prove this by following an induction argument. To this end, we start by computing $V_1=\mathcal{F}V_0$. For any admissible control \(u\), it holds that
\begin{multline*}
\mathcal F V_0(x,z^{(+)},z^{(-)})
=
\min_{|u|\le \tfrac{a}{b}x}\ \max_{v\ge 0}
\max\Big(
\mathbb E\!\left[
x+v-\gamma\bigl(z^{(+)}+|v-a x-b u|\bigr)
\right],\\
\mathbb E\!\left[
x+v-\gamma\bigl(z^{(-)}+|v-a x+b u|\bigr)
\right]
\Big)
=
\min_{|u|\le \tfrac{a}{b}x}
\max\Bigg\{
\max_{v\ge0}
\mathbb E\!\left[
x+v-\gamma\bigl(z^{(+)}+|v-a x-b u|\bigr)
\right],\\
\max_{v\ge0}
\mathbb E\!\left[
x+v-\gamma\bigl(z^{(-)}+|v-a x+b u|\bigr)
\right]
\Bigg\}.
\end{multline*}
By Lemma~\ref{lem:sup-exp}\textit{(L1)}, the maximization is finite if and only if \(\gamma\geq1\). The same lemma gives
\[
\max_{v\geq0}\mathbb E\left[v-\gamma|v-ax\mp bu|\right]=\mathbb E[ax\pm bu],
\]
with maximizer \(v_\pm^\star=ax\pm bu\geq0\) almost surely, since \(|u|\leq\tfrac{a}{b}x\) almost surely. Therefore,
\[
\max_{v\geq0}
\mathbb E\!\left[
x+v-\gamma\left(z^{(\pm)}+|v-ax\mp bu|\right)
\right]
=
\mathbb E\!\left[
(1+a)x\pm bu-\gamma z^{(\pm)}
\right].
\]
Consequently, we obtain 
\begin{align*}
\mathcal F V_0(x,z^{(+)},z^{(-)})
=
\min_{|u|\leq \tfrac{a}{b}x}
\max\Bigg\{
\underbrace{
\mathbb E\!\left[(1+a)x+bu-\gamma z^{(+)}\right]
}_{\eqqcolon\,V_+(u)},
\;
\underbrace{
\mathbb E\!\left[(1+a)x-bu-\gamma z^{(-)}\right]
}_{\eqqcolon\,V_-(u)}
\Bigg\}.
\end{align*}
Next, we proceed with the minimization over $u$. To this end, consider the difference 
\begin{multline*}
V_+(u)-V_-(u)
=
\mathbb E\!\Big[(1+a)x+bu-\gamma z^{(+)}\Big]
-
\mathbb E\!\Big[(1+a)x-bu-\gamma z^{(-)}\Big]
\\
=
\mathbb E\!\Big[2bu+\gamma\bigl(z^{(-)}-z^{(+)}\bigr)\Big].
\end{multline*}
We distinguish three regimes depending on the states $(x,z^{(+)},z^{(-)})$.
\begin{itemize}
\item[\textit{(1)}] \emph{Regime 1:
\(\gamma\bigl(z^{(+)}-z^{(-)}\bigr)\leq-2ax\).}
Since \(|bu|\leq ax\) almost surely, we have \(bu\geq-ax\), and therefore
\[
V_+(u)-V_-(u)
=
\mathbb E\!\Big[2bu+\gamma\bigl(z^{(-)}-z^{(+)}\bigr)\Big]
\geq
\mathbb E\!\Big[-2ax+\gamma\bigl(z^{(-)}-z^{(+)}\bigr)\Big]
\geq0.
\]
Hence, \(V_+(u)\geq V_-(u)\), so \(V_-(u)\) never determines the maximum. The
minimizer is therefore obtained by minimizing \(V_+(u)\), which yields
$
u^\star=-\tfrac{a}{b}x.
$
The corresponding value is
$
\mathcal F V_0
=
\mathbb E\!\left[x-\gamma z^{(+)}\right].
$

\item[\textit{(2)}] \emph{Regime 2:
\(\gamma\bigl(z^{(+)}-z^{(-)}\bigr)\geq 2ax\).}
Since \(|bu|\leq ax\) almost surely, we also have \(-bu\geq-ax\), and therefore
\[
V_-(u)-V_+(u)
=
\mathbb E\!\Big[-2bu+\gamma\bigl(z^{(+)}-z^{(-)}\bigr)\Big]
\geq
\mathbb E\!\Big[-2ax+\gamma\bigl(z^{(+)}-z^{(-)}\bigr)\Big]
\geq0.
\]
Hence, \(V_-(u)\geq V_+(u)\), so \(V_+(u)\) never determines the maximum. The
minimizer is therefore obtained by minimizing \(V_-(u)\), which yields
$
u^\star=\tfrac{a}{b}x.
$
The corresponding value is
$
\mathcal F V_0
=
\mathbb E\!\left[x-\gamma z^{(-)}\right].
$
\item[\textit{(3)}] \emph{Regime 3:
\(\bigl|\gamma(z^{(+)}-z^{(-)})\bigr|<2ax\).}
Neither term dominates in this regime. We therefore choose \(u\) so that the
two terms are equal:
\begin{align}\label{v1u}
     V_+(u)=(1+a)x+bu-\gamma z^{(+)}
=
(1+a)x-bu-\gamma z^{(-)}= V_-(u),
\end{align}
which happens at 
$
u^\star
=
\tfrac{\gamma\bigl(z^{(+)}-z^{(-)}\bigr)}{2b}.
$
The regime condition ensures that this choice is admissible, since
$
\bigl|
\tfrac{\gamma\bigl(z^{(+)}-z^{(-)}\bigr)}{2b}
\bigr|
<
\tfrac{a}{b}x.
$
Substituting this choice into~\eqref{v1u} yields the value
$
\mathcal F V_0
=
\mathbb E\!\left[
(1+a)x-\tfrac{\gamma}{2}\bigl(z^{(+)}+z^{(-)}\bigr)
\right].
$
\end{itemize}
Combining the three regimes, we obtain the optimal policy 
\[
\begin{cases}
u^\star=-\tfrac{a}{b}x,
& z^{(-)}-z^{(+)}>\tfrac{2a}{\gamma}x,\\[8pt]
u^\star=\tfrac{\gamma\bigl(z^{(+)}-z^{(-)}\bigr)}{2b},
& \bigl|z^{(+)}-z^{(-)}\bigr|\le \tfrac{2a}{\gamma}x,\\[10pt]
u^\star=\tfrac{a}{b}x,
& z^{(+)}-z^{(-)}>\tfrac{2a}{\gamma}x,
\end{cases}
\]
which results in the following value function
\[
V_1(x,z^{(+)},z^{(-)})
=
\max\Bigg\{
\mathbb E\!\Big[
(1+a)x-\tfrac{\gamma}{2}\bigl(z^{(+)}+z^{(-)}\bigr)
\Big],\;
\mathbb E[x-\gamma z^{(+)}],\;
\mathbb E[x-\gamma z^{(-)}]
\Bigg\}.
\]
$V_1$ is of the form given in Theorem~\ref{thm:bellman-termes}\textit{(i)}. Indeed, since
\(h_0(a,\gamma)=1\), \(h_1(a,\gamma)=1+a\), and the condition
$
\gamma\geq \max\{h_0(a,\gamma)\}=\gamma\geq1
$
is exactly the finiteness condition for \(V_1\). Moreover, when \(k=1\),
the form given in Theorem~\ref{thm:bellman-termes}\textit{(i)} reduces to
\[
V_1(x,z^{(+)},z^{(-)})
=
\max\Bigg\{
\mathbb E[x-\gamma z^{(+)}],\;
\mathbb E[x-\gamma z^{(-)}],\;
\mathbb E\!\Big[
h_1(a,\gamma)x-\tfrac{\gamma}{2}\bigl(z^{(+)}+z^{(-)}\bigr)
\Big]
\Bigg\},
\]
since the last two terms coincide.

We proceed by induction to show that, for every \(k\geq1\),
\begin{multline}\label{eq:Vk-claim1}
V_k(x,z^{(+)},z^{(-)})
=
\max\Bigg\{
\mathbb E\!\left[x-\gamma z^{(+)}\right],\;
\mathbb E\!\left[x-\gamma z^{(-)}\right],\;
\mathbb E\!\Big[
h_1(a,\gamma)x-\tfrac{\gamma}{2}\bigl(z^{(+)}+z^{(-)}\bigr)
\Big],\\
\mathbb E\!\Big[
h_k(a,\gamma)x-\tfrac{\gamma}{2}\bigl(z^{(+)}+z^{(-)}\bigr)
\Big]
\Bigg\},
\end{multline}
if and only if
$
\gamma\geq \max\{h_0(a,\gamma),h_1(a,\gamma),\ldots,h_{k-1}(a,\gamma)\},
$
where
\begin{align}\label{hk-append}
    h_0(a,\gamma)=1,\qquad
    h_1(a,\gamma)=1+a,\qquad
    h_{j+1}(a,\gamma)
    =
    1+2a\,h_j(a,\gamma)-a\gamma,
    \quad j\geq1.
\end{align}

Fix \(k\geq2\) and suppose
that the claim holds at stage \(k-1\); namely, whenever
$$
\gamma\geq \max\{h_0(a,\gamma),h_1(a,\gamma),\ldots,h_{k-2}(a,\gamma)\},
$$
it holds that 
\begin{multline}\label{eq:Vk-1-form}
V_{k-1}(x,z^{(+)},z^{(-)})
=
\max\Bigg\{
\mathbb E\!\left[x-\gamma z^{(+)}\right],\;
\mathbb E\!\left[x-\gamma z^{(-)}\right],\;
\mathbb E\!\Big[
h_1(a,\gamma)x-\tfrac{\gamma}{2}\bigl(z^{(+)}+z^{(-)}\bigr)
\Big],\\
\mathbb E\!\Big[
h_{k-1}(a,\gamma)x-\tfrac{\gamma}{2}\bigl(z^{(+)}+z^{(-)}\bigr)
\Big]
\Bigg\}.
\end{multline}
We now apply the Bellman operator to~\eqref{eq:Vk-1-form} and show that
\(V_k=\mathcal F V_{k-1}\) is finite if and only if
$
\gamma\geq \max\{h_0(a,\gamma),h_1(a,\gamma),\ldots,h_{k-1}(a,\gamma)\}
$. Whenever this condition holds, the resulting
expression for \(V_k\) is precisely~\eqref{eq:Vk-claim1}. To this end, substituting the induction hypothesis \eqref{eq:Vk-1-form}, we obtain
\begin{multline*}
V_k(x,z^{(+)},z^{(-)})
=
\min_{|u|\le \tfrac{a}{b}x}\ \max_{v\ge0}
\max\Big(
\mathbb E\!\Big[
x+v-\gamma\bigl(z^{(+)}+|v-ax-bu|\bigr)
\Big],\\
\mathbb E\!\Big[
x+v-\gamma\bigl(z^{(-)}+|v-ax+bu|\bigr)
\Big],
\mathbb E\!\Big[
x+h_1(a,\gamma)v
-\tfrac{\gamma}{2}\bigl(z^{(+)}+|v-ax-bu|+z^{(-)}+|v-ax+bu|\bigr)
\Big],\\
\mathbb E\!\Big[
x+h_{k-1}(a,\gamma)v
-\tfrac{\gamma}{2}\bigl(z^{(+)}+|v-ax-bu|+z^{(-)}+|v-ax+bu|\bigr)
\Big]
\Big).
\end{multline*}
Exchanging \(\max_{v\ge0}\) with the outer maximum yields
\begin{multline*}
V_k(x,z^{(+)},z^{(-)})
=
\min_{|u|\le \tfrac{a}{b}x}
\max\Bigg\{
\max_{v\ge0}\mathbb E\!\Big[
x+v-\gamma\bigl(z^{(+)}+|v-ax-bu|\bigr)
\Big],\\
\max_{v\ge0}\mathbb E\!\Big[
x+v-\gamma\bigl(z^{(-)}+|v-ax+bu|\bigr)
\Big],
\max_{v\ge0}\mathbb E\!\Big[
x+h_1(a,\gamma)v
-\tfrac{\gamma}{2}\bigl(z^{(+)}+|v-ax-bu|+z^{(-)}+|v-ax+bu|\bigr)
\Big],\\
\max_{v\ge0}\mathbb E\!\Big[
x+h_{k-1}(a,\gamma)v
-\tfrac{\gamma}{2}\bigl(z^{(+)}+|v-ax-bu|+z^{(-)}+|v-ax+bu|\bigr)
\Big]
\Bigg\}.
\end{multline*}
Assume \(\gamma\ge1\) and \(ax\pm bu\ge0\). Following Lemma~\ref{lem:sup-exp}(L1),
\[
\max_{v\ge0}\mathbb E\!\Big[v-\gamma|ax\pm bu-v|\Big]
=
\mathbb E[ax\pm bu],
\]
with maximizer \(v^\star=ax\pm bu\). Thus, the first two terms become
\[
\mathbb E[(1+a)x+bu-\gamma z^{(+)}],
\qquad
\mathbb E[(1+a)x-bu-\gamma z^{(-)}].
\]
For the average terms, we consider
\[
\max_{v\geq0}
\mathbb E\!\Big[
h_j(a,\gamma)v
-\tfrac{\gamma}{2}|ax+bu-v|
-\tfrac{\gamma}{2}|ax-bu-v|
\Big],
\qquad j\in\{1,k-1\}.
\]
By Lemma~\ref{lem:sup-exp}(L2), these maxima are finite if and only if
$\gamma\ge h_1(a,\gamma)$
 and  
$\gamma\ge h_{k-1}(a,\gamma)$,
respectively. In both cases the maximizer is
$
v^\star=ax+|bu|,
$
and the corresponding optimal value is
\[
\mathbb E\!\Big[
h_j(a,\gamma)ax+\bigl(h_j(a,\gamma)-\gamma\bigr)|bu|
\Big],
\qquad j\in\{1,k-1\}.
\]
Therefore the two average terms simplify to
\[
\mathbb E\!\Big[
x+h_j(a,\gamma)ax+\bigl(h_j(a,\gamma)-\gamma\bigr)|bu|
-\tfrac{\gamma}{2}\bigl(z^{(+)}+z^{(-)}\bigr)
\Big],
\qquad j\in\{1,k-1\}.
\]

Collecting all four terms, we obtain
\begin{multline*}
V_k(x,z^{(+)},z^{(-)})
=
\min_{|u|\le \tfrac{a}{b}x}
\max\Big\{
\mathbb E[(1+a)x+bu-\gamma z^{(+)}],\;
\mathbb E[(1+a)x-bu-\gamma z^{(-)}],\\
\mathbb E\!\Big[
x+h_1(a,\gamma)ax+\bigl(h_1(a,\gamma)-\gamma\bigr)|bu|
-\tfrac{\gamma}{2}\bigl(z^{(+)}+z^{(-)}\bigr)
\Big],\\
\mathbb E\!\Big[
x+h_{k-1}(a,\gamma)ax+\bigl(h_{k-1}(a,\gamma)-\gamma\bigr)|bu|
-\tfrac{\gamma}{2}\bigl(z^{(+)}+z^{(-)}\bigr)
\Big]
\Big\},
\end{multline*}
if and only if $
\gamma\geq \max\{h_0(a,\gamma),h_1(a,\gamma),\ldots,h_{k-1}(a,\gamma)\}.
$
Next, we proceed with the minimization over $u$. The minimization over \(u\) now has the following structure
\[
V_k(x,z^{(+)},z^{(-)})
=
\min_{|u|\leq \tfrac{a}{b}x}
\max\left\{
V_+(u),\,
V_-(u),\,
V_{\mathrm{avg}}^{(2)}(u),\,
V_{\mathrm{avg}}^{(k)}(u)
\right\},
\]
where
\[
\begin{aligned}
V_{\pm}(u)
&\coloneqq
\mathbb E\!\Big[(1+a)x\pm bu-\gamma z^{(\pm)}\Big],\\
V_{\mathrm{avg}}^{(j+1)}(u)
&\coloneqq
\mathbb E\!\Big[
x+a h_j(a,\gamma)x
+\bigl(h_j(a,\gamma)-\gamma\bigr)|bu|
-\tfrac{\gamma}{2}\bigl(z^{(+)}+z^{(-)}\bigr)
\Big],
\qquad j\in\{1,k-1\}.
\end{aligned}
\]
Under the finiteness conditions established above,
\(h_1(a,\gamma)-\gamma\leq0\) and
\(h_{k-1}(a,\gamma)-\gamma\leq0\). Hence both average terms are minimized by
maximizing \(\mathbb E[|bu|]\), see Lemma~\ref{lemm1} for further details. We distinguish three regimes.
\begin{itemize}
\item[\textit{(1)}] \emph{Regime 1:
\(\gamma(z^{(-)}-z^{(+)})\geq 2ax\).}
Since \(|bu|\leq ax\) almost surely,
\[
V_+(u)-V_-(u)
=
\mathbb{E}\!\Big[2bu+\gamma(z^{(-)}-z^{(+)})\Big]
\geq
\mathbb{E}\!\Big[-2ax+\gamma(z^{(-)}-z^{(+)})\Big]
\geq 0.
\]
Thus, \(V_+(u)\geq V_-(u)\), so \(V_-(u)\) does not determine the maximum.

The deterministic control \(u^\star=-\tfrac{a}{b}x\) simultaneously minimizes \(V_+(u)\) and maximizes \(\mathbb{E}[|bu|]\), and hence it is also the optimal minimizer of the average terms. Substituting this control and using
\[
h_2(a,\gamma)=1+2a\,h_1(a,\gamma)-a\gamma,
\qquad
h_k(a,\gamma)=1+2a\,h_{k-1}(a,\gamma)-a\gamma,
\]
gives
\[
V_k
=
\max\Bigg\{
\mathbb{E}\!\left[x-\gamma z^{(+)}\right],\;
\mathbb{E}\!\Big[
h_2(a,\gamma)x-\tfrac{\gamma}{2}\bigl(z^{(+)}+z^{(-)}\bigr)
\Big],\;
\mathbb{E}\!\Big[
h_k(a,\gamma)x-\tfrac{\gamma}{2}\bigl(z^{(+)}+z^{(-)}\bigr)
\Big]
\Bigg\}.
\]
Notice that the condition \(\gamma\ge \max_{0\le j\le k-1} h_j(a,\gamma)\), together with the regime condition \(\gamma\bigl(z^{(-)}-z^{(+)}\bigr)\ge 2ax\)
is not sufficient to drop the averaged terms from the maximization. Doing so would require an additional condition on \(\gamma\). However, the optimal minimizing policy is the same for both averaged terms and the \(+\) term.
\item[\textit{(2)}] \emph{Regime 2:
\(\gamma(z^{(+)}-z^{(-)})\geq2ax\).}
Similarly,
\[
V_-(u)-V_+(u)
=
\mathbb E\!\Big[-2bu+\gamma(z^{(+)}-z^{(-)})\Big]
\geq
\mathbb E\!\Big[-2ax+\gamma(z^{(+)}-z^{(-)})\Big]
\geq0.
\]
Hence, \(V_-(u)\geq V_+(u)\), and the minimizing deterministic control is
$
u^\star=\tfrac{a}{b}x.
$
Consequently,
\[
V_k
=
\max\Bigg\{
\mathbb E\!\left[x-\gamma z^{(-)}\right],\;
\mathbb E\!\Big[
h_2(a,\gamma)x-\tfrac{\gamma}{2}(z^{(+)}+z^{(-)})
\Big],\;
\mathbb E\!\Big[
h_k(a,\gamma)x-\tfrac{\gamma}{2}(z^{(+)}+z^{(-)})
\Big]
\Bigg\}.
\]

\item[\textit{(3)}] \emph{Regime 3:
\(\bigl|\gamma(z^{(+)}-z^{(-)})\bigr|<2ax\).}
In this regime, neither \(V_+(u)\) nor \(V_-(u)\) dominates. We therefore choose
a randomized control that equalizes the two terms in expectation.
Since the controller observes \((x,z^{(+)},z^{(-)})\) before acting, the
condition
$
V_+(u)=V_-(u)
$
is equivalent to
$
\mathbb E[bu]
=
\tfrac{\gamma}{2}\bigl(z^{(+)}-z^{(-)}\bigr).
$
Moreover, since
\(h_1(a,\gamma)-\gamma\leq0\) and \(h_{k-1}(a,\gamma)-\gamma\leq0\),
both \(V_{\mathrm{avg}}^{(2)}(u)\) and
\(V_{\mathrm{avg}}^{(k)}(u)\) are minimized by maximizing
\(\mathbb E[|bu|]\). Under the almost sure constraint
$
|u|\leq\tfrac{a}{b}x,
$
we have $\mathbb E[|u|]\leq \tfrac{a}{b}x \Longleftrightarrow \mathbb E[|bu|]\leq ax$.
Hence, at the optimum
\[
\mathbb E[u]
=
\tfrac{\gamma}{2b}\bigl(z^{(+)}-z^{(-)}\bigr) \quad
\mathbb E[|u|]=\tfrac{a}{b}x,
\]
Substituting these identities gives
\[
V_+(u)=V_-(u)
=
\mathbb E\!\Big[
h_1(a,\gamma)x
-\tfrac{\gamma}{2}\bigl(z^{(+)}+z^{(-)}\bigr)
\Big],
\]
while
\[
V_{\mathrm{avg}}^{(j)}(u)
=
\mathbb E\!\Big[
h_j(a,\gamma)x
-\tfrac{\gamma}{2}\bigl(z^{(+)}+z^{(-)}\bigr)
\Big],
\qquad j\in\{2,k\}.
\]
Therefore,
\begin{multline*}
V_k(x,z^{(+)},z^{(-)})
=
\max\Bigg\{
\mathbb E\!\Big[
h_1(a,\gamma)x
-\tfrac{\gamma}{2}\bigl(z^{(+)}+z^{(-)}\bigr)
\Big],\\
\mathbb E\!\Big[
h_2(a,\gamma)x
-\tfrac{\gamma}{2}\bigl(z^{(+)}+z^{(-)}\bigr)
\Big],\;
\mathbb E\!\Big[
h_k(a,\gamma)x
-\tfrac{\gamma}{2}\bigl(z^{(+)}+z^{(-)}\bigr)
\Big]
\Bigg\}.
\end{multline*}
\end{itemize}
Combining the three regimes, we conclude that \(V_k\) is finite if and only if
$
\gamma \ge \max_{0\le j\le k-1} h_j(a,\gamma),
$
in which case
\begin{multline*}
V_k(x,z^{(+)},z^{(-)})
=
\max\Bigg\{
\mathbb E\!\left[x-\gamma z^{(+)}\right],\;
\mathbb E\!\left[x-\gamma z^{(-)}\right],\;
\mathbb E\!\Big[
h_1(a,\gamma)x-\tfrac{\gamma}{2}\bigl(z^{(+)}+z^{(-)}\bigr)
\Big],\\
\mathbb E\!\Big[
h_2(a,\gamma)x-\tfrac{\gamma}{2}\bigl(z^{(+)}+z^{(-)}\bigr)
\Big],\;
\mathbb E\!\Big[
h_k(a,\gamma)x-\tfrac{\gamma}{2}\bigl(z^{(+)}+z^{(-)}\bigr)
\Big]
\Bigg\}.
\end{multline*}
To eliminate the \(h_2\)-term, first observe that
\[
h_2(a,\gamma)-h_1(a,\gamma)
=
a\bigl(1+2a-\gamma\bigr).
\]
Moreover, the recursion for \(h_j\) gives
\[
\begin{aligned}
h_{j+1}(a,\gamma)-h_j(a,\gamma)
&=
\bigl(1+2a h_j(a,\gamma)-a\gamma\bigr)
-\bigl(1+2a h_{j-1}(a,\gamma)-a\gamma\bigr)\\
&=
2a\bigl(h_j(a,\gamma)-h_{j-1}(a,\gamma)\bigr),
\qquad j\geq2.
\end{aligned}
\]
Therefore,
\[
h_{j+1}(a,\gamma)-h_j(a,\gamma)
=
(2a)^{j-1}a\bigl(1+2a-\gamma\bigr),
\qquad j\geq1.
\]
Since \(a\geq0\), all successive differences have the same sign. More
precisely,
\[
\begin{cases}
h_1(a,\gamma)\leq h_2(a,\gamma)\leq\cdots\leq h_k(a,\gamma),
& \gamma\leq1+2a,\\[2pt]
h_1(a,\gamma)\geq h_2(a,\gamma)\geq\cdots\geq h_k(a,\gamma),
& \gamma\geq1+2a.
\end{cases}
\]
Thus, in either case,
$
h_2(a,\gamma)
\leq
\max\bigl\{h_1(a,\gamma),h_k(a,\gamma)\bigr\}.
$
Since all three corresponding terms have the same history penalty
\(-\tfrac{\gamma}{2}(z^{(+)}+z^{(-)})\), it follows that
$
\mathbb E\!\Big[
h_2(a,\gamma)x
-\tfrac{\gamma}{2}\bigl(z^{(+)}+z^{(-)}\bigr)
\Big]
$
can never exceed both the \(h_1\)-term and the \(h_k\)-term for all admissible values of $\gamma$ satisfying $
\gamma\geq \max\{h_0(a,\gamma),h_1(a,\gamma),\ldots,h_{k-1}(a,\gamma)\}
$. Hence, it is
redundant in the maximum. Consequently,
\begin{multline*}
V_k(x,z^{(+)},z^{(-)})
=
\max\Bigg\{
\mathbb E\!\left[x-\gamma z^{(+)}\right],\;
\mathbb E\!\left[x-\gamma z^{(-)}\right],\;
\mathbb E\!\Big[
h_1(a,\gamma)x-\tfrac{\gamma}{2}\bigl(z^{(+)}+z^{(-)}\bigr)
\Big],\\
\mathbb E\!\Big[
h_k(a,\gamma)x-\tfrac{\gamma}{2}\bigl(z^{(+)}+z^{(-)}\bigr)
\Big]
\Bigg\},
\end{multline*}
which matches the claimed form and completes the induction argument. This then concludes the Proof of~Theorem~\ref{thm:bellman-termes}\textit{(i)}.

\noindent\emph{Proof of Theorem~\ref{thm:bellman-termes}\textit{(ii)}.}
By Theorem~\ref{thm:bellman-termes}\textit{(i)}, the \(k\)-th value-iteration
iterate \(V_k\) is finite if and only if
$
\gamma \ge \max_{0\le j\le k-1} h_j(a,\gamma).
$
Therefore, to determine whether the infinite horizon value is finite and to
identify the limiting value function, we must characterize the behavior of
\(h_k(a,\gamma)\) as \(k\to\infty\). Accordingly, we analyze the linear
recursion in~\eqref{hk-append} and determine for which values of \(\gamma\) the
finiteness conditions hold for every \(k\). We begin by analyzing the strict case \(a>\tfrac{1}{2}\). The marginal case
\(a=\tfrac{1}{2}\) is treated separately afterward.
For \(a>\tfrac12\), this recursion
has the unique fixed point
\[
h_\star(a,\gamma)=1+2a\,h_\star(a,\gamma)-a\gamma
\quad\Longrightarrow\quad
h_\star(a,\gamma)=\tfrac{1-a\gamma}{1-2a}.
\]
Subtracting \(h_\star(a,\gamma)\) from the recursion gives
\[
h_{k+1}(a,\gamma)-h_\star(a,\gamma)
=
2a\bigl(h_k(a,\gamma)-h_\star(a,\gamma)\bigr),
\]
and hence, by applying the recursion repeatedly,
\begin{equation}\label{eq:h-closed-new}
h_k(a,\gamma)
=
h_\star(a,\gamma)
+(2a)^{k-1}\bigl(h_1(a,\gamma)-h_\star(a,\gamma)\bigr),
\qquad k\geq1.
\end{equation}
Using \(h_1(a,\gamma)=1+a\), we obtain
$
h_1(a,\gamma)-h_\star(a,\gamma)
=
\tfrac{a}{1-2a}\bigl(\gamma-1-2a\bigr).
 $
Define
$
\gamma_\star\coloneqq1+2a.
$
Since \(a>\tfrac12\), we have \(2a>1\) and \(1-2a<0\). Therefore, the sign of
\(h_1(a,\gamma)-h_\star(a,\gamma)\) is opposite to that of
\(\gamma-\gamma_\star\). We now distinguish three cases.

\begin{itemize}
\item[\textit{(i)}] \emph{Suppose that \(\gamma<\gamma_\star\).}
Then \(h_1(a,\gamma)-h_\star(a,\gamma)>0\). Since \(2a>1\), we have
\((2a)^{k-1}\to+\infty\). Hence, the second term on the right-hand side
of~\eqref{eq:h-closed-new} diverges to \(+\infty\), and therefore
\[
h_k(a,\gamma)\longrightarrow+\infty
\qquad\text{as }k\to\infty.
\]
Consequently, there exists \(K\) such that
$
\gamma<h_{K-1}(a,\gamma).
$
By Theorem~\ref{thm:bellman-termes}\textit{(i)}, \(V_K\) is then infinite.
Hence, the sequence of value functions diverges.

\item[\textit{(ii)}] \emph{Suppose that \(\gamma=\gamma_\star\).}
In this case,
$
h_1(a,\gamma_\star)-h_\star(a,\gamma_\star)=0,
$
and therefore~\eqref{eq:h-closed-new} gives
\[
h_k(a,\gamma_\star)=h_\star(a,\gamma_\star)=1+a,
\qquad k\geq1.
\]
Moreover,
\[
h_k(a,\gamma_\star)=1+a\leq1+2a=\gamma_\star,
\]
so the finiteness condition \(\gamma_\star =1+2a \geq \max\{h_0(a,\gamma_\star),h_1(a,\gamma_\star),\ldots,h_{k-1}(a,\gamma_\star)\}=1+a\) holds for every \(k\geq1\).
Substituting \(h_k(a,\gamma_\star)=h_1(a,\gamma_\star)\) into the expression for
\(V_k\) yields
\[
V_k(x,z^{(+)},z^{(-)})
=
\max\Bigg\{
\mathbb E[x-\gamma_\star z^{(+)}],\;
\mathbb E[x-\gamma_\star z^{(-)}],\;
\mathbb E\!\Big[
(1+a)x-\tfrac{\gamma_\star}{2}\bigl(z^{(+)}+z^{(-)}\bigr)
\Big]
\Bigg\},
\, k\geq1.
\]

\item[\textit{(iii)}] \emph{Suppose that \(\gamma>\gamma_\star\).}
Then \(h_1(a,\gamma)-h_\star(a,\gamma)<0\). Since \(2a>1\),~\eqref{eq:h-closed-new} gives
\[
h_k(a,\gamma)\longrightarrow-\infty
\qquad\text{as }k\to\infty.
\]
More precisely,
\[
h_2(a,\gamma)-h_1(a,\gamma)
=
a\bigl(1+2a-\gamma\bigr)<0,
\]
and the recursion also implies
\[
h_{k+1}(a,\gamma)-h_k(a,\gamma)
=
2a\bigl(h_k(a,\gamma)-h_{k-1}(a,\gamma)\bigr).
\]
Thus, \(\{h_k(a,\gamma)\}_{k\geq1}\) is strictly decreasing, and
\[
h_k(a,\gamma)\leq h_1(a,\gamma)=1+a<1+2a=\gamma_\star<\gamma,
\qquad k\geq1.
\]
Hence, \(\gamma\geq \max\{h_0(a,\gamma),h_1(a,\gamma),\ldots,h_{k-1}(a,\gamma)\}\) for every \(k\geq1\). Furthermore, since \(x\geq0\),
\[
h_k(a,\gamma)x-\tfrac{\gamma}{2}\bigl(z^{(+)}+z^{(-)}\bigr)
\leq
(1+a)x-\tfrac{\gamma}{2}\bigl(z^{(+)}+z^{(-)}\bigr).
\]
Therefore, the \(h_k\)-term is dominated by the \(h_1\)-term, and
\[
V_k(x,z^{(+)},z^{(-)})
=
\max\Bigg\{
\mathbb E[x-\gamma z^{(+)}],\;
\mathbb E[x-\gamma z^{(-)}],\;
\mathbb E\!\Big[
(1+a)x-\tfrac{\gamma}{2}\bigl(z^{(+)}+z^{(-)}\bigr)
\Big]
\Bigg\},
\qquad k\geq1.
\]
\end{itemize}
Combining the three cases, we conclude that the value-iteration sequence
converges to a finite limit if and only if
$
\gamma\geq\gamma_\star=1+2a.
$
For every \(\gamma\geq\gamma_\star\), the \(h_k\)-term is dominated by the
\(h_1\)-term for all \(k\geq1\). Hence,
\[
V_k(x,z^{(+)},z^{(-)})=V_1(x,z^{(+)},z^{(-)}),
\qquad k\geq1,
\]
and the sequence reaches its limit after one iteration. Therefore,
\[
V_\star(x,z^{(+)},z^{(-)})
=
\max\Bigg\{
\mathbb E[x-\gamma z^{(+)}],\;
\mathbb E[x-\gamma z^{(-)}],\;
\mathbb E\!\Big[
(1+a)x-\tfrac{\gamma}{2}\bigl(z^{(+)}+z^{(-)}\bigr)
\Big]
\Bigg\}.
\]
Moreover, \(V_2=\mathcal F V_1=V_1=V_\star\), and thus
$
\mathcal F V_\star=V_\star.
$
Hence, \(V_\star\) is a fixed point of the Bellman operator.

We now turn our attention to the marginal case \(a=\tfrac12\). Then, the recursion~\eqref{hk-append} becomes
\[
h_{k+1}\!\left(\tfrac12,\gamma\right)
=
1+h_k\!\left(\tfrac12,\gamma\right)-\tfrac{\gamma}{2}.
\]
Since
$
h_1\!\left(\tfrac12,\gamma\right)=1+\tfrac12=\tfrac32,
$
this recursion solves explicitly to
\[
h_k\!\left(\tfrac12,\gamma\right)
=
\tfrac32+(k-1)\Bigl(1-\tfrac{\gamma}{2}\Bigr),
\qquad k\geq 1.
\]
Hence, the behavior is completely determined by the sign of $1-\tfrac{\gamma}{2}$.
\begin{itemize}
\item[\textit{(1)}] If $\gamma<2$, then $1-\gamma/2>0$, so
$
h_k\!\left(\tfrac12,\gamma\right)\to +\infty.
$
In particular, there exists $K$ such that $h_K\!\left(\tfrac12,\gamma\right)>\gamma$, and Theorem~\ref{thm:bellman-termes}\textit{(i)} implies that $V_{K+1}$ is infinite. Therefore the value iteration sequence cannot converge to a finite limit.

\item[\textit{(2)}]  If $\gamma=2$, then
$
h_k\!\left(\tfrac12,2\right)=\tfrac32=h_1\!\left(\tfrac12,2\right)<\gamma
$ for $k\ge 1$.
Thus, the moving term and the $h_1$ term coincide, and for every $k\geq 1$ we have
\begin{multline*}
V_k(x,z^{(+)},z^{(-)})
=V_\star(x,z^{(+)},z^{(-)})=
\max\Bigg\{
\mathbb{E}\bigl[x-2z^{(+)}\bigr],\;
\mathbb{E}\bigl[x-2z^{(-)}\bigr],
\mathbb{E}\!\Big[
\tfrac32\,x-\bigl(z^{(+)}+z^{(-)}\bigr)
\Big]
\Bigg\}.
\end{multline*}
\item[\textit{(3)}] If $\gamma>2$, then $1-\gamma/2<0$, so
$
h_k\!\left(\tfrac12,\gamma\right)\to -\infty.
$
Hence the $h_k$ term disappears in the limit, while the $h_1$ term remains. Therefore,
\[
V_\star(x,z^{(+)},z^{(-)})
=
\max\Bigg\{
\mathbb E\!\left[x-\gamma z^{(+)}\right],\;
\mathbb E\!\left[x-\gamma z^{(-)}\right],\;
\mathbb E\!\Big[
\tfrac32\,x-\tfrac{\gamma}{2}\bigl(z^{(+)}+z^{(-)}\bigr)
\Big]
\Bigg\}.
\]
\end{itemize}
Consequently, for $a=\tfrac12$, the value iteration sequence also converges if and only if
$
\gamma\ge \left.(1+2a)\right|_{a=\frac12}=2.
$ In that case, 
\begin{align}\label{Vstar11}
    V_\star(x,z^{(+)},z^{(-)})
=
\max\Bigg\{
\mathbb E\!\left[x-\gamma z^{(+)}\right],\;
\mathbb E\!\left[x-\gamma z^{(-)}\right],\;
\mathbb E\!\Big[
\left.(1+a)\right|_{a=\frac12}
\,x-\tfrac{\gamma}{2}\bigl(z^{(+)}+z^{(-)}\bigr)
\Big]
\Bigg\}.
\end{align}
Combining the two regimes $a>\tfrac12$ and $a=\tfrac12$, we conclude that the value iteration sequence converges if and only if $\gamma\ge 1+2a$. In that case,
\[
V_\star(x,z^{(+)},z^{(-)})
=
\max\Bigg\{
\mathbb E[x-\gamma z^{(+)}],\;
\mathbb E[x-\gamma z^{(-)}],\;
\mathbb E\!\Big[
(1+a)x-\tfrac{\gamma}{2}\bigl(z^{(+)}+z^{(-)}\bigr)
\Big]
\Bigg\}.
\]

\noindent\emph{Proof of Theorem~\ref{thm:bellman-termes}\textit{(iii)}.}
By Theorem~\ref{thm:bellman-termes}\textit{(ii)}, for every
\(\gamma\geq\gamma_\star\), the value iteration satisfies \(V_k=V_\star\) for all
\(k\geq1\), and
$
\mathcal F V_\star=V_\star.
$
Moreover, the minimization in \(\mathcal F V_\star\) has exactly the same structure as the minimization appearing in the induction step from \(V_{k-1}\) to \(V_k\). In particular, after applying the Bellman operator and maximizing over \(v\ge 0\), we obtain the same minimization problem as in the induction argument.
\begin{multline*}
\mathcal F V_\star(x,z^{(+)},z^{(-)})
=
\min_{|u|\leq\tfrac{a}{b}x}
\max\Bigg\{
\underbrace{
\mathbb E\!\left[(1+a)x+bu-\gamma z^{(+)}\right]
}_{V_+(u)},
\;
\underbrace{
\mathbb E\!\left[(1+a)x-bu-\gamma z^{(-)}\right]
}_{V_-(u)},
\\
\underbrace{
\mathbb E\!\Big[
x+a(1+a)x+(1+a-\gamma)|bu|
-\tfrac{\gamma}{2}\bigl(z^{(+)}+z^{(-)}\bigr)
\Big]
}_{V_{\mathrm{avg}}(u)}
\Bigg\}.
\end{multline*}
Since \(\gamma\ge \gamma_\star=1+2a\), we have
$
1+a-\gamma \le 1+a-\gamma_\star \le 0,
$
which is strictly negative when \(a\neq 0\), so
\(V_{\mathrm{avg}}(u)\) is minimized by maximizing \(\mathbb E[|bu|]\) as Lemma~\ref{lemm1} shows, which gives
\(\mathbb E[|bu|]=ax\).
We distinguish three regimes as before.
\begin{itemize}
\item[\textit{(i)}] If \(z^{(+)}-z^{(-)}<-\tfrac{2a}{\gamma}x\), then
$
V_+(u)-V_-(u)
=
\mathbb E\!\left[2bu+\gamma\bigl(z^{(-)}-z^{(+)}\bigr)\right]>0
$
for every admissible \(u\). Hence \(V_+(u)\) dominates \(V_-(u)\), and both
\(V_+(u)\) and \(V_{\mathrm{avg}}(u)\) are minimized by
$
u^\star=-\tfrac{a}{b}x.
$
\item[\textit{(ii)}] If \(z^{(+)}-z^{(-)}>\tfrac{2a}{\gamma}x\), then
$
V_-(u)>V_+(u)
$
for every admissible \(u\). Hence \(V_-(u)\) dominates \(V_+(u)\), and both
\(V_-(u)\) and \(V_{\mathrm{avg}}(u)\) are minimized by
$
u^\star=\tfrac{a}{b}x.
$
\item[\textit{(iii)}] If
$
\bigl|z^{(+)}-z^{(-)}\bigr|\leq \tfrac{2a}{\gamma}x,
$
then neither boundary term dominates. They are equalized by
\[
V_+(u)=V_-(u)
\quad\Longleftrightarrow\quad
\mathbb E[u]
=
\tfrac{\gamma\bigl(z^{(+)}-z^{(-)}\bigr)}{2b},
\]
while \(V_{\mathrm{avg}}(u)\) is minimized by
$
\mathbb E[|u|]=\tfrac{a}{b}x.
$
The regime condition guarantees that these moment constraints are compatible with
\(|u|\leq \tfrac{a}{b}x\) almost surely.
\end{itemize}
Therefore, the minimizing policy \(u_t=\eta^\star(x_t,z_t^{(+)},z_t^{(-)})\) of
\(\mathcal F V_\star\) is given by
\begin{align}\label{optimalagehalf}
    \left\{
\begin{aligned}
u^\star
&=
-\tfrac{a}{b}x,
&&
\text{if }
z^{(+)}-z^{(-)}
<
-\tfrac{2a}{\gamma}x,
\\[2pt]
u^\star
&=
\tfrac{a}{b}x,
&&
\text{if }
z^{(+)}-z^{(-)}
>
\tfrac{2a}{\gamma}x,
\\[2pt]
\mathbb E[u^\star]
&=
\tfrac{\gamma\bigl(z^{(+)}-z^{(-)}\bigr)}{2b},
\qquad
\mathbb E[|u^\star|]
=
\tfrac{a}{b}x,
&&
\text{if }
\bigl|z^{(+)}-z^{(-)}\bigr|
\leq
\tfrac{2a}{\gamma}x,
\end{aligned}
\right.
\end{align}
and the resulting value is 
\begin{multline*}
\mathcal F V_\star(x,z^{(+)},z^{(-)})
=
\max\Bigg\{
\mathbb E\!\left[x-\gamma z^{(+)}\right],\;
\mathbb E\!\left[x-\gamma z^{(-)}\right],\\
\mathbb E\!\left[
h_1(a,\gamma)x-\tfrac{\gamma}{2}\bigl(z^{(+)}+z^{(-)}\bigr)
\right],\;
\mathbb E\!\left[
h_2(a,\gamma)x-\tfrac{\gamma}{2}\bigl(z^{(+)}+z^{(-)}\bigr)
\right]
\Bigg\}.
\end{multline*}
where
$
h_1(a,\gamma)=1+a$ and 
$h_2(a,\gamma)=1+2a\,h_1(a,\gamma)-a\gamma.$
Since \(\gamma\ge\gamma_\star=1+2a\), we have
\[
h_2(a,\gamma)-h_1(a,\gamma)=a(1+2a-\gamma)\le 0,
\]
so \(h_2(a,\gamma)\le h_1(a,\gamma)\), and the \(h_2\)-term is redundant in the maximum and therefore can be dropped. Hence, the Bellman update is closed
\[
\mathcal F V_\star(x,z^{(+)},z^{(-)})=V_\star(x,z^{(+)},z^{(-)}).
\]
Since \(\mathcal F V_\star=V_\star\), Theorem~\ref{thm:l1-dp}\textit{(ii)} implies that the minimizing policy in~\eqref{optimalagehalf} coincides with~\eqref{eq:ustar-cases1-thm} and is optimal for Problem~\ref{pb2}. Moreover, Theorem~\ref{thm:l1-dp}\textit{(ii)} implies that substituting
\(z_t^{(+)}=\sum_{\tau=0}^{t-1}\lvert x_{\tau+1}-ax_\tau-bu_\tau\rvert\) and
\(z_t^{(-)}=\sum_{\tau=0}^{t-1}\lvert x_{\tau+1}-ax_\tau+bu_\tau\rvert\)
into~\eqref{optimalagehalf} yields the causal randomized policy stated in~\eqref{dual-cont}, which is optimal for Problem~\ref{prob:l1-robust-control}.

\noindent\emph{Proof of Theorem~\ref{thm:bellman-termes}\textit{(iv)}.}
By Theorem~\ref{thm:l1-dp}\textit{(i)}, Problems~\ref{pb2} and
\ref{prob:l1-robust-control} have the same value, which is finite if and only if
the value iteration sequence is bounded above. By
Theorem~\ref{thm:bellman-termes}\textit{(ii)}, this occurs if and only if
\(\gamma\geq\gamma_\star\).

For \(\gamma\geq\gamma_\star\), we have \(V_k=V_\star\) for every \(k\geq1\).
Evaluating the expression for \(V_\star\) at \(z^{(+)}=z^{(-)}=0\) gives
\[
J_\star(x_0)=V_\star(x_0,0,0)
=\max\left\{x_0,x_0,(1+a)x_0\right\}
=(1+a)x_0.
\]
where the last equality follows from \(a\ge0\) and \(x_0\geq0\). Hence,
\begin{align}\label{opt_cost}
    J_\star(x_0)=V_k(x_0,0,0)=(1+a)x_0,
\qquad k\geq1.
\end{align}
By Theorem~\ref{thm:bellman-termes}\textit{(iii)}, \(u_t=\eta^\star(x_t,z_t^{(+)},z_t^{(-)})\) achieves the optimal cost in~\eqref{opt_cost} for Problem~\ref{pb2}. Moreover, substituting \(z_t^{(+)}=\sum_{\tau=0}^{t-1}\lvert x_{\tau+1}-ax_\tau-bu_\tau\rvert\) and \(z_t^{(-)}=\sum_{\tau=0}^{t-1}\lvert x_{\tau+1}-ax_\tau+bu_\tau\rvert\) into~\eqref{eq:ustar-cases1-thm} yields the optimal policy stated in~\eqref{dual-cont}, which attains the same optimal cost for Problem~\ref{prob:l1-robust-control}.

This completes the proof of all statements in Theorem~\ref{thm:bellman-termes}.
\qedblack

\subsection{Proof of Theorem~\ref{thm:small-a-zero-policy}}\label{thm:small-a-zero-policy-append}
We proceed in four steps, corresponding to parts \textit{(i)}--\textit{(iv)} of
Theorem~\ref{thm:small-a-zero-policy}. Let \(0\leq a<\tfrac{1}{2}\) and, without loss
of generality, \(b>0\). Recall that the Bellman operator is given by
\[
\begin{aligned}
\mathcal{F}V(x,z^{(+)},z^{(-)})
=
\min_{|u|\leq \tfrac{a}{b}x}
\ \max_{v\geq 0}\,
\Bigg(
\mathbb E[x]
+
V\!\left(
v,\,
z^{(+)}+\left|v-ax-bu\right|,
z^{(-)}+\left|v-ax+bu\right|
\right)
\Bigg)
\end{aligned}
\]
and that the value iteration sequence is initialized and updated according to~\eqref{vk}.

\medskip
\noindent\emph{Proof of Theorem~\ref{thm:small-a-zero-policy}\textit{(i)}.}
The proof is identical to that of Theorem~\ref{thm:bellman-termes}\textit{(i)}. Indeed, the argument there does not use any restriction on the parameter \(a\). Therefore, the proof carries over verbatim for all \(a\ge 0\), and in particular for \(0\le a< \tfrac12\).

\noindent\emph{Proof of Theorem~\ref{thm:small-a-zero-policy}\textit{(ii)}.}
By Theorem~\ref{thm:small-a-zero-policy}\textit{(i)}, the \(k\)-th value iteration
iterate \(V_k\) is finite if and only if
$
\gamma \ge \max_{0\le j\le k-1} h_j(a,\gamma).
$
Therefore, to determine whether the infinite horizon value is finite and to
identify the limiting value function, we must characterize the behavior of
\(h_k(a,\gamma)\) as \(k\to\infty\). Accordingly, we analyze the linear
recursion in~\eqref{hk1} and determine for which values of \(\gamma\) the
finiteness conditions hold for every \(k\).

Recall from the proof of Theorem~\ref{thm:bellman-termes}\textit{(ii)} that the recursion in~\eqref{hk1} has the fixed point \(h_\star(a,\gamma)=1+2a\,h_\star(a,\gamma)-a\gamma\Leftrightarrow h_\star(a,\gamma)=\tfrac{1-a\gamma}{1-2a}\), and that \(h_{k+1}(a,\gamma)-h_\star(a,\gamma)=2a\bigl(h_k(a,\gamma)-h_\star(a,\gamma)\bigr)\). Hence, for every \(k\ge1\),
\begin{equation}\label{eq:hclosed-a-less-half}
h_k(a,\gamma)=h_\star(a,\gamma)+(2a)^{k-1}\bigl(h_1(a,\gamma)-h_\star(a,\gamma)\bigr).
\end{equation}
Since \(h_1(a,\gamma)=1+a\), we get \(h_1(a,\gamma)-h_\star(a,\gamma)=\tfrac{a}{1-2a}\bigl(\gamma-(1+2a)\bigr)\), and therefore \(h_{k+1}(a,\gamma)-h_k(a,\gamma)=a(2a)^{k-1}\bigl(1+2a-\gamma\bigr)\), \(k\ge1\).
Since $a\ge0$ and $0\le 2a<1$, the sign of this difference is independent of $k$. Hence, the sequence $\{h_k(a,\gamma)\}_{k\ge1}$ is monotone: it is increasing if $\gamma<1+2a$, decreasing if $\gamma>1+2a$, and constant if $\gamma=1+2a$. In particular, every term lies between $h_1(a,\gamma)$ and $h_\star(a,\gamma)$, and
\[
h_k(a,\gamma)\longrightarrow h_\star(a,\gamma)
\quad\text{as }k\to\infty.
\]
We now determine exactly when the finiteness condition holds for every $k$. By Theorem~\ref{thm:small-a-zero-policy}\textit{(i)}, this is equivalent to
$
\gamma \ge \max_{0\le j\le k-1} h_j(a,\gamma)$ for $k\ge 0$.
If $\gamma<h_\star(a,\gamma)$, then since $h_k(a,\gamma)\to h_\star(a,\gamma)$, for some large enough $K$ we have
$
h_K(a,\gamma)>\gamma.
$
Then, the finiteness condition fails at step $K+1$, so $V_{K+1}$ is infinite by Theorem~\ref{thm:small-a-zero-policy}\textit{(i)}.

Conversely, suppose $\gamma\ge h_\star(a,\gamma)$. Then
\[
\gamma\ge \tfrac{1-a\gamma}{1-2a}
\quad\Longleftrightarrow\quad
\gamma(1-a)\ge 1
\quad\Longleftrightarrow\quad
\gamma\ge \tfrac{1}{1-a}\eqqcolon \gamma_\star.
\]
Therefore also
$
\gamma\ge \gamma_\star> h_1(a,\gamma)=1+a
$ for $0\leq a<\tfrac{1}{2}$.
Since $\{h_k(a,\gamma)\}_{k\ge1}$ is monotone and converges to $h_\star(a,\gamma)\le\gamma$, we obtain
\[
h_k(a,\gamma)\le \gamma
\qquad\text{for all }k\ge1.
\]
Consequently, the finiteness conditions
$
\gamma \ge \max_{0\le j\le k-1} h_j(a,\gamma)
$
holds for every $k$ indicating that the iterate $V_k$ is finite for all $k$.
Hence, for $0\le a<\tfrac12$, the value iteration sequence converges if and only if
$\gamma\ge \tfrac{1}{1-a}$.
In that case, the limiting value function is 
\begin{multline}\label{Vstar}
V_\star(x,z^{(+)},z^{(-)})
=
\max\Bigg\{
\mathbb{E}\bigl[x-\gamma z^{(+)}\bigr],\;
\mathbb{E}\bigl[x-\gamma z^{(-)}\bigr],\\
\mathbb{E}\!\Big[
\tfrac{1-a\gamma}{1-2a}\,x
-\tfrac{\gamma}{2}\bigl(z^{(+)}+z^{(-)}\bigr)
\Big],\;
\mathbb{E}\!\Big[
(1+a)x
-\tfrac{\gamma}{2}\bigl(z^{(+)}+z^{(-)}\bigr)
\Big]
\Bigg\}.
\end{multline}
\noindent\emph{Proof of Theorem~\ref{thm:small-a-zero-policy}\textit{(iii)}.}
Let \(\gamma=\gamma_\star=\tfrac{1}{1-a}\). Then, it holds that
$
h_\star(a,\gamma_\star)=\tfrac{1-a\gamma_\star}{1-2a}=\gamma_\star.
$
Moreover,
$
\gamma_\star-(1+a)=\tfrac{a^2}{1-a}\geq0,
$
so the \((1+a)\)-term in~\eqref{Vstar} is dominated by the \(h_\star(a,\gamma_\star)\)-term. Hence, \eqref{Vstar} becomes
\begin{multline}\label{eq:Vk-small-a}
V_\star(x,z^{(+)},z^{(-)})
=
\max\Bigg\{
\mathbb E\!\left[x-\gamma_\star z^{(+)}\right],\;
\mathbb E\!\left[x-\gamma_\star z^{(-)}\right],
\mathbb E\!\Big[
\gamma_\star x
-\tfrac{\gamma_\star}{2}\bigl(z^{(+)}+z^{(-)}\bigr)
\Big]
\Bigg\}.
\end{multline}
Applying the Bellman operator to~\eqref{eq:Vk-small-a} gives
\begin{multline*}
\mathcal F V_\star(x,z^{(+)},z^{(-)})
=
\min_{\lvert u\rvert\leq\tfrac{a}{b}x}
\max\Bigg\{
\max_{v\geq0}\mathbb E\!\Big[
x+v-\gamma_\star\bigl(z^{(+)}+\lvert v-ax-bu\rvert\bigr)
\Big],\\
\max_{v\geq0}\mathbb E\!\Big[
x+v-\gamma_\star\bigl(z^{(-)}+\lvert v-ax+bu\rvert\bigr)
\Big],\\
\max_{v\geq0}\mathbb E\!\Bigg[
x+\gamma_\star v
-\tfrac{\gamma_\star}{2}\Big(
z^{(+)}+\lvert v-ax-bu\rvert
+z^{(-)}+\lvert v-ax+bu\rvert
\Big)
\Bigg]
\Bigg\}.
\end{multline*}
Applying Lemma~\ref{lem:sup-exp}\textit{(L1)} to the first two maximizations, and since \(\gamma_\star \geq 1\), we obtain
\begin{multline*}
\mathcal F V_\star(x,z^{(+)},z^{(-)})
=
\min_{\lvert u\rvert\leq\tfrac{a}{b}x}
\max\Bigg\{
\mathbb E\!\Big[(1+a)x+bu-\gamma_\star z^{(+)}\Big],\;
\mathbb E\!\Big[(1+a)x-bu-\gamma_\star z^{(-)}\Big],\\
V_{\mathrm{avg}}(u)
\Bigg\},
\end{multline*}
where
\[
V_{\mathrm{avg}}(u)
\coloneqq
\max_{v\geq0}
\mathbb E\!\Bigg[
x+\gamma_\star v
-\tfrac{\gamma_\star}{2}\Big(
z^{(+)}+\lvert v-ax-bu\rvert
+z^{(-)}+\lvert v-ax+bu\rvert
\Big)
\Bigg].
\]
Applying Lemma~\ref{lem:sup-exp}\textit{(L2)} to \(V_{\mathrm{avg}}(u)\), we get
\begin{align*}
V_{\mathrm{avg}}(u)
=
\mathbb E\!\Big[
x+a\gamma_\star x
+\bigl(\gamma_\star-\gamma_\star\bigr)\lvert bu\rvert
-\tfrac{\gamma_\star}{2}\bigl(z^{(+)}+z^{(-)}\bigr)
\Big]
=
\mathbb E\!\Big[
\gamma_\star x
-\tfrac{\gamma_\star}{2}\bigl(z^{(+)}+z^{(-)}\bigr)
\Big],
\end{align*}
where the last equality follows from \(1+a\gamma_\star=\gamma_\star\). Thus, the coefficient in front of the exploration incentivizing term \(\mathbb E[\lvert bu\rvert]\) vanishes, and \(V_{\mathrm{avg}}(u)\) is independent of \(u\). Consequently,
\begin{multline*}
\mathcal F V_\star(x,z^{(+)},z^{(-)})
=
\min_{\lvert u\rvert\leq\tfrac{a}{b}x}
\max\Bigg\{
\mathbb E\!\Big[(1+a)x+bu-\gamma_\star z^{(+)}\Big],\;
\mathbb E\!\Big[(1+a)x-bu-\gamma_\star z^{(-)}\Big],\\
\mathbb E\!\Big[
\gamma_\star x
-\tfrac{\gamma_\star}{2}\bigl(z^{(+)}+z^{(-)}\bigr)
\Big]
\Bigg\}.
\end{multline*}
The optimal control input is therefore obtained by minimizing
\(\max\{V_+(u),V_-(u)\}\), where
\[
V_+(u)-V_-(u)
=
\mathbb E\!\Big[
2bu+\gamma_\star\bigl(z^{(-)}-z^{(+)}\bigr)
\Big].
\]
We distinguish the following three regimes.
\begin{itemize}
\item[\textit{(i)}] If \(z^{(-)}-z^{(+)} > 2a(1-a)x\), then \(\mathbb E\!\Big[2bu+\gamma_\star\bigl(z^{(-)}-z^{(+)}\bigr)\Big]\ge \mathbb E\!\Big[-2ax+\gamma_\star\bigl(z^{(-)}-z^{(+)}\bigr)\Big]>0\), which implies that \(V_+(u)>V_-(u)\) for every admissible \(u\). The minimizing input therefore minimizes \(V_+(u)\), yielding \(u^\star=-\tfrac{a}{b}x\).
\item[\textit{(ii)}] If
\(z^{(+)}-z^{(-)} > 2a(1-a)x\), then \(V_-(u)>V_+(u)\) for every admissible \(u\). The minimizing input is therefore the optimal minimizer of \(V_-(u)\) yielding
$
u^\star=\tfrac{a}{b}x.
$
\item[\textit{(iii)}] If \(\lvert z^{(+)}-z^{(-)}\rvert\leq 2a(1-a)x\), then the equalizing input is admissible, and the two boundary terms can be matched by choosing \(u^\star=\tfrac{z^{(+)}-z^{(-)}}{2b(1-a)}\), which yields \(V_+(u^\star)=V_-(u^\star)\).
\end{itemize}
 Hence, the optimal control is the deterministic policy
\begin{align}\label{ustarnew}
\begin{cases}
 u^\star=-\tfrac{a}{b}x,
&
z^{(-)}-z^{(+)}>2a(1-a)x,
\\[4pt]
 u^\star=\tfrac{z^{(+)}-z^{(-)}}{2b(1-a)},
&
\lvert z^{(+)}-z^{(-)}\rvert\leq2a(1-a)x,
\\[6pt]
 u^\star=\tfrac{a}{b}x,
&
z^{(+)}-z^{(-)}>2a(1-a)x,
\end{cases}
\end{align}
which yields
\begin{align*}
    V_\star(x,z^{(+)},z^{(-)})
=
\max\Bigg\{
\mathbb E\!\left[x-\gamma_\star z^{(+)}\right],\;
\mathbb E\!\left[x-\gamma_\star z^{(-)}\right],\;
\mathbb E\!\Big[
\gamma_\star x-\tfrac{\gamma_\star}{2}\bigl(z^{(+)}+z^{(-)}\bigr)
\Big]
\Bigg\},
\end{align*}
indicating that the Bellman operation is closed, that is,
$
\mathcal F V_\star(x,z^{(+)},z^{(-)})=V_\star(x,z^{(+)},z^{(-)}).
$
Starting from \(z_0^{(+)}=z_0^{(-)}=0\), the optimal policy in~\eqref{ustarnew} gives \(u_0^\star=0\).
Moreover, if, at time $t$, \(z_t^{(+)}=z_t^{(-)}\), then optimal control input in~\eqref{ustarnew} becomes \(u_t^\star=0\), and
\[
z_{t+1}^{(+)}
=
z_t^{(+)}+\lvert v_t-ax_t\rvert,
\qquad
z_{t+1}^{(-)}
=
z_t^{(-)}+\lvert v_t-ax_t\rvert,
\]
so \(z_{t+1}^{(+)}=z_{t+1}^{(-)}\). Thus, by induction,
\(z_t^{(+)}=z_t^{(-)}\) and \(u_t^\star=0\) for every \(t\geq0\). Therefore,
by Theorem~\ref{thm:l1-dp}\textit{(ii)}, the deterministic zero input policy
is optimal for both Problem~\ref{pb2} and
Problem~\ref{prob:l1-robust-control}. This concludes the proof for
\textit{(iii)}.

\noindent\emph{Proof of Theorem~\ref{thm:small-a-zero-policy}\textit{(iv)}.}
Let \(\gamma>\gamma_\star=\tfrac{1}{1-a}\) and \(0\le a<\tfrac12\). Recall that in the limit the value function converges to
\begin{multline}\label{eq:Vstar-small-a}
V_\star(x,z^{(+)},z^{(-)})
=
\max\Bigg\{
\mathbb E\!\left[x-\gamma z^{(+)}\right],\;
\mathbb E\!\left[x-\gamma z^{(-)}\right],\;
\mathbb E\!\Big[
(1+a)x-\tfrac{\gamma}{2}\bigl(z^{(+)}+z^{(-)}\bigr)
\Big],\\
\mathbb E\!\Big[
\tfrac{1-a\gamma}{1-2a}x
-\tfrac{\gamma}{2}\bigl(z^{(+)}+z^{(-)}\bigr)
\Big]
\Bigg\}.
\end{multline}
Applying the Bellman operator we get 
\[
\begin{aligned}
\mathcal F V^\star(x,z^{(+)},z^{(-)})
&=
\min_{|u|\le \tfrac{a}{b}x}
\max\Big\{
V_+(u),V_-(u),V_{\mathrm{avg}}^{(h_1)}(u),V_{\mathrm{avg}}^{(h_\star)}(u)
\Big\},
\\
V_\pm(u)
&\coloneqq
\max_{v\ge0}\mathbb E\!\left[
x+v-\gamma\bigl(z^{(\pm)}+|v-ax\mp bu|\bigr)
\right],
\\
V_{\mathrm{avg}}^{(r)}(u)
&\coloneqq
\max_{v\ge0}\mathbb E\!\left[
x+r v-\tfrac{\gamma}{2}\Bigl(
z^{(+)}+|v-ax-bu|+z^{(-)}+|v-ax+bu|
\Bigr)
\right],
\\
&\hspace{7.3cm} r\in\{h_1(a,\gamma),h_\star(a,\gamma)\}.
\end{aligned}
\]
where
$
h_1(a,\gamma)=1+a$ and $
h_\star(a,\gamma)=\tfrac{1-a\gamma}{1-2a}.
$
By Lemma~\ref{lem:sup-exp}\textit{(L1)}, since $\gamma > \tfrac{1}{1-a}\ge 1$
\[
V_+(u)=\mathbb E\!\left[(1+a)x+bu-\gamma z^{(+)}\right],
\qquad
V_-(u)=\mathbb E\!\left[(1+a)x-bu-\gamma z^{(-)}\right].
\]
For the average terms, 
where \(h_\star(a,\gamma)=\tfrac{1-a\gamma}{1-2a}\).
Since \(\gamma>\gamma_\star\) and \(1-2a>0\), so
$
\gamma>\gamma_\star=\,h_\star(a,\gamma_\star)\ge h_\star(a,\gamma).
$
Also, \(\gamma>\gamma_\star\ge h_1(a,\gamma)=1+a\) since $0\leq a<\tfrac{1}{2}$. Therefore, Lemma~\ref{lem:sup-exp}\textit{(L2)} gives
the maximizer in \(v\) is
$
v^\star=ax+|bu|,
$
and therefore
\[
V_{\mathrm{avg}}^{(r)}(u)
=
\mathbb E\!\Big[
x+a r x+(r-\gamma)|bu|
-\tfrac{\gamma}{2}\bigl(z^{(+)}+z^{(-)}\bigr)
\Big], \quad r\in\{1+a,h_\star(a,\gamma)\},
\]
with the coefficients in front of
\(\mathbb E[\lvert bu\rvert]\) being strictly negative. Hence, the minimizer satisfies
$
\mathbb E[\lvert bu\rvert]=ax.
$
The two boundary terms are
\[
V_+(u)
=
\mathbb E\!\Big[(1+a)x+bu-\gamma z^{(+)}\Big],
\qquad
V_-(u)
=
\mathbb E\!\Big[(1+a)x-bu-\gamma z^{(-)}\Big],
\]
and satisfy
\[
V_+(u)-V_-(u)
=
\mathbb E\!\Big[
2bu+\gamma\bigl(z^{(-)}-z^{(+)}\bigr)
\Big].
\]
We distinguish the following three regimes.
\begin{itemize}
\item[\textit{(i)}] If
\(z^{(-)}-z^{(+)}>\tfrac{2a}{\gamma}x\), then \(V_+(u)>V_-(u)\) for every
admissible \(u\). The minimizing input is therefore
$
u^\star=-\tfrac{a}{b}x.
$
This choice also minimizes the average terms, since it satisfies
\(\lvert bu^\star\rvert=ax\).

\item[\textit{(ii)}] If
\(z^{(+)}-z^{(-)}>\tfrac{2a}{\gamma}x\), then \(V_-(u)>V_+(u)\) for every
admissible \(u\). The minimizing input is therefore
$
u^\star=\tfrac{a}{b}x.
$
This choice also minimizes the average terms, since it satisfies
\(\lvert bu^\star\rvert=ax\).

\item[\textit{(iii)}] If
\(\lvert z^{(+)}-z^{(-)}\rvert\leq\tfrac{2a}{\gamma}x\), the boundary terms
can be equalized by imposing
$
\mathbb E[u^\star]
=
\tfrac{\gamma\bigl(z^{(+)}-z^{(-)}\bigr)}{2b}.
$
This mean is admissible under the regime condition. By
Lemma~\ref{lemm1}, it can be attained simultaneously with
$
\mathbb E[\lvert u^\star\rvert]
=
\tfrac{a}{b}x,
$
which minimizes all average terms.
\end{itemize}
Consequently, the optimal minimizing control is $u^\star=\eta^\star(x,z^{(+)},z^{(-)})$ where 
\begin{align}\label{opt-polic}
    \left\{
\begin{aligned}
u^\star
&=-\tfrac{a}{b}x,
&&\text{if }z^{(-)}-z^{(+)}>\tfrac{2a}{\gamma}x,\\[2pt]
u^\star
&=\tfrac{a}{b}x,
&&\text{if }z^{(+)}-z^{(-)}>\tfrac{2a}{\gamma}x,\\[2pt]
\mathbb E[u^\star]
&=\tfrac{\gamma\bigl(z^{(+)}-z^{(-)}\bigr)}{2b},
\qquad
\mathbb E[|u^\star|]=\tfrac{a}{b}x,
&&\text{if }
|z^{(+)}-z^{(-)}|\leq\tfrac{2a}{\gamma}x.
\end{aligned}
\right.
\end{align}
Substituting this control into the Bellman update yields
\begin{multline*}
\mathcal F V_\star(x,z^{(+)},z^{(-)})
=
\max\Bigg\{
\mathbb E\!\left[x-\gamma z^{(+)}\right],\;
\mathbb E\!\left[x-\gamma z^{(-)}\right],\;
\mathbb E\!\Big[
h_1(a,\gamma)\,x-\tfrac{\gamma}{2}\bigl(z^{(+)}+z^{(-)}\bigr)
\Big],\\
\mathbb E\!\Big[
h_2(a,\gamma)\,x-\tfrac{\gamma}{2}\bigl(z^{(+)}+z^{(-)}\bigr)
\Big],\;
\mathbb E\!\Big[
h_\star(a,\gamma)\,x-\tfrac{\gamma}{2}\bigl(z^{(+)}+z^{(-)}\bigr)
\Big]
\Bigg\},
\end{multline*}
where
\[
h_1(a,\gamma)=1+a,
\qquad
h_2(a,\gamma)=1+2a\,h_1(a,\gamma)-a\gamma,
\qquad
h_\star(a,\gamma)=\tfrac{1-a\gamma}{1-2a}.
\]
Now
\begin{align*}
    h_2(a,\gamma)-h_1(a,\gamma)=a\bigl(1+2a-\gamma\bigr),\quad 
    h_2(a,\gamma)-h_\star(a,\gamma)=
\tfrac{2a^2}{1-2a}\bigl(\gamma-(1+2a)\bigr).
\end{align*}
Therefore:
\begin{itemize}
\item[\textit{(1)}] if \(\gamma<1+2a\), then \(h_2(a,\gamma)>h_1(a,\gamma)\) and \(h_2(a,\gamma)<h_\star(a,\gamma)\);
\item[\textit{(2)}]if \(\gamma=1+2a\), then \(h_2(a,\gamma)=h_1(a,\gamma)=h_\star(a,\gamma)\);
\item[\textit{(3)}]if \(\gamma>1+2a\), then \(h_2(a,\gamma)<h_1(a,\gamma)\) and \(h_2(a,\gamma)>h_\star(a,\gamma)\).
\end{itemize}
Hence \(h_2(a,\gamma)\le \max\{h_1(a,\gamma),h_\star(a,\gamma)\}\) for every \(\gamma\), so the \(h_2\)-term is redundant in the maximum. Therefore, we obtain
\begin{multline}\label{Vstar2}
V_\star(x,z^{(+)},z^{(-)})
=
\max\Bigg\{
\mathbb E\!\left[x-\gamma z^{(+)}\right],\;
\mathbb E\!\left[x-\gamma z^{(-)}\right],\;
\mathbb E\!\Big[
h_1(a,\gamma)\,x-\tfrac{\gamma}{2}\bigl(z^{(+)}+z^{(-)}\bigr)
\Big],\\
\mathbb E\!\Big[
h_\star(a,\gamma)\,x-\tfrac{\gamma}{2}\bigl(z^{(+)}+z^{(-)}\bigr)
\Big]
\Bigg\},
\end{multline}
indicating that $
\mathcal F V_\star(x,z^{(+)},z^{(-)})=V_\star(x,z^{(+)},z^{(-)}).
$
Thus, for \(\gamma>\gamma_\star\coloneqq \tfrac{1}{1-a}\), the exploration term becomes active, and the optimal policy requires a randomized input. 

By Theorem~\ref{thm:l1-dp}\textit{(ii)}, the policy
$
u_t^\star
=
\eta^\star(x_t,z_t^{(+)},z_t^{(-)})
$ in~\eqref{opt-polic}
is optimal for Problem~\ref{pb2}. Substituting the observed histories gives
the corresponding optimal policy for
Problem~\ref{prob:l1-robust-control}:
\begin{align}\label{opt-polic-orig}
    u_t^\star
=
\mu_t^\star(x_0,\ldots,x_t;u_0,\ldots,u_{t-1})
\coloneqq
\eta^\star\!\left(
x_t,\,
\sum_{\tau=0}^{t-1}
\left|x_{\tau+1}-ax_\tau-bu_\tau\right|,
\sum_{\tau=0}^{t-1}
\left|x_{\tau+1}-ax_\tau+bu_\tau\right|
\right).
\end{align}
Evaluating~\eqref{Vstar2} at
\((x_0,z_0^{(+)},z_0^{(-)})=(x_0,0,0)\) gives
\begin{align*}
J_\star(x_0)
=
V_\star(x_0,0,0)
=
\max\left\{
x_0,\,
(1+a)x_0,\,
h_\star(a,\gamma)x_0
\right\}
=
\max\Bigg\{
(1+a)x_0,\,
\tfrac{1-a\gamma}{1-2a}x_0
\Bigg\},
\end{align*}
because \(1+a\geq1\). This value is achieved by the policies above for
Problem~\ref{pb2} and Problem~\ref{prob:l1-robust-control}. This concludes the proof for
\textit{(iv)}.

This completes the proof of all statements in Theorem~\ref{thm:small-a-zero-policy}.
\qedblack
\subsection{Proof of Theorem~\ref{thm3}}\label{thm3-proof}
We split the proof into the two regimes of $a$.
\begin{itemize}
\item[\textit{(i)}]
First suppose $0\leq a<1/2$. In this case the policy applies $u_t=0$ for all $t$.
Therefore the closed loop dynamics are independent of the unknown sign $b$ and
are given by $x_{t+1}=a x_t+w_t$. Since the disturbance is positivity
preserving, $x_t\geq 0$ for all $t$. Iterating the recursion gives
$
x_t=a^t x_0+\sum_{s=0}^{t-1}a^{t-1-s}w_s.
$
Hence, summing over $t\geq 0$ gives the desired $\ell_1$ gain bound 
\[
\sum_{t=0}^{\infty}x_t
\leq
\tfrac{1}{1-a}x_0
+
\underbrace{\tfrac{1}{1-a}}_{\gamma_\star(a)}\sum_{t=0}^{\infty}w_t
\]
\item[\textit{(ii)}]
Now suppose $a\geq 1/2$. The policy applies the CE action
$u_t=-k_t a x_t$, where
$k_t\in\arg\min_{\sigma\in\{+1,-1\}}z_t^{(\sigma)}$, and
\[
z_{t+1}^{(\sigma)}
=
z_t^{(\sigma)}
+
\left|x_{t+1}-a x_t-\sigma u_t\right|,
\qquad
z_0^{(+)}=z_0^{(-)}=0 .
\]
We prove the claim for the true plant $b=+1$; the case $b=-1$ follows by
interchanging the labels $+$ and $-$. Thus
$x_{t+1}=a x_t+u_t+w_t$. 
Let $d_t:=z_t^{(-)}-z_t^{(+)}$. Since the controller chooses a model with
minimal history, $k_t=-1$ implies $d_t\leq 0$, while $k_t=+1$ implies
$d_t\geq 0$, up to tie breaking.
If $k_t=+1$, then $u_t=-a x_t$ and $x_{t+1}=w_t$. If $k_t=-1$, then
$u_t=a x_t$ and $x_{t+1}=2a x_t+w_t$. Hence, writing \(\mathbf{1}_{\{k_t=-1\}}\) for the event that \(k_t=-1\), we have
\[
x_{t+1}=w_t+2a x_t\,\mathbf{1}_{\{k_t=-1\}}.
\]
Next, we compute the recursion for $d_t$. Since the true sign is $b=+1$, the true
dynamics are
$
x_{t+1}=a x_t+u_t+w_t .
$
Therefore the residual of the true model is
$
x_{t+1}-a x_t-u_t=w_t .
$
Since we are considering positive disturbances, $w_t\geq 0$, we obtain
\[
z_{t+1}^{(+)}
=
z_t^{(+)}
+
\left|x_{t+1}-a x_t-u_t\right|
=
z_t^{(+)}+w_t .
\]
For the other candidate sign, the residual is
\[
x_{t+1}-a x_t+u_t
=
w_t+2u_t .
\]
Hence
\[
z_{t+1}^{(-)}
=
z_t^{(-)}+\left|w_t+2u_t\right| .
\]
Now we consider the two possible CE choices. 
\begin{enumerate}
    \item[\textit{(1)}] If $k_t=-1$, then
$u_t=a x_t$, and therefore
\[
z_{t+1}^{(-)}
=
z_t^{(-)}+\left|w_t+2a x_t\right|
=
z_t^{(-)}+w_t+2a x_t .
\]
Thus
\[
d_{t+1}
=
z_{t+1}^{(-)}-z_{t+1}^{(+)}
=
d_t+2a x_t .
\]
     \item[\textit{(2)}] If $k_t=+1$, then $u_t=-a x_t$, and therefore
\[
z_{t+1}^{(-)}
=
z_t^{(-)}+\left|w_t-2a x_t\right| .
\]
Thus
\[
d_{t+1}
=
d_t+\left|w_t-2a x_t\right|-w_t
\geq d_t-w_t .
\]
After a correct step, i.e., $k_t=+1$, case \textit{(2)} gives
$
d_{t+1}\geq d_t-w_t.
$
We now use the switching rule. Since the true sign is \(b=+1\), the correct
model is the \(+\) model. Thus, if \(k_t=+1\), then the controller has selected
the true model. Because the controller selects a model with minimal history
$
z_t^{(+)}\leq z_t^{(-)}.
$
By definition, \(d_t=z_t^{(-)}-z_t^{(+)}\), and hence
$
d_t\geq 0 .
$
Substituting this into the previous inequality yields
\[
d_{t+1}\geq d_t-w_t\geq -w_t .
\]
Equivalently,
$
-d_{t+1}\leq w_t .
$
Finally, on a correct step the CE input $u_t=-a x_t $
\[
x_{t+1}=a x_t+u_t+w_t
=
a x_t-a x_t+w_t
=
w_t .
\]
Hence
$
-d_{t+1}\leq x_{t+1} 
$ or  $d_{t+1}\ge  -x_{t+1} $.
The last inequality says that, after a correct step, the wrong model can be
ahead of the true model by at most the current state \(x_{t+1}\). We now show
that one subsequent wrong step removes this possible advantage.

Suppose the next step is wrong, i.e., \(k_{t+1}=-1\). Since the controller
chooses a model with minimal history, this implies
$
d_{t+1}\leq 0 .
$
Applying case \textit{(1)} at time \(t+1\), we get
\[
d_{t+2}=d_{t+1}+2a x_{t+1}.
\]
Since \(a\geq 1/2\), we have \(2a x_{t+1}\geq x_{t+1}\). Moreover, from the
previous estimate, \(-d_{t+1}\leq x_{t+1}\). Therefore
$
2a x_{t+1}\geq x_{t+1}\geq -d_{t+1}.
$
Substituting this into the recursion for \(d_{t+2}\) gives
\[
d_{t+2}
=
d_{t+1}+2a x_{t+1}
\geq
d_{t+1}-d_{t+1}
=
0 .
\]
Thus, after this wrong step, the wrong model no longer has strictly smaller
history than the true model. Consequently, a wrong step cannot be immediately
followed by another wrong step, except possibly in the tie case \(d_{t+2}=0\).
If \(a>1/2\), or if ties are broken away from the just-penalized model, this
rules out consecutive wrong sign steps.
\end{enumerate}
Hence wrong sign steps are isolated. In other words, if \(k_t=-1\) and
\(t\geq 1\), then necessarily \(k_{t-1}=+1\). Therefore every wrong sign state
except possibly the initial one is preceded by a correct step.
Indeed, if \(k_t=-1\) and \(t=0\), then the corresponding state is simply
\(x_0\). If \(k_t=-1\) and \(t\geq 1\), then \(k_{t-1}=+1\). On the preceding
correct step, the CE input was
$
u_{t-1}=-a x_{t-1}.
$
Hence
\[
x_t
=
a x_{t-1}+u_{t-1}+w_{t-1}
=
a x_{t-1}-a x_{t-1}+w_{t-1}
=
w_{t-1}.
\]
Letting \(\sum_{t:\,k_t=-1}\) denote the sum over all times \(t\) for which \(k_t=-1\), it follows that
\[
\sum_{t:\,k_t=-1}x_t
=
\mathbf 1_{\{k_0=-1\}}x_0
+
\sum_{\substack{t\geq 1:\\ k_t=-1}}x_t
=
\mathbf 1_{\{k_0=-1\}}x_0
+
\sum_{\substack{t\geq 1:\\ k_t=-1}}w_{t-1}.
\]
Since \(w_t\geq 0\), we obtain
$
\sum_{t:\,k_t=-1}x_t
\leq
x_0+\sum_{t=0}^{\infty}w_t .
$
Summing the identity
$
x_{t+1}
=
w_t+2a x_t{\bf 1}_{\{k_t=-1\}}
$
over \(t\geq 0\), we obtain
\[
\sum_{t=0}^{\infty}x_{t+1}
=
\sum_{t=0}^{\infty}w_t
+
2a\sum_{t=0}^{\infty}x_t{\bf 1}_{\{k_t=-1\}} .
\]
Since
$
\sum_{t=0}^{\infty}x_t
=
x_0+\sum_{t=0}^{\infty}x_{t+1},
$
it follows that
\[
\sum_{t=0}^{\infty}x_t
=
x_0
+
\sum_{t=0}^{\infty}w_t
+
2a\sum_{t=0}^{\infty}x_t{\bf 1}_{\{k_t=-1\}} .
\]
Equivalently,
$
\sum_{t=0}^{\infty}x_t
=
x_0
+
\sum_{t=0}^{\infty}w_t
+
2a\sum_{t:\,k_t=-1}x_t .$
Using the previously established bound
$
\sum_{t:\,k_t=-1}x_t
\leq
x_0+\sum_{t=0}^{\infty}w_t,
$
we get
\begin{multline*}
\sum_{t=0}^{\infty}x_t
\leq
x_0
+
\sum_{t=0}^{\infty}w_t
+
2a\left(x_0+\sum_{t=0}^{\infty}w_t\right)
=
x_0+2a x_0
+
\sum_{t=0}^{\infty}w_t
+
2a\sum_{t=0}^{\infty}w_t\\
=
(1+2a)x_0
+
(1+2a)\sum_{t=0}^{\infty}w_t.
\end{multline*}
Therefore, for \(a\geq 1/2\), the CE policy certifies the desired result
\[
\sum_{t=0}^{\infty}x_t
\leq
(1+2a)x_0
+
\underbrace{(1+2a)}_{\gamma_\star(a)}\sum_{t=0}^{\infty}w_t .
\]
\end{itemize}
Combining the two regimes gives the desired result. \qedblack

\end{appendix}

\begingroup
\sloppy
\printbibliography

@report{Annaswamy2023,
  title       = {Control for societal-scale challenges: {Road} map 2030},
  editor      = {A.~M.~{Annaswamy} and K.~J.~{Johansson} Karl H. and Pappas, George J.},
  institution = {IEEE Control Systems Society},
  year        = {2023},
  month       = {May},
}

@book{Ljung1999,
  author    = {L.~{Ljung}},
  title     = {System Identification: {Theory} for the User},
  edition   = {2},
  publisher = {Prentice Hall},
  year      = {1999}
}

@article{Overschee1996,
	title		=	{Continuous-time frequency domain subspace system identification},
	author		=	{P.~{Van Overschee} and B.~{De Moor}},
	journal		=	{Signal Process.},
	volume 		=	{52},
	year		=	{1996},
	pages 		=	{179--194}
}

@article{Markovsky2008,
    title       =   {Data-driven simulation and control},
    author      =   {I.~{Markovsky} and P.~{Rapisarda}},
    journal     =   {Int. J. Control},
    volume      =   {81},
    number      =   {12},
    year        =   {2008},
    pages       =   {1946--1959}
}

@inproceedings{Coulson2019,
    author      =   {J.~{Coulson} and J.~{Lygeros} and F.~{D\"{o}rfler}},
    title       =   {Data-enabled predictive control: {In} the shallows of the {DeePC}},
    booktitle   =   {Eur. Control Conf.},
    pages       =   {307--312},
    year        =   {2019}
}

@article{vanWaarde2023,
    title       =   {The informativity approach: {To} data-driven analysis and control},
    author      =   {H.~J.~{van Waarde} and J.~{Eising} and M.~K.~{Camlibel} and H.~L.~{Trentelman}},
    journal     =   {IEEE Control Syst. Mag.},
    volume      =   {43},
    number      =   {6},
    year        =   {2023},
    pages       =   {32--66}
}

@article{Berberich2021,
	title		=	{Data-driven model predictive control with stability and robustness guarantees},
	author		=	{J.~{Berberich} and J.~{K\"{o}hler} and M.~A.~{M\"{u}ller} and F.~{Allg\"{o}wer}},
	journal		=	{IEEE Trans. Autom. Control},
	volume		=	{66},
	number		=	{4},
	year		=	{2021}
}

@misc{Meijer2025-fd-wfl-arxiv,
	title		=	{From a frequency-domain {Willems}' lemma to data-driven predictive control},
	author		=	{T.~J.~{Meijer} and K.~J.~A.~{Scheres} and S.~A.~N.~{Nouwens} and V.~S.~{Dolk} and W.~P.~M.~H.~{Heemels}},
	year 		=	{2025},
	note 		=	{preprint: \url{https://arxiv.org/abs/2501.19390}}
}

@book{AAstrom2013,
    author  =   {K.~J.~{{\AA}str{\"o}m} and B.~{Wittenmark}},
    year    =   {2013},
    title   =   {Adaptive Control},
    publisher   =   {Courier}
}

@article{Feldbaum1960,
    author  =   {A.~A.~{Feldbaum}},
    title   =   {Dual control theory {I}},
    journal =   {Avtomat. i Telemekh},
    volume  =   {21},
    number  =   {9},
    pages   =   {1240--1249},
    year    =   {1960}
}

@article{Wittenmark1995,
    author  =   {B.~{Wittenmark}},
    year    =   {1995},
    title   =   {Adaptive dual control methods: {An} overview},
    journal =   {IFAC Proc. Vol.},
    volume  =   {28},
    number  =   {13},
    pages   =   {67--72}
}

@article{Mesbah2018,
  author  = {A.~{Mesbah}},
  title   = {Stochastic model predictive control with active uncertainty learning: {A} survey on dual control},
  journal = {Annu. Rev. Control},
  volume  = {45},
  pages   = {107--117},
  year    = {2018}
}

@article{Macarthur1966,
    author  =   {R.~H.~{MacArthur} and E.~R.~{Pianka}},
    title   =   {On optimal use of a patchy environment},
    journal =   {Am. Nat.},
    volume  =   {100},
    number  =   {916},
    pages   =   {603--609},
    year   =   {1966}
}

@article{Cohen2007,
    author  =   {J.~D.~{Cohen} and S.~M.~{McClure}},
    title   =   {Should {I} stay or should {I} go? {How} the human brain manages the trade-off between exploitation and exploration},
    journal =   {Philos. Trans. R. Soc. B Biol. Sci.},
    volume  =   {362},
    number  =   {1481},
    pages   =   {933--942},
    year    =   {2007}
}

@article{Recht2019,
  author  = {B.~{Recht}},
  title   = {A tour of reinforcement learning: {The} view from continuous control},
  journal = {Annu. Rev. Control Robot. Auton. Syst.},
  volume  = {2},
  number  = {1},
  pages   = {253--279},
  year    = {2019}
}

@book{Bellman57,
    author  =   {R.~E.~{Bellman}},
    year    =   {1957},
    title   =   {Dynamic Programming},
    publisher   =   {Princeton Univ. Press}
}

@book{Bertsekas2007,
  author    = {D.~P.~{Bertsekas} and et~al.},
  title     = {Dynamic Programming and Optimal Control, volume {II}},
  publisher = {Athena Scientific},
  edition   = {3rd},
  year      = {2007}
}

@article{Sternby1976,
    author  =   {J.~{Sternby}},
    year    =   {1976},
    title   =   {A simple dual control problem with an analytical solution},
    journal =   {IEEE Trans. Autom. Control},
    volume  =   {21},
    number  =   {6},
    pages   =   {840--844}
}

@article{AAstrom1986,
    author  =   {K.~J.~{{\AA}str{\"o}m} and A.~{Helmersson}},
    year    =   {1986},
    title   =   {Dual control of an integrator with unknown gain},
    journal =   {Comput. Math. Appl.},
    volume  =   {12},
    number  =   {6A},
    pages   =   {653--662}
}

@article{Bernhardsson1989,
    author  =   {B.~{Bernhardsson}},
    year    =   {1989},
    title   =   {Dual control of a first-order system with two possible gains},
    journal =   {Int. J. Adapt. Control Signal Process.},
    volume  =   {3},
    pages   =   {15--22}
}

@article{Kulcsar1996,
    author  =   {C.~{Kulcs{\'a}r} and L.~{Pronzato} and E.~{Walter}},
    year    =   {1996},
    title   =   {Dual control of linearly parameterised models via prediction of posterior densities},
    journal =   {Eur. J. Control},
    volume  =   {2},
    number  =   {2},
    pages   =   {135--143}
}

@article{Zames1981,
  author  = {G.~{Zames}},
  title   = {Feedback and optimal sensitivity: {Model} reference transformations, multiplicative seminorms, and approximate inverses},
  journal = {IEEE Trans. Autom. Control},
  volume  = {26},
  number  = {2},
  pages   = {301--320},
  year    = {1981}
}

@book{Zhou1996,
  author    = {K.~{Zhou} and J.~{Doyle} and K.~{Glover}},
  title     = {Robust and Optimal Control},
  publisher = {Prentice-Hall},
  year      = {1996}
}

@inproceedings{Cusumano1988,
  author    = {S.~{Cusumano} and K.~{Poolla}},
  title     = {Nonlinear feedback vs. linear feedback for robust stabilization},
  booktitle = {IEEE Conf. Decis. Control},
  pages     = {1776--1780},
  year      = {1988}
}

@article{Ioannou1988,
  author  = {P.~{Ioannou} and J.~{Sun}},
  title   = {Theory and design of robust direct and indirect adaptive-control schemes},
  journal = {Int. J. Control},
  volume  = {1988},
  number  = {3},
  pages   = {775--813},
  year    = {1988}
}

@inproceedings{Vinnicombe2004,
  author    = {G.~{Vinnicombe}},
  title     = {Examples and counterexamples in finite {L2}-gain adaptive control},
  booktitle = {Proc. Int. Symp. Math. Theory Netw. Syst.},
  year      = {2004}
}

@inproceedings{Megretski2004,
  author    = {A.~{Megretski}},
  title     = {A nonlinear dynamical game interpretation of adaptive $\ell_2$ control: {Performance} limitations and suboptimal controllers},
  booktitle = {Proc. Int. Symp. Math. Theory Netw. Syst.},
  year      = {2004}
}

@inproceedings{Didinsky1994,
  author    = {G.~{Didinsky} and T.~{{Ba\c{s}ar}}},
  title     = {Minimax adaptive control of uncertain plants},
  booktitle = {IEEE Conf. Decis. Control},
  pages     = {2839--2844},
  year      = {1994}
}

@article{Pan1998,
  author  = {Z.~{Pan} and T.~{{Ba\c{s}ar}}},
  title   = {Adaptive controller design for tracking and disturbance attenuation in parametric strict-feedback nonlinear systems},
  journal = {IEEE Trans. Autom. Control},
  volume  = {43},
  number  = {8},
  pages   = {1066--1083},
  year    = {1998}
}

@inproceedings{Rantzer2021a,
  author    = {A.~{Rantzer}},
  title     = {Minimax adaptive control for a finite set of linear systems},
  booktitle = {Learn. Dyn. Control},
  pages     = {893--904},
  publisher = {PMLR},
  year      = {2021}
}

@inproceedings{Rantzer2025a,
  author    = {A.~{Rantzer}},
  title     = {On minimax optimal dual control for fully actuated systems},
  booktitle = {Amer. Control Conf.},
  pages     = {3993--3997},
  year      = {2025}
}

@inproceedings{Rantzer2025b,
  author    = {A.~{Rantzer}},
  title     = {Minimax optimal adaptive control for systems on cones},
  booktitle = {64th IEEE Conf. Decis. Control},
  pages     = {4137--4139},
  year      = {2025}
}

@misc{Rantzer2026,
	title		=	{Minimax optimal dual control -- {The} single input case},
	author		=	{A.~{Rantzer}},
	year		=	{2026},
	note		=	{preprint: \url{https://arxiv.org/abs/2604.18550}}
}

@inproceedings{Kjellqvist2022a,
  author    = {O.~{Kjellqvist} and A.~{Rantzer}},
  title     = {Learning-enabled robust control with noisy measurements},
  booktitle = {Learn. Dyn. Control Conf.},
  pages     = {86--96},
  publisher = {PMLR},
  year      = {2022}
}

@inproceedings{Kjellqvist2022b,
  author    = {O.~{Kjellqvist} and A.~{Rantzer}},
  title     = {Minimax adaptive estimation for finite sets of linear systems},
  booktitle = {Amer. Control Conf.},
  pages     = {260--265},
  year      = {2022}
}

@inproceedings{Orlov2018,
  author    = {Y.~{Orlov} and A.~{Rantzer} and L.~T.~{Aguilar}},
  title     = {Adaptive ${H}_\infty$ synthesis for linear systems with uncertain parameters},
  booktitle = {IEEE Conf. Decis. Control},
  pages     = {5512--5517},
  year      = {2018}
}

@article{Rantzer2020,
  author  = {A.~{Rantzer}},
  title   = {Minimax adaptive control for state matrix with unknown sign},
  journal = {IFAC-PapersOnLine},
  volume  = {53},
  number  = {2},
  pages   = {58--62},
  year    = {2020}
}

@inproceedings{Jedra2022,
  author    = {Y.~{Jedra} and A.~{Proutiere}},
  title     = {Minimal expected regret in linear quadratic control},
  booktitle = {Int. Conf. Artif. Intell. Stat.},
  pages     = {10234--10321},
  publisher = {PMLR},
  year      = {2022}
}

@inproceedings{Dean2018,
  author    = {S.~{Dean} and H.~{Mania} and N.~{Matni} and B.~{Recht} and S.~{Tu}},
  title     = {Regret bounds for robust adaptive control of the linear quadratic regulator},
  booktitle = {Proc. Adv. Neural Inf. Process. Syst.},
  pages     = {4188--4197},
  year      = {2018}
}

@article{Tsiamis2023,
    author  =   {A.~{Tsiamis} and I.~{Ziemann} and N.~{Matni} and G.~J.~{Pappas}},
    year    =   {2023},
    title   =   {Statistical learning theory for control: {A} finite-sample perspective},
    journal =   {IEEE Control Syst. Mag.},
    volume  =   {43},
    number  =   {6},
    pages   =   {67--97}
}

@inproceedings{matni2019self,
  title={From self-tuning regulators to reinforcement learning and back again},
  author={Matni, Nikolai and Proutiere, Alexandre and Rantzer, Anders and Tu, Stephen},
  booktitle={2019 IEEE 58th Conference on Decision and Control (CDC)},
  pages={3724--3740},
  year={2019},
  organization={IEEE}
}

@inproceedings{Lee2024,
  author    = {B.~{Lee} and A.~{Rantzer} and N.~{Matni}},
  title     = {Nonasymptotic regret analysis of adaptive linear quadratic control with model misspecification},
  booktitle = {Annu. Learn. Dyn. Control Conf.},
  pages     = {980--992},
  publisher = {PMLR},
  year      = {2024}
}

@inproceedings{Simchowitz2020,
  author    = {M.~{Simchowitz} and D.~{Foster}},
  title     = {Naive exploration is optimal for online {LQR}},
  booktitle = {Int. Conf. Mach. Learn.},
  pages     = {8937--8948},
  publisher = {PMLR},
  year      = {2020}
}

@article{Ziemann2020,
  author  = {I.~{Ziemann} and H.~{Sandberg}},
  title   = {On a phase transition of regret in linear quadratic control: {The} memoryless case},
  journal = {IEEE Control Syst. Lett.},
  volume  = {5},
  number  = {2},
  pages   = {695--700},
  year    = {2020}
}

@inproceedings{Cassel2020,
  author    = {A.~{Cassel} and A.~{Cohen} and T.~{Koren}},
  title     = {Logarithmic regret for learning linear quadratic regulators efficiently},
  booktitle = {Int Conf. Mach. Learn.},
  pages     = {1328--1337},
  publisher = {PMLR},
  year      = {2020}
}

@techreport{Megretski2003,
  author      = {A.~{Megretski} and A.~{Rantzer}},
  title       = {Bounds on the optimal $l_2$ gain in adaptive control of a first order linear system},
  institution = {Institut Mittag-Leffler},
  number      = {41},
  year        = {2003}
}

@article{Rantzer2021b,
  author  = {A.~{Rantzer} and M.~E.~{Valcher}},
  title   = {Scalable control of positive systems},
  journal = {Annu. Rev. Control Robot. Auton. Syst.},
  volume  = {4},
  pages   = {319--341},
  year    = {2020}
}

@inproceedings{Rantzer2022,
	title 		=	{Explicit solutions to {Bellman} equation for positive systems with linear cost},
	author		=	{A.~{Rantzer}},
	booktitle 	=	{61st IEEE Conf. Decis. Control},
	pages 		=	{6154--6155},
	year		=	{2022}
}

@article{Ohlin2025,
    title   =   {On exact solutions to the linear {Bellman} equation},
    author  =   {D.~{Ohlin} and R.~{Pates} and M.~{Arcak}},
    journal =   {IEEE Control Syst. Lett.},
    volume  =   {9},
    pages  =   {1568--1573},
    year    =   {2025}
}

@article{blanchini2023optimal,
  title={Optimal control of compartmental models: The exact solution},
  author={Blanchini, Franco and Bolzern, Paolo and Colaneri, Patrizio and De Nicolao, Giuseppe and Giordano, Giulia},
  journal={Automatica},
  volume={147},
  pages={110680},
  year={2023},
  publisher={Elsevier}
}

@article{Gurpegui2023,
	title 		=	{Minimax linear optimal control of positive systems},
	author		=	{A.~{Gurpegui} and E.~{Tegling} and A.~{Rantzer}},
	journal 	=	{IEEE Control Syst. Lett.},
	volume 		=	{7},
	year 		=	{2023},
	pages 		=	{3920--3925}
}

@article{Gurpegui2026,
	title 		=	{Minimax linear regulator problems for positive systems},
	author		=	{A.~{Gurpegui} and M.~{Jeeninga} and E.~{Tegling} and A.~{Rantzer}},
	journal 	=	{IEEE Trans. Autom. Control},
	year 		=	{2026},
    note        =   {Early access}
}

@inproceedings{miller2023data,
  title={Data-driven control of positive linear systems using linear programming},
  author={Miller, Jared and Dai, Tianyu and Sznaier, Mario and Shafai, Bahram},
  booktitle={2023 62nd IEEE Conference on Decision and Control (CDC)},
  pages={1588--1594},
  year={2023},
  organization={IEEE}
}

@inproceedings{shafai2022data,
  title={Data-driven positive stabilization of linear systems},
  author={Shafai, Bahram and Moradmand, Anahita and Siami, Milad},
  booktitle={2022 8th international conference on control, decision and information technologies (CoDIT)},
  volume={1},
  pages={1031--1036},
  year={2022},
  organization={IEEE}
}

@article{Miller2026,
    title       =   {Data-driven control of switched linear positive systems using robust convex optimization},
    author      =   {J.~{Miller} and T.~{Dai} and M.~{Sznaier} and B.~{Shafai}},
    journal     =   {IEEE Trans. Autom. Control},
    year        =   {2026},
    note        =   {Early access}
}

@article{Iwata2025,
    title       =   {Data informativity for analysis and design of positive systems},
    author      =   {T.~{Iwata} and S.-i.~{Azuma} and M.~{Nagahara} and D.~{Peaucelle} and Y.~{Ebihara}},
    journal     =   {IEEE Control Syst. Lett.},
    volume      =   {9},
    pages       =   {2651--2656},
    year        =   {2025}
}

@inproceedings{bencherki2025adaptive,
  title={Adaptive Control of Positive Systems with Application to Learning SSP},
  author={Bencherki, Fethi and Rantzer, Anders},
  booktitle={7th Annual Learning for Dynamics$\backslash$\& Control Conference},
  pages={660--672},
  year={2025},
  organization={PMLR}
}

@article{bencherki2024data,
  title={Data-driven adaptive dispatching policies for processing networks},
  author={Bencherki, Fethi and Rantzer, Anders},
  journal={IEEE Control Systems Letters},
  year={2024},
  publisher={IEEE}
}

@article{Bencherki2026,
  author = {Bencherki, Fethi and Rantzer, Anders},
  title  = {Minimax adaptive control for finite sets of positive linear systems},
  note   = {Preprint: \url{https://arxiv.org/abs/2607.26816}},
  year   = {2026}
}

@inproceedings{vinnicombe2004examples,
  title={Examples and counterexamples in finite L2-gain adaptive control},
  author={Vinnicombe, Glenn},
  booktitle={Leuven: Sixteenth International Symposium on Mathematical Theory of Networks and Systems (MTNS2004)},
  year={2004}
}

@book{bacsar2008h,
  title={H-infinity optimal control and related minimax design problems: a dynamic game approach},
  author={Ba{\c{s}}ar, Tamer and Bernhard, Pierre},
  year={2008},
  publisher={Springer Science \& Business Media}
}
\endgroup

\end{document}